\documentclass[final]{siamltex}

\usepackage{times}                
\usepackage{bm,amsbsy}
\usepackage{amsmath, amssymb,verbatim,graphicx, amsfonts}
\usepackage{mathrsfs}
\usepackage{verbatim}
\usepackage{multirow}
\usepackage{subfigure}
\usepackage{epsfig}
\usepackage{float}

\usepackage{pdfpages,setspace}
\usepackage[left=1in,top=1in,right=1in,bottom=1.0in]{geometry}
\usepackage{xcolor}
\allowdisplaybreaks
\expandafter\def\expandafter\normalsize\expandafter{%
    \normalsize
    \setlength\abovedisplayskip{5pt}
    \setlength\belowdisplayskip{5pt}
    \setlength\abovedisplayshortskip{0pt}
    \setlength\belowdisplayshortskip{5pt}
}

\newcommand{\mb}[1]{{\mathbf{#1}}}
\newcommand{\ds}{\displaystyle}

\newcommand{\mc}[1]{{\mathcal{#1}}}

\newcommand{\dd}[2]{{\partial_{#1} #2}}

\newcommand{\WH}[1]{\mathbf{\widehat{\text{$#1$}}}}
\newtheorem{Remark}{Remark}[section]

\newcommand{\WHO}{{\WH{\omega}}}
\newcommand{\WHG}{{\WH{\gamma}}}

\DeclareMathOperator{\arcsinh}{arcsinh}

\title{Dispersion Analysis of Finite Difference and Discontinuous Galerkin Schemes for Maxwell's Equations in Linear Lorentz Media}

\author{Yan Jiang
	\thanks{School of Mathematical Sciences, University of Science and Technology of China, 
	Hefei, Anhui 230026, People’s Republic of China
		{\tt jiangy@ustc.edu.cn}. }
	\and
        Puttha Sakkaplangkul
	\thanks{Department of Mathematics, Faculty of Science, King Mongkut's Institute of Technology Ladkrabang, Ladkrabang, Bangkok 10520, Thailand.
		{\tt puttha.sa@kmitl.ac.th}. }
	\and
        Vrushali A. Bokil
	\thanks{Department of Mathematics, Oregon State University,
		Corvallis, OR 97331 U.S.A.
		{\tt bokilv@math.oregonstate.edu}. Research is supported by NSF grant  DMS-1720116}%
	\and
	Yingda Cheng
        \thanks{Department of Mathematics, Department of  Computational Mathematics, Science and Engineering, Michigan State University,
               East Lansing, MI 48824 U.S.A.
          {\tt ycheng@msu.edu}. Research is supported by NSF grants  DMS-1453661, DMS-1720023 and the Simons Foundation under award number 558704.}
	\and
	Fengyan Li
	\thanks{Department of Mathematical Sciences, Rensselaer Polytechnic Institute, Troy, NY 12180 U.S.A.
		{\tt  lif@rpi.edu}. Research is supported by NSF grant  DMS-1719942.}
	}

\date{\today}

\begin{document}

\maketitle

\begin{abstract}
  In this paper, we consider Maxwell's equations in linear dispersive media described by a single-pole Lorentz model for electronic polarization. We study two classes of commonly used spatial discretizations: finite difference methods (FD) with arbitrary even order accuracy in space and high spatial order discontinuous Galerkin (DG) finite element methods. Both types of spatial discretizations are coupled with second order semi-implicit leap-frog and implicit trapezoidal temporal schemes. By performing detailed dispersion analysis for the semi-discrete and fully discrete schemes, we obtain rigorous quantification of the dispersion error for Lorentz dispersive dielectrics. In particular, comparisons of dispersion error can be made taking into account the model parameters, and mesh sizes in the design of the two types of schemes. This work is a continuation of our previous research on energy-stable numerical schemes for nonlinear dispersive optical media \cite{bokil2017energy,bokil2018high}. The results for the numerical dispersion analysis of the reduced linear model, considered in the present paper, can guide us in the optimal choice of discretization parameters for the more complicated and nonlinear models. The numerical dispersion analysis of the fully discrete FD and DG schemes, for the dispersive Maxwell model considered in this paper, clearly indicate the dependence of the numerical dispersion errors on spatial and temporal discretizations, their order of accuracy, mesh discretization parameters and model parameters. The results obtained here cannot be arrived at by considering discretizations of Maxwell's equations in free space. In particular, our results contrast the advantages and disadvantages of using high order FD or DG schemes and leap-frog or trapezoidal time integrators over different frequency ranges using a variety of measures of numerical dispersion errors. Finally, we highlight the limitations of the second order accurate temporal discretizations considered.
  
\end{abstract}


\section{Introduction}

%
%
%
%
%
%
%
%

The electromagnetic (EM) field inside a material is governed by the macroscopic Maxwell's equations along with constitutive laws that account for the response of the material to the electromagnetic field. In this work, we consider a linear dispersive material in which the delayed response to the EM field is modeled as a damped vibrating system for the polarization accounting for the average dipole moment per unit volume over the atomic structure of the material. The corresponding mathematical equations are called the Lorentz model for electronic polarization. Such dielectric materials have actual physical dispersion. The complex-valued electric permittivity of such a dispersive material is frequency dependent and includes physical dissipation, or attenuation. It is well known that numerical discretizations of (systems of) partial differential equations (PDEs) will have numerical errors. These errors include dissipation, the dampening of some frequency modes, and dispersion, the frequency dependence of the phase velocity of numerical wave modes \cite{trefethen1982group}. To preserve the correct physics, it is important that the dispersion and dissipation effects are accurately captured by numerical schemes, particularly for long time simulations. Thus, an understanding of how numerical discretizations affect the dispersion relations of PDEs is important in constructing good numerical schemes that correctly predict wave propagation over long distances. 
  When the PDEs have physical dispersion modeling a retarded response of the material to the imposed electromagnetic field, the corresponding numerical discretizations will support numerical dispersion errors that have a complicated dependence on mesh step sizes, spatial and temporal accuracy and model parameters. 
 
  In this paper, we perform dispersion analysis of high spatial order discontinuous Galerkin (DG) and  a class of high order finite difference (FD) schemes, both coupled with second order implicit trapezoidal or semi-implicit leap-frog temporal discretizations for Maxwell's equations in linear Lorentz media.  The fully discrete time domain (TD) methods are the leap-frog DGTD or FDTD methods and the trapezoidal DGTD or FDTD methods. This paper is a continuation of our recent efforts on energy stable numerical schemes for nonlinear dispersive optical media. In \cite{bokil2017energy,bokil2018high}, we developed fully discrete energy stable DGTD and FDTD methods, respectively, for Maxwell's equations with linear Lorentz and nonlinear Kerr and Raman responses via the {\it auxiliary differential equation} (ADE) approach. These schemes include second order modified leap-frog or trapezoidal temporal schemes combined with high order DG or FD methods for the spatial discretization. In the ADE approach, ordinary differential equations (ODEs) for the evolution of the electric polarization are appended to Maxwell's equations. The two spatial discretizations that were used, the DG method and the FD method are very popular methods for electromagnetic simulations in the literature.  The DG methods, which are a class of finite element methods using discontinuous polynomial spaces, have grown to be broadly adopted for EM simulations in the past two decades. They have been developed and analyzed for time dependent linear models, including Maxwell's equations in free space (e.g.,
\cite{chung2013convergence, cockburn2004locally, hesthaven2002nodal}), and in dispersive media (e.g., \cite{gedney2012discontinuous, huang2011interior, lanteri2013convergence,lu2004discontinuous}).  The Yee scheme \cite{yee1966numerical} is a leap-frog FDTD method that was initially developed for Maxwell's equations in linear dielectrics, and is one of the gold standards for numerical simulation of EM wave propagation in the time domain. The Yee scheme has been extended to linear dispersive media \cite{joseph1991direct, KF_Deb, KF_Lor} (see the books \cite{taflove2005computational, taflove2013advances} and references therein), and then to nonlinear dispersive media \cite{ziolkowski1994nonlinear, goorjian1992computational, joseph1991direct, hile1996numerical, sorensen2005kink}. Additional references for Yee and other FDTD methods for EM wave propagation in linear and nonlinear Lorentz dispersion can be found in \cite{joseph1997fdtd, joseph1994spatial, bourgeade2005numerical, greene2006general, ramadan2015systematic} for the 1D case, and in \cite{fujii2004high, joseph1993direct, ziolkowski1994nonlinear} for 2D and 3D cases. 

In our recent work \cite{bokil2017energy,bokil2018high}, we proved energy stability of fully discrete new FDTD and DGTD schemes for Maxwell's equations with Lorentz, Kerr and Raman effects. Both types of schemes employ second order time integrators, while utilizing high order discretizations in space. The schemes are benchmarked on several one-dimensional test examples and their performance in stability and accuracy are validated. The objective of the present work, is to conduct numerical dispersion analysis of the aforementioned DGTD and FDTD schemes for Maxwell's equations in linear Lorentz media, which can guide us in the optimal choice of numerical discretization parameters for more general dispersive and nonlinear models.

There has been abundant study on the dispersion analysis of DG methods. Most work was carried out for semi-discrete schemes, e.g., for scalar linear conservation laws \cite{ainsworth2004dispersive, hu1999analysis, hu2002eigensolution, sherwin2000dispersion}, and for the second-order wave equation \cite{ainsworth2006dispersive}. Dispersive behavior of fully discrete DGTD schemes is studied for the one-way wave equation \cite{yang2013dispersion, ainsworth2014dispersive} and two-way wave equations \cite{cheng2017L2}. Particularly, in \cite{sarmany2007dispersion} the accuracy order of the dispersion and dissipation errors of nodal DG methods with Runge-Kutta time discretization for Maxwell's equations in free space are analyzed numerically.
The stability and dispersion properties of a variety of FDTD schemes applied to Maxwell's equations in free space are also well known (see \cite{taflove2005computational}). Additionally, various time domain finite element methods have been devised for the numerical approximation of Maxwell's equations in free space (see \cite{monk1992comparison, lee1997time}  and the references therein). There has been relatively less work on phase error analysis for dispersive dielectrics; see \cite{taflove2005computational, petropoulos1994stability, prokopidis2004fdtd, bokil2012} for finite difference methods and \cite{banks2009analysis} for finite element methods.

To the best of our knowledge, the present work is the first in the literature to conduct dispersion analysis of fully discrete DGTD methods for Maxwell's equations in  Lorentz dispersive media and providing comparisons of the numerical dispersion errors with those of fully discrete FDTD methods. By rigorous quantification of the numerical dispersion error for such dispersive Maxwell systems, we make  comparisons of the DGTD and FDTD methods taking into account the model parameters, spatial and temporal accuracy and mesh sizes in the design of the schemes. Given the popularity of both DGTD and FDTD methods in science and engineering, such a comparison of errors between the two schemes will provide practioners of these methods with guidelines on their proper implementation. 

Our dispersion analysis indicates that there is a complicated dependence of dispersion errors on the model parameters, orders of spatial and temporal discretizations, CFL conditions as well as mesh discretization parameters. We compute and plot a variety of different measures of numerical dispersion errors as functions of the quantity $\frac{\omega}{\omega_1}$, where $\omega_1$ is the resonance frequency of the Lorentz material, and $\omega$ is an angular frequency. These measures include normalized phase and group velocities, attenuation constants and an energy velocity \cite{gilles2000comparison}. The parameter range of the quantity $\frac{\omega}{\omega_1}$ separates the response of the material into distinct bands. We find that some counterintuitive results can occur for high-loss materials where a low order scheme can have smaller numerical dispersion error than a higher order scheme. Since this situation does not occur in non-dispersive dielectrics, our results demonstrates the need to analyze and study the numerical dispersion relation for the Lorentz media beyond those for the case of free space that commonly appear in the literature. We have made quantitative comparisons of the high order FD and DG schemes based on the metrics discussed above. We also identify the differences in numerical dispersion due to the temporal integrator used. In particular, our results clearly identify the limitation of the second order temporal accuracy of our time discretizations, by identifying distinct bands in the frequency parameter ranges where the high order spatial accuracy of either the DG or FD schemes is unable to alleviate the error in numerical dispersion due to time discretization. 

The rest of the paper is organized as follows. In Section \ref{model} we introduce Maxwell's equations in a one spatial dimensional Lorentz dispersive material. In Sections \ref{dispersion} and \ref{time}, we present and analyze the dispersion relations and the relative phase errors for the PDE model, and two semi-discrete in time finite difference numerical schemes, respectively. In Sections \ref{semifdtd} and \ref{fdtd} numerical dispersion errors in semi-discrete in space staggered FD methods, and fully space-time discrete FDTD methods, respectively, are considered, while numerical dispersion errors in semi-discrete in space DG methods and fully discrete DGTD methods are studied in Sections \ref{semidg}, and \ref{dgtd}, respectively. In Section \ref{sec:numerical}, we define four quantities that provide different measures of numerical dispersion error and compare these for the FDTD and DGTD methods. Interpretations and conclusions of our results are made in Section \ref{conclude}.

\section{Maxwell's Equations in a Linear Lorentz Dielectric}
\label{model}

We begin by introducing Maxwell's equations in a non-magnetic, non-conductive medium $\Omega \subset \mathbb{R}^d$, $d=1,2,3$, from time 0 to $T$, containing no free charges, that govern the dynamic evolution of the electric field $\mathbf{E}$ and the magnetic field $\mathbf{H}$ in the form 
\begin{subequations}\label{eq:max}
	\begin{align}\label{eq:max1}
		&\ds\dd{t}{\mathbf{B}}+{\bf{\nabla}}\times \mathbf{E} = 0, \ \text{in} \  (0,T]\times \Omega,  \\[1.5ex]
		\label{eq:max2}
		&\ds\dd{t}{\mathbf{D}} -{\bf{\nabla}}\times \mathbf{H} = 0, \  \text{in} \ (0,T]\times \Omega, \\[1.5ex]
		\label{eq:max3}
		& {\bf{\nabla}}\cdot \mathbf{B}  = 0, \ {\bf{\nabla}}\cdot \mathbf{D} = 0, \ \text{in} \  (0,T]\times \Omega,
	\end{align}
\end{subequations}
along with initial data that satisfies the Gauss laws \eqref{eq:max3}, and appropriate boundary data. System \eqref{eq:max} has to be completed by constitutive laws on $[0,T] \times \Omega$. The electric flux density $\mathbf{D}$, and the magnetic induction $\mathbf{B}$, are related to the electric field and magnetic field, respectively, via the constitutive laws 
\begin{equation}
	\label{eq:constD}
	\mathbf{D} = \epsilon_0(\epsilon_\infty\mathbf{E}+\mathbf{P}), \ \ \mathbf{B} = \mu_0\mathbf{H}.
\end{equation}
The parameter $\epsilon_{0}$ is the electric permittivity of free space, while $\mu_0$ is the magnetic permeability of free space. The term $\epsilon_0 \epsilon_\infty \mb{E}$ captures the linear instantaneous response of the medium to the EM fields, with $\epsilon_{\infty}$ defined as the relative electric permittivity in the limit of infinite frequencies. The macroscopic  {\em (electric) retarded polarization} $\mathbf{P}$  is modeled as a single pole resonance Lorentz dispersion mechanism, 
%
in which the time dependent evolution of the polarization follows the second order ODE \cite{gilles2000comparison, taflove2005computational} 
\begin{equation}
	\label{eq:polar:P}
	\frac{\partial^2 \mathbf{P}}{\partial t^2} + 2\gamma\frac{\partial\mathbf{P}}{\partial t} +\omega_1^2\mathbf{P} =\omega_p^2\mb{E}.
\end{equation}
In the ODE \eqref{eq:polar:P}, $\omega_1$ and $\omega_p$ are the resonance and plasma frequencies of the medium, respectively, and $\gamma$ is a damping constant. The plasma frequency is related to the resonance frequency via the relation $\omega_p^2 = (\epsilon_s-\epsilon_\infty)\omega_1^2 := \epsilon_d\omega_1^2$. Here $\epsilon_s$ is defined as the relative permittivity at zero frequency, and $\epsilon_d$ measures the strength of the electric field coupling to the linear Lorentz dispersion model. We note that the limit $\epsilon_d\rightarrow 0$, or $\epsilon_s\rightarrow \epsilon_\infty$ corresponds to a  linear dispersionless dielectric.

In this paper, we focus on a one dimensional Maxwell model on $\Omega=\mathbb{R}$  that is obtained from \eqref{eq:max}, \eqref{eq:constD} and \eqref{eq:polar:P} by assuming an isotropic and homogeneous material in which electromagnetic plane waves are linearly polarized and propagate in the $x$ direction. Thus, the electric field is represented by one scalar component $E := E_z$, while the magnetic field is represented by the one component $H := H_y$. All the other variables are similarly represented by single scalar components.  We convert the second order ODE \eqref{eq:polar:P} for the linear retarded polarization $P$ to first order form by introducing the linear polarization  current density $J$,
\begin{align}
\label{eq:P_ODE}
	\frac{\partial P}{\partial t} = J, \quad
	\frac{\partial J}{\partial t} = -2\gamma J -\omega_1^2P+\omega_p^2E. 
\end{align} 

%

We consider a rescaled formulation of the resulting one spatial dimensional Maxwell-Lorentz system with the following scaling: let the reference time scale be $t_0$, and reference space  scale be $x_0$ with $x_{0}=ct_{0}$ and $c=1/\sqrt{\mu_0\epsilon_0}$. Henceforth, the rescaled fields and constants are defined based on  a reference electric field $E_0$ as follows,
\begin{align*}
& (H/E_{0})\sqrt{\mu_{0}/\epsilon_{0}}\rightarrow H, \ \ \
 D/(\epsilon_{0}E_{0}) \rightarrow D, \ \ \
 P/E_{0}\rightarrow P, \ \ \
(J/E_{0})t_{0}\rightarrow J,\ \ \
E/E_{0}\rightarrow E, \\
& \omega_{1}t_{0}\rightarrow\omega_{1}, \ \ \
 \omega_{p}t_{0}\rightarrow\omega_{p}, \ \ \
 \gamma t_{0}\rightarrow \gamma, 
\end{align*}
where for simplicity, we have used the same notation to denote the scaled and original variables. In summary,  we arrive at the following dimensionless Maxwell's equations with linear Lorentz dispersion in one dimension:
\begin{subequations}
\label{eq:sys}
	\begin{align}
	\frac{\partial H}{\partial t} 
	&= \frac{\partial E}{\partial x}, \label{eq:sys1}\\ 
	\frac{\partial D}{\partial t} 
	&= \frac{\partial H}{\partial x}, \label{eq:sys2}\\
	\frac{\partial P}{\partial t} 
	&= J, \label{eq:sys3}\\
	\frac{\partial J}{\partial t} 
	&= -2\gamma J -\omega_1^2P+\omega_p^2E, \label{eq:sys4} \\
	D &= \epsilon_\infty E +P. \label{eq:sys7}
	\end{align}
\end{subequations} 

\section{Dispersion Relations}
\label{dispersion}
The Maxwell-Lorentz system \eqref{eq:sys} is a linear dispersive system, i.e. it admits plane wave solutions of the form $e^{i\left( kx - \omega t \right)}$ for all its unknown field variables, with the property that the speed of propagation of these waves is not independent of the wave number $k$ or the angular frequency $\omega$ \cite{trefethen1982group}. 
In this section, we derive the dispersion relation of  \eqref{eq:sys} and highlight its main properties. We assume the space-time harmonic variation
\begin{align}
\displaystyle
\label{ExDis}
X(x,t) \equiv X_0 e^{i\left( kx - \omega t \right)},
\end{align}

\noindent of all field components $X \in \{H, E, P, J\}$. Substituting \eqref{ExDis} in \eqref{eq:sys} yields the system 
\begin{subequations}
	\label{DisEx}
	\begin{align}
	\omega H_0 + k E_0 &= 0, \\
	kH_0 + \epsilon_\infty \omega E_0 + \omega P_0 &= 0, \\
	i \omega P_0 + J_0 &= 0, \\
	\omega_p^2 E_0 - \omega_1^2 P_0 + \left(i \omega - 2\gamma\right)J_0 &= 0.
	\end{align}
\end{subequations}

Define the vector $\displaystyle \textbf{U}= [H_0, E_0, P_0, J_0]^T$ containing all amplitudes of the field solution, then \eqref{DisEx} can be rewritten as a linear system, given by
\begin{align}
\label{DisEx2}
\mathcal{A}\textbf{U} = \textbf{0}, \quad \text{with} \quad
\mathcal{A} =\begin{pmatrix}
\omega & k & 0 & 0 \\
k &  \epsilon_\infty \omega  & \omega & 0 \\
0 & 0 & i \omega & 1  \\
0 & \omega_p^2  &  - \omega_1^2  & i \omega - 2\gamma \\
\end{pmatrix}.
\end{align}
By solving $\det(\mathcal{A})=0$, we obtain the exact dispersion relation for  \eqref{eq:sys} as
 \begin{align}
	\displaystyle
	\label{DisEx4}
	k 
	= \pm k^{\text{ex}}, \qquad \textrm{with} \quad k^{\text{ex}} 
	= \omega \sqrt{\epsilon(\WH{\omega}; \mathbf{p})}, 
	\ \ \textrm{and} \ \  
	\epsilon(\WH{\omega}; \mathbf{p}) = \epsilon_\infty \left(1- 
	\frac{\epsilon_{d}/\epsilon_\infty }{\displaystyle \WH{\omega}^2 + 2 i \WH{\gamma} \WH{\omega} - 1}
	\right).
	\end{align}
Here, $\epsilon(\WH{\omega}; \mathbf{p}) $ is the permittivity of the medium dependent on the ``relative" frequency $\WH{\omega}=\omega/\omega_{1}$ and the parameter set $\mathbf{p} = [\epsilon_s, \epsilon_\infty, \WH{\gamma}]$, with $\WH{\gamma}=\gamma/\omega_1$. The permittivity is clearly frequency dependent and displays the dispersive nature of the system. A major goal in the design and construction of numerical methods for linear dispersive PDEs is to devise methods that accurately capture the medium's complex permittivity \cite{taflove2005computational}. We will assume that $\epsilon_s >0, \epsilon_\infty >0$ and  $\epsilon_d=\epsilon_s-\epsilon_\infty>0$. These assumptions are based on physical considerations \cite{taflove2005computational}. 
In the dispersion analysis, we assume $\omega$ is a real number, and 
  restrict
$\omega\geq0$ in this work.
Note that $\epsilon(\WH{\omega}; \mathbf{p}) $ and $k$ can be complex, depending on the values that certain parameters assume.

For  lossless materials (i.e. $\WH{\gamma}=0$), the {\it medium absorption band} is defined by $\WH{\omega}\in [1, \sqrt{\epsilon_{s}/\epsilon_{\infty}}]$, in which $\epsilon(\WH{\omega}; \mathbf{p})\le0$  and $k^{\text{ex}}$ is an imaginary number or zero.  Outside the medium absorption band, i.e. for other $\WH{\omega}$ values, we have $\epsilon(\WH{\omega}; \mathbf{p})>0$  and $k^{\text{ex}}$ is a real number.
Moreover, it is easy to check $|k^{\text{ex}}|\rightarrow\infty$ as $\WH{\omega}$ approaches $1$ (the \emph{resonance frequency}, which is  also the lower bound of the medium absorption band) and $k^{\text{ex}}=0$ at the upper bound $\WH{\omega}=\sqrt{\epsilon_{s}/\epsilon_{\infty}}.$ In this paper, we are mainly interested in low-loss materials, i.e.   $\WH{\gamma}>0$ with $\WH{\gamma}\ll1$. In this case, the dispersion relation retains similar properties, which means $|k^{\text{ex}}|$ is a large number around $\WH{\omega}=1$ and a small number near $\WH{\omega}=\sqrt{\epsilon_{s}/\epsilon_{\infty}}$. This behavior of the exact dispersion relation has implications for the numerical dispersion errors, as illustrated in later sections.


\begin{Remark}
\label{rk1}
  In the literature, dispersion relations can be presented in two ways; 1) representing the continuous or discrete angular frequency $\omega\in\mathbb{C}$ as a function of the exact and continuous  wave number $k\in\mathbb{R}$
(and also  of the model parameters and possible mesh parameters);
 2) representing the continuous or discrete wave number $k\in \mathbb{C}$ as a function of the exact and continuous 
 angular frequency $\omega\in \mathbb{R}$ \cite{taflove2005computational}.
 In the first approach, we will obtain a fourth order polynomial for $\omega$ as a function of $k$ and other parameters. We provide some insight into approach 1 in Appendix A, for the semi-discrete in space FDTD discretizations in which the effect of high order FDTD spatial approximations on the dispersion relation is clearly evident in terms of the {\it symbol} of the spatial discretization operators.  
 In this paper, we mainly use  the second approach since in this approach we are able to explicitly identify the effects of discretization on the permittivity of the Maxwell-Lorentz model \eqref{eq:sys}. 
\end{Remark}


Before we proceed, for convenience of the readers, we gather some notations frequently used in the paper, together with the place of their first appearances in a table listed below. 



\begin{table}[H]
	\label{tab:notation}
	\centering
	\caption{Notations used and the place of their first appearance. The symbol $*$ in the superscript can be LF (for the leap-frog temporal scheme) or TP (for the trapezoidal temporal scheme).}
	\renewcommand{\arraystretch}{1.5}
	\begin{tabular}{c c c c c c c c c} \hline
		&	$\WH{\omega}$  &  $\WH{\gamma}$ & $W$ & $W_{1}$ &  $K$ & $\epsilon(\WH{\omega}; \mathbf{p})$ &  $ \delta(\WH{\omega};\mathbf{p})$  
		&  $\Psi^*(\WH{\omega})$
		\\\hline
		\textbf{Def} &	$\omega / \omega_{1}$ &  
		$\gamma/ \omega_{1}$ & 
		$\omega\Delta t$ & 
		$\omega_{1}\Delta t$ & 
		$ k^{\text{ex}} h$  & 
		$\ds  \epsilon_\infty - \frac{\epsilon_{d}}{\displaystyle \WH{\omega}^2 + 2 i \WH{\gamma}\, \WH{\omega} - 1}$  &
		$\ds\frac{\epsilon_{d} \, \WH{\omega} \left( \WH{\omega}+i  \WH{\gamma} \right)} { ( \WH{\omega}^2+2i \WH{\gamma}\,\WH{\omega} -1)^2}$  
		& $\ds \left| \frac{k^{\text{ex}}(\WH{\omega}) - k^*(\WH{\omega})}{k^{\text{ex}}(\WH{\omega})} \right|$\\
		\textbf{Eqn} &	  \eqref{DisEx4}  &   \eqref{DisEx4}  & 
		\eqref{eq:matrixA_LF}  &  \eqref{eq:matrixA_LF} & 
		\eqref{Dissemi4}  &    \eqref{DisEx4} & 
		\eqref{delta}
		& \eqref{eq:error}, \eqref{ERRTP} \\
		\hline
	\end{tabular}
	\renewcommand{\arraystretch}{1}
\end{table}

\vspace{-.5cm}
\begin{table}[H]
	\centering
	\renewcommand{\arraystretch}{1.5}
	\begin{tabular}{c c c c c } \hline
		&  $\Psi_{\text{FD},2M}(\WH{\omega})$
	    &  $\Psi^*_{\text{FD},2M}(\WH{\omega})$ 
	    &  $\Psi_{\text{DG},p}(\WH{\omega})$
	    &  $\Psi^*_{\text{DG},p}(\WH{\omega})$ 
		\\\hline
		\textbf{Def} 
	    & $\ds \left| \frac{k^{\text{ex}}(\WH{\omega}) - k_{\text{FD},2M}(\WH{\omega})}{k^{\text{ex}}(\WH{\omega})} \right|$
		& $\ds \left| \frac{k^{\text{ex}}(\WH{\omega}) - k^*_{\text{FD},2M}(\WH{\omega})}{k^{\text{ex}}(\WH{\omega})} \right|$ 
		 & $\ds \left| \frac{k^{\text{ex}}(\WH{\omega}) - k_{\text{DG},p}(\WH{\omega})}{k^{\text{ex}}(\WH{\omega})} \right|$
		& $\ds \left| \frac{k^{\text{ex}}(\WH{\omega}) - k^*_{\text{DG},p}(\WH{\omega})}{k^{\text{ex}}(\WH{\omega})} \right|$ \\
		\textbf{Eqn} 
	    &\eqref{ERRTP2}  & 
		Figure \ref{Fig:Phase_Error_FD_fully1}  & Figure \ref{Fig:semi_DG} & Figure \ref{Fig:Phase_Error_DG_fully1}  \\
		\hline
	\end{tabular}
	\renewcommand{\arraystretch}{1}
\end{table}

  
\section{Second Order Accurate Temporal Discretizations}
\label{time}

This section concerns the dispersion analysis of the semi-discrete in time schemes. 
Continuing from our previous work  \cite{bokil2017energy,bokil2018high}, we consider two types of commonly used second-order time schemes for the linear system \eqref{eq:sys}, both implicit in the ODE parts. Let $\Delta t > 0$ be a temporal mesh step. Suppose  $u^{n}(x)$ is the solution at time $t^{n}= n\Delta t, n \in \mathbb{N}$, with $u=H,\, E,\, D,\, P,\, J$. Then, we compute $u^{n+1}(x)$ at time $t^{n+1}=t^{n}+\Delta t$ by the following methods. The first scheme uses a staggered leap-frog discretization in time for the PDE part, with the magnetic field $H$ staggered in time from the rest of the field components. The scheme is given by:
\begin{subequations}
	\label{eq:LF}
	\begin{align}
	\frac{H^{n+1/2}-H^{n}}{\Delta t/2}
	&= \frac{\partial E^n}{\partial x}, \label{eq:LF1}\\
	\frac{D^{n+1}-D^{n}}{\Delta t}
	&= \frac{\partial H^{n+1/2}}{\partial x}, \label{eq:LF2}\\
	\frac{P^{n+1}-P^{n}}{\Delta t} 
	&= \frac{1}{2} \left(J^{n}+J^{n+1}\right), \label{eq:LF3}\\
	\frac{J^{n+1}-J^{n}}{\Delta t} 
	&= -\gamma  \left(J^{n}+J^{n+1}\right) -\frac{\omega_{1}^{2}}{2} \left(P^{n}+P^{n+1}\right) + \frac{\omega_{p}^2}{2} \left(E^{n}+E^{n+1}\right), \label{eq:LF4}\\
	D^{n+1} 
	&= \epsilon_{\infty} E^{n+1} + P^{n+1}, \label{eq:LF5}\\
	\frac{H^{n+1}-H^{n+1/2}}{\Delta t/2}
	&= \frac{\partial E^{n+1}}{\partial x}. \label{eq:LF6}
	\end{align}
\end{subequations}

The second scheme, which is a fully implicit scheme based on the trapezoidal rule, is given as follows:
\begin{subequations}
	\label{eq:TP}
	\begin{align}
	\frac{H^{n+1}-H^{n}}{\Delta t} 
	 =& \frac{1}{2} \left( \frac{\partial E^{n+1}}{\partial x} + \frac{\partial E^{n}}{\partial x}\right), \label{eq:TR1} \\
	\frac{D^{n+1}-D^{n}}{\Delta t} 
	 =& \frac{1}{2} \left( \frac{\partial H^{n+1}}{\partial x} + \frac{\partial H^{n}}{\partial x}\right), \label{eq:TR2} \\
	\frac{P^{n+1}-P^{n}}{\Delta t} =& \frac{1}{2} \left(J^{n}+J^{n+1}\right), \label{eq:TR3}\\
	\frac{J^{n+1}-J^{n}}{\Delta t} =& -\gamma  \left(J^{n}+J^{n+1}\right) -\frac{\omega_{1}^{2}}{2} \left(P^{n}+P^{n+1}\right) + \frac{\omega_{p}^2}{2} \left(E^{n}+E^{n+1}\right), \label{eq:TR4}\\
	D^{n+1} =& \epsilon_{\infty} E^{n+1} + P^{n+1}. \label{eq:TR5}
	\end{align}
\end{subequations}

Similar to the continuous case, we can perform dispersion analysis on the semi-discrete schemes \eqref{eq:LF} and \eqref{eq:TP} by  assuming the time discrete plane wave solution as
\begin{align}
\label{eq:disper_semi}
X^{n}(x)\equiv X_{0} e^{i(k^{\text{*}}x-\omega t_{n})},
\end{align} 
where $*$ can be LF (with respect to the leap-frog scheme \eqref{eq:LF}) or TP (with respect to the trapezoidal scheme \eqref{eq:TP}).
Define  $\displaystyle \textbf{U}= [H_0, E_0, P_0, J_0]^T$ as the vector containing all amplitudes of the field solutions. Substituting \eqref{eq:disper_semi} in the schemes \eqref{eq:LF} or \eqref{eq:TP}, we obtain linear systems for each case in the form 
\begin{align}
\mathcal{A}^{*}\textbf{U} = \textbf{0}.
\end{align} 
The semi-discrete numerical dispersion relation can be then obtained from $\det(\mathcal{A}^{*})=0.$ 

For the leap-frog scheme \eqref{eq:LF}, we have 
\renewcommand{\arraystretch}{1.5}
\begin{align}
\label{eq:matrixA_LF}
\mathcal{A}^{\text{LF}} =\begin{pmatrix}
	\sin\left( \frac{W}{2} \right) &
\frac{\Delta t}{2}k^{\text{LF}}  & 0 & 0 \\
	\frac{\Delta t}{2}k^{\text{LF}} & 
 \epsilon_\infty \sin\left( \frac{W}{2} \right) & 
  \sin\left( \frac{W}{2} \right)  & 0 \\
0 & 0 &	i \sin\left( \frac{W}{2} \right) &
 \frac{\Delta t}{2} \cos\left( \frac{W}{2} \right) \\
 0 & \frac{\Delta t}{2} \omega_p^2 \cos\left( \frac{W}{2} \right) & 
 - \frac{\Delta t}{2} \omega_1^2 \cos\left( \frac{W}{2} \right) &
 i \sin\left( \frac{W}{2} \right) - \gamma \Delta t \cos\left( \frac{W}{2} \right) \\
\end{pmatrix},
\end{align}
\renewcommand{\arraystretch}{1}
 where $W:=\omega \Delta t = \WH{\omega} W_1,$ with $W_1:=\omega_1 \Delta t.$ This yields the dispersion relation
\begin{align}
\label{DissemiLF3-2}
k^{\text{LF}}
=
\pm \omega
\sqrt{\epsilon(\WH{\omega}^{\mathrm{LF}}; \mathbf{p}^{\mathrm{LF}})}, \ \ \
\textrm{with} \quad
\epsilon(\WH{\omega}^{\mathrm{LF}};\mathbf{p}^{\mathrm{LF}}) =
\epsilon_\infty^{\mathrm{LF}}\left( 
1 - 
\frac{
	\epsilon_{d}^{\mathrm{LF}}/\epsilon_\infty^{\mathrm{LF}}
}{
	\left(\WH{\omega}^{\mathrm{LF}}\right)^2 + 2i \WH{\gamma}^{\mathrm{LF}}\, \WH{\omega}^{\mathrm{LF}} - 1
}\right),
\end{align}
where 
 $s_\omega := \displaystyle \frac{\sin(\frac{W}{2})}{\frac{W}{2}} $ and $r_\omega := \displaystyle \frac{\tan(\frac{W}{2})}{\frac{W}{2}} $ as in  \cite{petropoulos1994stability}, and 
 $\WH{\omega}^{\mathrm{LF}} = \WH{\omega} r_\omega, \mathbf{p}^{\mathrm{LF}} = [\epsilon_s^{\mathrm{LF}}, \epsilon_\infty^{\mathrm{LF}}, \WH{\gamma}^{\mathrm{LF}}]$, with components given by the identities
\begin{equation}
\label{notation1.1}
\displaystyle
\epsilon_s^{\mathrm{LF}} = \epsilon_s s_\omega^2, \ \ \epsilon_\infty^{\mathrm{LF}} = \epsilon_\infty s_\omega^2, \ \ \WH{\gamma}^{\mathrm{LF}} = \WH{\gamma}. 
\end{equation}  
In this form, we can clearly identify how the leap-frog time discretization misrepresents the permittivity  by misrepresenting   the parameters of the model. These misrepresentations are solely due to the discretizations of the ODEs by the leap-frog time integrator. The misrepresentations depend on the value of the (exact) angular frequency that is chosen, and in particular as $\frac{W}{2}$ approaches zero, the discrete parameters approach the continuous ones. Thus, a guideline for practitioners using this time integrator to control these misrepresentation, is to choose $\Delta t$ so that $\cos \left( \frac{W}{2} \right) \approx 1$ across the range of frequencies present in the short pulse that propagates in the medium \cite{petropoulos1994stability}. 

To further analyze the dispersion error, we consider the regime when $W  \ll 1$, and obtain the Taylor expansion of \eqref{DissemiLF3-2} with respect to $W$ as
\begin{align}
\label{eq:dis_LF}
k^{\text{LF}} 
= \pm  k^{\text{ex}}\left( 1 + \frac{1}{12}\left(  \frac{\delta(\WH{\omega};\mathbf{p})} {\epsilon(\WH{\omega};\mathbf{p})} -\frac{1}{2} \right) W^2 + \mathcal{O}(W^4) \right),  
\end{align}
where 
\begin{equation}
\label{delta} 
\displaystyle
\delta(\WH{\omega};\mathbf{p}) 
= \frac{\epsilon_{d} \, \WH{\omega} \left( \WH{\omega}+i  \, \WH{\gamma} \right)} { \left( \WH{\omega}^2 +2i \, \WH{\gamma}\,\WH{\omega} -1 \right)^2}.  
\end{equation}

   We define the {\it relative phase error} for the LF scheme to be the ratio 
\begin{equation}
\label{eq:error}
\Psi^{\text{LF}}(\WH{\omega})
:=
\left|\frac{k^{\text{LF}}(\WH{\omega})-k^{\text{ex}}(\WH{\omega})}{k^{\text{ex}}(\WH{\omega})} \right| 
=  \left|\frac{\sqrt{\epsilon(\WH{\omega}^{\text{LF}};\mathbf{p}^{\text{LF}})} -\sqrt{\epsilon(\WH{\omega};\mathbf{p})}}{\sqrt{\epsilon(\WH{\omega};\mathbf{p})}}\right |. 
\end{equation}
Here, we consider $k^{{\text{LF}}}$ in \eqref{DissemiLF3-2}  with plus sign in front.  A similar definition will be used for all semi-discrete and fully discrete schemes that appear in this paper, and provides quantitative measurement of the numerical dispersion error. Equation \eqref{eq:dis_LF} verifies a second order dispersion error in time of the leap-frog scheme in the small time step limit.

%
%
Similarly,  for the trapezoidal method \eqref{eq:TP}, we can obtain
	\renewcommand{\arraystretch}{1.5}
	\begin{align}
	\mathcal{A}^{\text{TP}} =\begin{pmatrix}
	\sin\left( \frac{W}{2} \right) &
	\frac{\Delta t}{2}k^{\text{TP}}\cos\left( \frac{W}{2} \right) & 0 & 0 \\
	\frac{\Delta t}{2}k^{\text{TP}}	\cos\left( \frac{W}{2} \right) & 
	\epsilon_\infty \sin\left( \frac{W}{2} \right) & 
	\sin\left( \frac{W}{2} \right)  & 0 \\
	0 & 0 &	i \sin\left( \frac{W}{2} \right) &
	\frac{\Delta t}{2} \cos\left( \frac{W}{2} \right) \\
	0 & \frac{\Delta t}{2} \omega_p^2 \cos\left( \frac{W}{2} \right) & 
	- \frac{\Delta t}{2} \omega_1^2 \cos\left( \frac{W}{2} \right) &
	i \sin\left( \frac{W}{2} \right) - \gamma \Delta t \cos\left( \frac{W}{2} \right) \\
	\end{pmatrix}.
	\end{align}
	\renewcommand{\arraystretch}{1}
 This leads to the dispersion relation 
	\begin{align}
	\label{DissemiLF3-3}
	\displaystyle 
	k^{\text{TP}}
	&=
	\pm \omega
	\sqrt{\epsilon(\WH{\omega}^{\mathrm{TP}}; \mathbf{p}^{\mathrm{TP}})}
	= \frac{s_\omega}{r_\omega}k^{\text{LF}}
	, \ \ \
	\textrm{with} \quad
	\epsilon(\WH{\omega}^{\mathrm{TP}}; \mathbf{p}^{\mathrm{TP}}) =
	\epsilon_\infty^{\mathrm{TP}}\left( 
	1 - 
	\frac{\displaystyle 
		\epsilon_{d}^{\mathrm{TP}}/\epsilon_\infty^{\mathrm{TP}}
	}{
		\displaystyle 
		\left(\WH{\omega}^{\mathrm{TP}}\right)^2 + 2i \, \WH{\gamma}^{\mathrm{TP}} \, \WH{\omega}^{\mathrm{TP}} - 1
	}\right),
	\end{align}
	where $\WH{\omega}^{\mathrm{TP}} = \WH{\omega} r_\omega, \mathbf{p}^{\mathrm{TP}} = [\epsilon_s^{\mathrm{TP}}, \epsilon_\infty^{\mathrm{TP}}, \WH{\gamma}^{\mathrm{TP}}]$, with components given as
	\begin{equation}
	\label{notation1.2}
	\displaystyle
	\epsilon_s^{\mathrm{TP}} = \epsilon_s r_\omega^2, \ \ \epsilon_\infty^{\mathrm{TP}}= \epsilon_\infty r_\omega^2, \ \ \WH{\gamma}^{\mathrm{TP}}= \WH{\gamma}. 
	\end{equation}  
Again, we can clearly identify how the trapezoidal  time discretization misrepresents the permittivity. In particular, this method misrepresents the dissipation and medium resonance in the same manner as the leap-frog method. However, the relative permittivities $\epsilon_\infty$ and $\epsilon_s$ are misrepresented in a different manner. Thus, the speeds of propagation of discrete plane waves are different in these two discretizations. In particular, the slow and fast speeds in the medium, corresponding to relative permittivities $\epsilon_s$ and $\epsilon_\infty$, respectively, are different.

In the small time step limit, for $W\ll 1$, we have
\begin{align}
\label{eq:dis_TP}
k^{\text{TP}} 
= \pm  k^{\text{ex}}\left( 1 + \frac{1}{12}\left(  \frac{\delta(\WH{\omega};\mathbf{p})}{\epsilon(\WH{\omega};\mathbf{p})} +1 \right) W^2 + \mathcal{O}(W^4) \right),
\end{align}
which indicates second order accuracy in time for the \emph{relative phase error} for the trapezoidal scheme defined, in a similar manner to the leap-frog scheme, as 
\begin{equation}
\label{ERRTP}
\displaystyle 
\Psi^{\text{TP}}(\WH{\omega})
:=\left| \frac{k^{\text{TP}}(\WH{\omega})-k^{\text{ex}}(\WH{\omega})}{k^{\text{ex}}(\WH{\omega})}\right|
=  \left|\frac{\sqrt{\epsilon(\WH{\omega}^{\text{TP}};\mathbf{p}^{\text{TP}})} -\sqrt{\epsilon(\WH{\omega};\mathbf{p})}}{\sqrt{\epsilon(\WH{\omega};\mathbf{p})}}\right |.
\end{equation}

Finally, we make qualitative comparisons of the leap-frog and trapezoidal temporal discretizations. 
For low-loss materials,  the conclusions can be implied from considering  the case of $\WH{\gamma}=0$. For this case, for a given set of parameters $\mathbf{p}$, $\epsilon(\WH{\omega}; \mathbf{p})$ and $\delta(\WH{\omega}; \mathbf{p})\geq0$ are real numbers. When $\WH{\omega}\rightarrow\sqrt{\epsilon_{s}/\epsilon_{\infty}}$, we have $\epsilon(\WH{\omega}; \mathbf{p})\rightarrow0$. This means  $\displaystyle\frac{\delta(\WH{\omega}; \mathbf{p})}{\epsilon(\WH{\omega}; \mathbf{p})}\rightarrow \infty$, and thus the leading error term of both temporal schemes would be approaching $\infty$. On the other hand, when $\WH{\omega} \rightarrow 1$, it is easy to check that $\displaystyle\frac{\delta(\WH{\omega};\mathbf{p})}{\epsilon(\WH{\omega}; \mathbf{p})}\rightarrow \infty$ as well. Hence, both time schemes will give large dispersion error at $\WH{\omega}=1$ and $\WH{\omega}=\sqrt{\epsilon_{s}/\epsilon_{\infty}}$, which are the two endpoints of the medium absorption band. 
In addition, when $W\ll 1$, if $\WH{\omega} \in (1, \sqrt{\epsilon_{s}/\epsilon_{\infty}})$, i.e. for values in the interior of the medium absorption band, we can prove that $\displaystyle\frac{\delta(\WH{\omega}; \mathbf{p})} {\epsilon(\WH{\omega}; \mathbf{p})}<-1$, which leads to the relation $\displaystyle \left | \frac{\delta(\WH{\omega}; \mathbf{p})} {\epsilon(\WH{\omega}; \mathbf{p})}-\frac{1}{2} \right | \geq \left | \frac{\delta(\WH{\omega}; \mathbf{p})} {\epsilon(\WH{\omega}; \mathbf{p})}+1 \right |$. 
This means the leap-frog scheme has a larger relative phase error than the trapezoidal scheme in the interior of the medium absorption band. For other values of $\WH{\omega}$ outside the medium absorption band we obtain $\epsilon(\WH{\omega}; \mathbf{p})>0$ and $\displaystyle \left | \frac{\delta(\WH{\omega}; \mathbf{p})} {\epsilon(\WH{\omega}; \mathbf{p})}-\frac{1}{2}\right | \leq \left | \frac{\delta(\WH{\omega};\mathbf{p})} {\epsilon(\WH{\omega};\mathbf{p})}+1\right |$. Hence, the leap-frog scheme would give a small relative phase error outside the absorption band.
For low-loss Lorentz medium, i.e., when $\WH{\gamma}\ll1$, we believe that these conclusions are still valid with a slight change in two peak positions (see Figure \ref{Fig:semi_time0}). 

Now we choose the following set of parameters, which are the same as in \cite{gilles2000comparison}, representing a low-loss Lorentz medium:
\begin{align}
\label{eq:parameter}
\displaystyle
\epsilon_s = 5.25, 
\quad  
\epsilon_\infty = 2.25, 
\quad 
\WH{\gamma}= 0.01.
\end{align}
 Taking $W_1$ as $\{\pi/15$, $\pi/30$, $\pi/60\}$, the relative phase errors  are plotted against $\WH{\omega}\in[0,3]$ in Figure \ref{Fig:semi_time0}.
We can observe that the phase errors always have two peaks around $\WH{\omega} = 1$ and $\WH{\omega} = \sqrt{\epsilon_{s}/\epsilon_{\infty}} \approx 1.527$, with no appreciable distinction between two temporal discretizations at these peaks. 
The relative phase errors in the two temporal schemes are basically the same for small $\WH{\omega}$. As $\WH{\omega}$ is increased towards 1, we see that the error in the leap-frog scheme is slightly smaller than that in the trapezoidal scheme. When $\WH{\omega}$ is between 1 and 1.527, the leap-frog scheme presents larger error than the trapezoidal method. Beyond 1.527, the trapezoidal scheme generates larger error than the leap-frog scheme. 
 There is no obvious difference between the dispersion errors of two time discretizations at the peaks.  
Therefore, in the last graph of Figure  \ref{Fig:semi_time0}, we  only plot the errors of the leap-frog time schemes at the peaks. We verify  the second order accuracy of the method when the mesh size is varied.  These observations are consistent with our analysis. 

\begin{figure}[h]
	\centering
	\includegraphics[scale= 0.28] {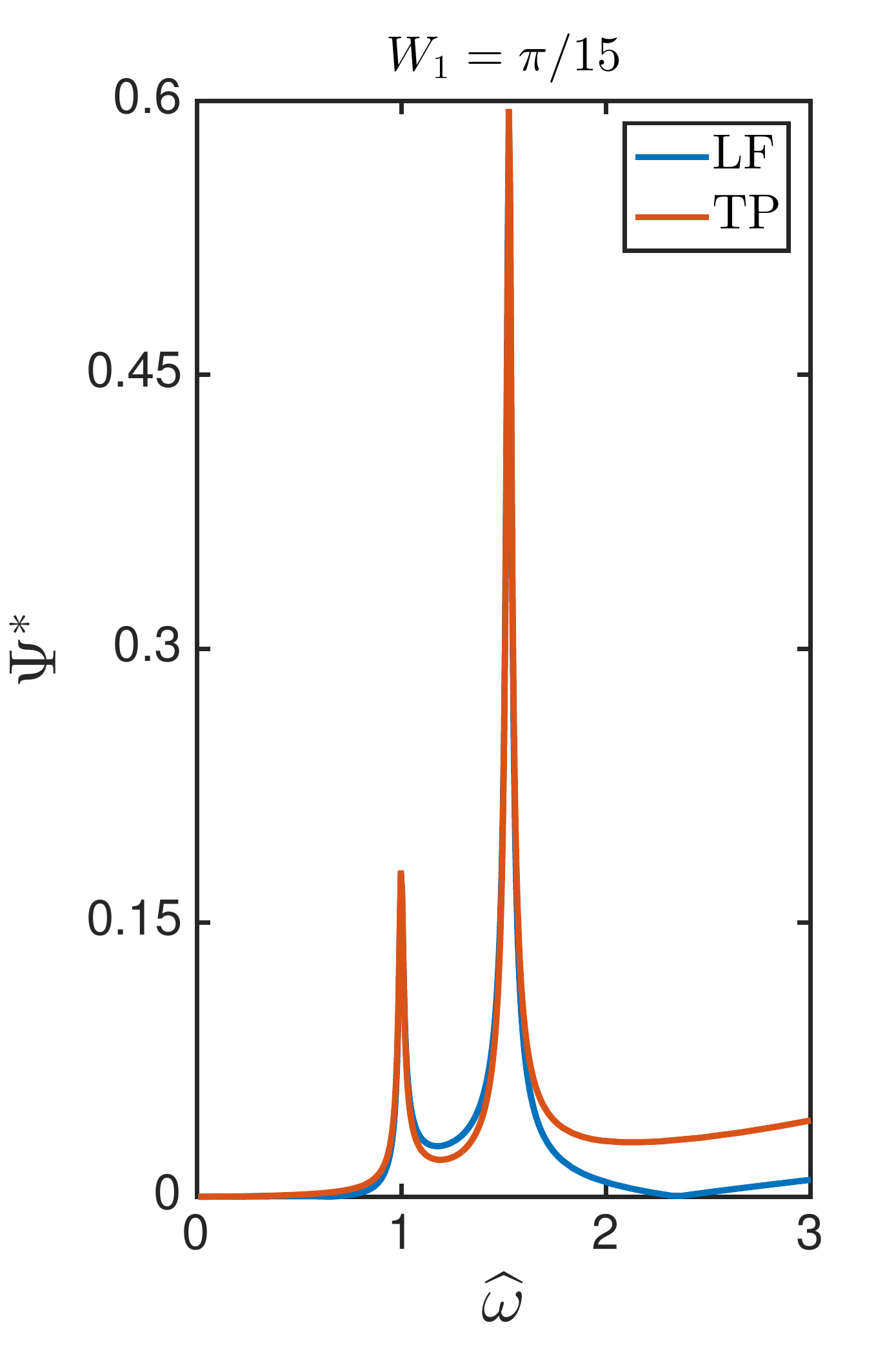}
	\includegraphics[scale= 0.28] {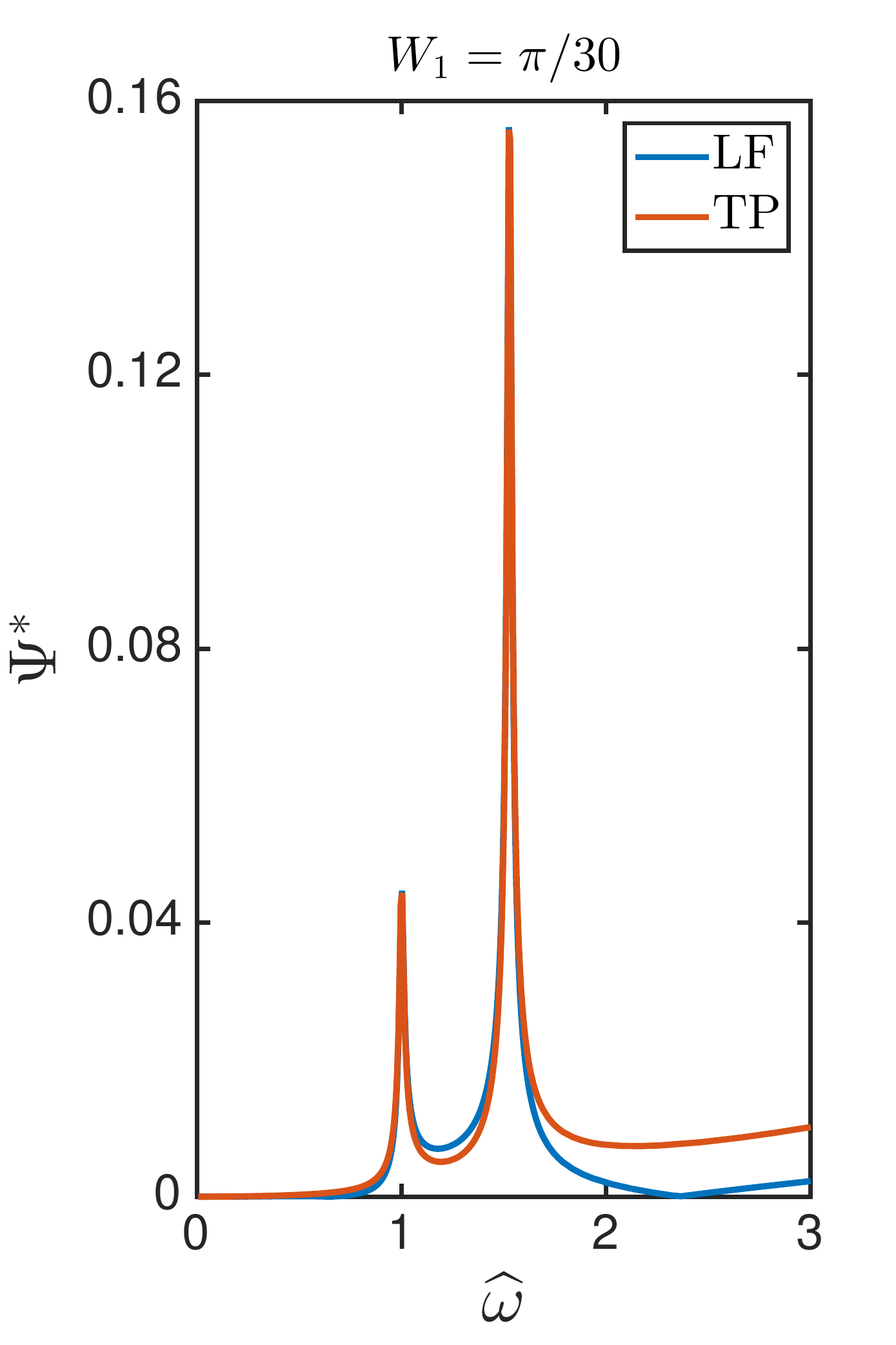}
	\includegraphics[scale= 0.28] {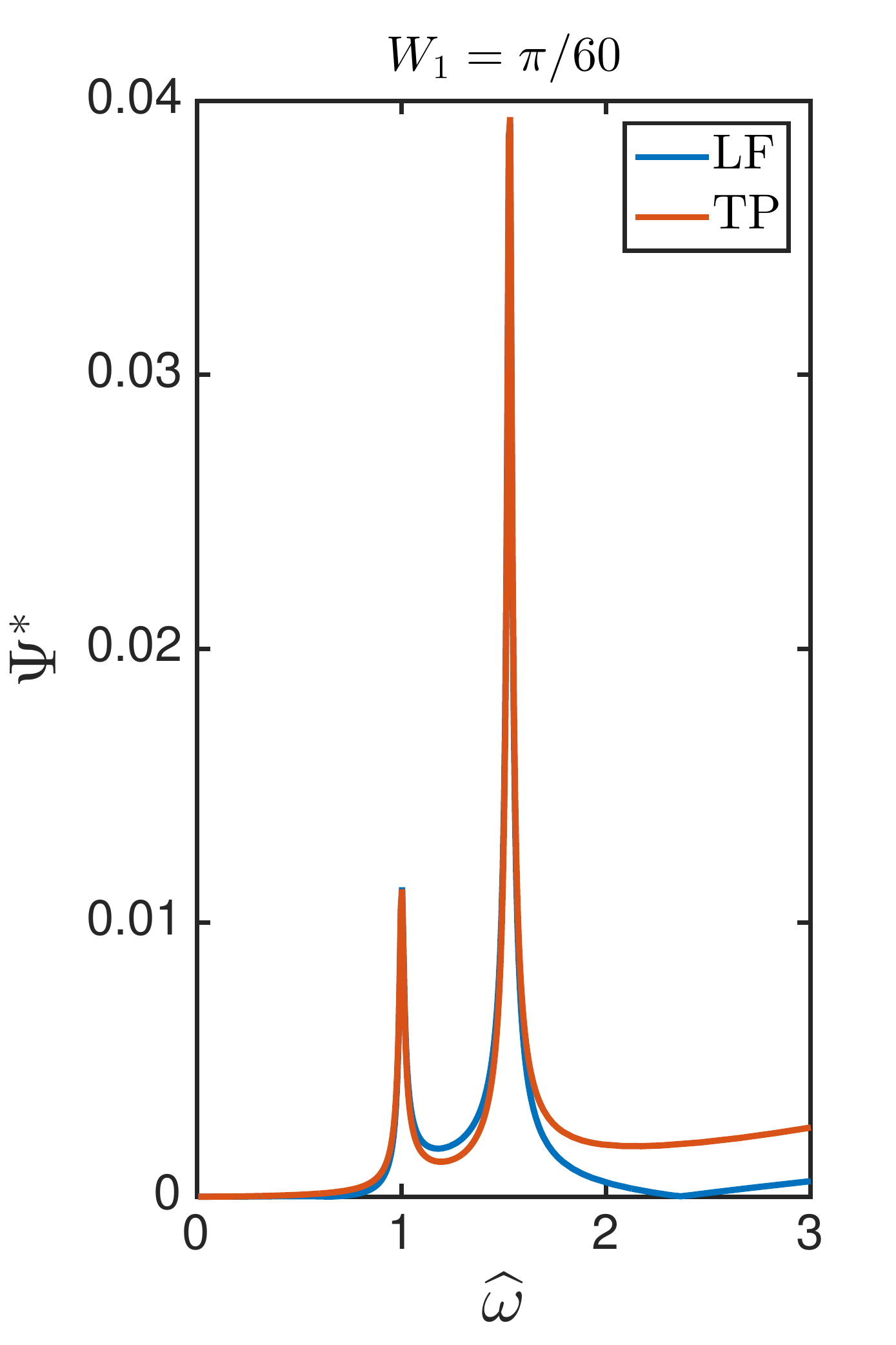}
	\includegraphics[scale= 0.28] {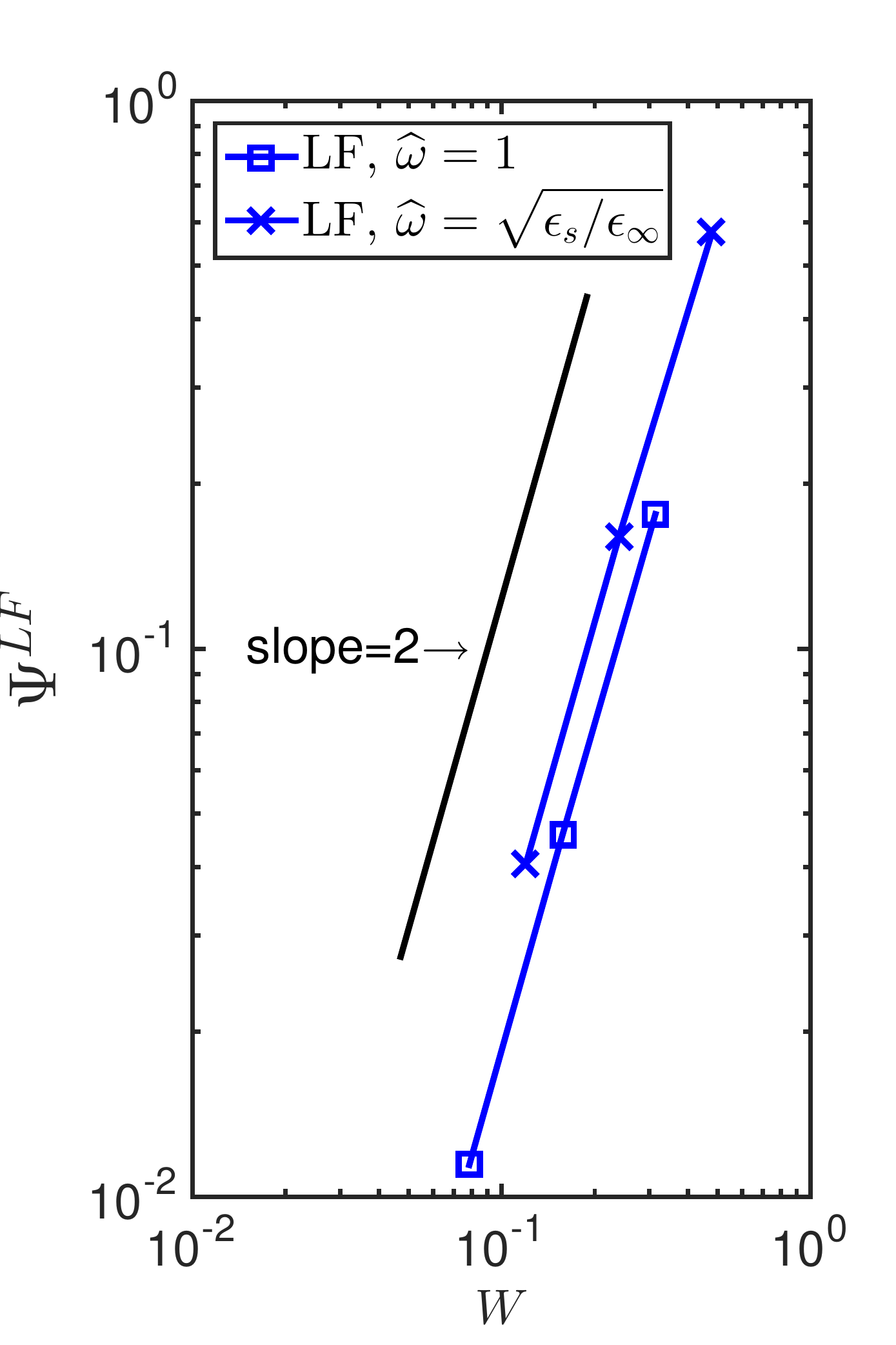}
	\caption{The relative phase error of leap-frog (LF) and trapezoidal (TP) time discretizations. In the first three plots we fix $W_1 \in \{\pi/15, \pi/30, \pi/60\}$, respectively, while we vary $\WH{\omega}\in[0,3]$. In the fourth plot, we fix $\WH{\omega} = 1$ or $\WH{\omega} = \sqrt{\epsilon_{s}/\epsilon_{\infty}}$ and consider three different values of $W$ corresponding to $W_1 \in \{\pi/15, \pi/30, \pi/60\}$, with leap-frog time discretization.}
	\label{Fig:semi_time0}
\end{figure}

\section{Spatial Discretization: High Order Staggered Finite Difference Methods}
\label{semifdtd}


In this section, we consider semi-discrete in space staggered finite difference schemes for   \eqref{eq:sys}. The spatial discretizations that we consider here for system \eqref{eq:sys} combined with a nonlinear instantaneous Kerr response and a Raman retarded nonlinear response have been recently developed in \cite{bokil2018high}.  The electric and magnetic fields are staggered in space and the discrete spatial operators have arbitrary even order, $2M, M \in \mathbb{N}$, accuracy in space. Below, we describe the semi-discrete spatial schemes, denoted as the FD2M scheme, and then we obtain and discuss dispersion relations of these  schemes.

As in \cite{bokil2018high}, we define two staggered grids on $\mathbb{R}$ with spatial step size $h$, the primal grid $G_{p}$, and the dual grid $G_{d}$, defined respectively, as 
\begin{align}
\label{eq:GP}
G_{p} = \{ j h \ | \  j\in\mathbb{Z} \}, \quad\text{and}\quad 
G_{d} = \{ (j+\frac{1}{2}) h \ | \ j\in\mathbb{Z} \}. 
\end{align}

The discrete magnetic field will be approximated at spatial nodes on the dual grid. These approximations are denoted by $H_{j+1/2}$, termed as  {\it degrees of freedom} (DoF) of $H$. All the other discrete fields will have their DoF at spatial nodes on the primal grid. For a continuous field variable $V$, $V_h$ denotes its corresponding {\it grid function}, defined as the set of all DoF on its respective grid. The semi-discrete scheme is given as follows:
\begin{subequations}
	\label{eq:semi}
	\begin{align}
	\frac{\partial H_{j+1/2}}{\partial t} 
	&= \left(\mathcal{D}^{(2M)}_{h} E_h\right)_{j +\frac{1}{2}}, \label{eq:semi1}\\ 
	\frac{\partial D_{j}}{\partial t} 
	&= \left(\widetilde{\mathcal{D}}^{(2M)}_{h} H_h\right)_j, \label{eq:semi2}\\
	\frac{\partial P_{j}}{\partial t} 
	&= J_{j}, \label{eq:semi3}\\
	\frac{\partial J_{j}}{\partial t} 
	&= -2\gamma J_{j} -\omega_1^2 P_{j}+\omega_p^2 E_{j}, \label{eq:semi4} \\
	D_{j}& = \epsilon_\infty E_{j} +P_{j}, \label{eq:semi5}
	\end{align}
\end{subequations} 
where $\displaystyle \mathcal{D}^{(2M)}_{h}$ and $\displaystyle \widetilde{\mathcal{D}}^{(2M)}_{h}$ are the  $2M$-th order finite difference approximations (with $M\in \mathbb{N}$) of the spatial differential operator $\partial_{x}$, on the primal and dual grids, respectively. These approximations are defined as 
\begin{subequations}
	\label{eq:D2M}
	\begin{align}
	& \left(\mathcal{D}^{(2M)}_{h} E_h \right)_{j+1/2}
	= \frac{1}{h} 
	\sum_{p=1}^{M} \frac{\lambda_{2p-1}^{2M}}{(2p-1)} 
	\left(E_{j+p} - E_{j-p+1}\right),\\
	& \left(\widetilde{\mathcal{D}}^{(2M)}_{h} H_h\right) _{j}
	= \frac{1}{h} 
	\sum_{p=1}^{M} \frac{\lambda_{2p-1}^{2M}}{(2p-1)} 
	\left(H_{j+p-\frac{1}{2}} - H_{j-p+\frac{1}{2}} \right),
	\end{align}
\end{subequations}

\noindent and $\lambda_{2p-1}^{2M},$ is  given as \cite{bokil2018high} 
\begin{equation}
\label{eq:lambda}
\lambda^{2M}_{2p-1} =
\ds\frac{2(-1)^{p-1}[(2M-1)!!]^2}{(2M+2p-2)!!(2M-2p)!!(2p-1)},
\end{equation} 
\noindent with the double factorial $n!!$  defined as
\begin{equation}
n!! = \begin{cases}
n\cdot (n-2)\cdot (n-4) \ldots 5\cdot 3\cdot 1 & n>0, \ \text{odd}\\ 
n\cdot (n-2)\cdot (n-4) \ldots 6\cdot 4\cdot 2 & n>0, \ \text{even}\\ 
1, & n = -1,0.
\end{cases}
\end{equation}

\subsection{Semi-discrete in space dispersion analysis} 
\label{semi}

In this section we analyze the spatial semi-discrete system \eqref{eq:semi}, i.e., the FD2M scheme. We assume that the semi-discrete system \eqref{eq:semi} has plane wave solutions of the form
\begin{align} 
\label{SemiDis} \displaystyle
X_j(t) \equiv
X_0 e^{i\left( k_{\text{FD},2M}j h - \omega  t \right)},
\end{align}
where $\ds k_{\text{FD},2M}$ represents the numerical wave number of the semi-discrete FD2M scheme. 
By substituting \eqref{SemiDis} in \eqref{eq:semi} we obtain the linear system   
\begin{align}
\label{Dissemi2}
\mathcal{A}_{\text{FD},2M}\textbf{U}_{\text{FD}} = \textbf{0}, 
\end{align} 
where  the vector $\displaystyle \textbf{U}_{\text{FD}} = [H_0, E_0, P_0, J_0]^T,$ and the matrix $\mathcal{A}_{\text{FD},2M}$ is given by 
\begin{align}
\mathcal{A}_{\text{FD},2M}=\begin{pmatrix}
\omega & \Lambda_{2M} & 0  & 0 \\
\Lambda_{2M} & \epsilon_\infty \omega & \omega & 0 \\
0 & 0 &	i \omega & 1 \\
0 &	\omega_p^2 & - \omega_1^2& i \omega - 2\gamma \\
\end{pmatrix}
\quad \text{with} \quad
\Lambda_{2M}=\frac{2}{h}\sum_{p=1}^{M} \frac{\lambda_{2p-1}^{2M}}{(2p-1)} 
\sin \left[ \left( p-\frac{1}{2} \right) k_{\text{FD},2M} h \right].
\end{align}
The numerical dispersion relation of the FD2M method is obtained by solving the characteristic equation of matrix $\mathcal{A}_{\text{FD},2M}$ and is given as 
\begin{align}
\displaystyle 
\label{Dissemi3}
\Lambda_{2M} 
= \pm \omega \sqrt{\epsilon(\WH{\omega};\mathbf{p})} = \pm \omega
\sqrt{\epsilon_\infty\left(1- 
	\frac{\epsilon_{d}/ \epsilon_\infty}{\displaystyle \WH{\omega}^2 + 2i \, \WH{\gamma} \, \WH{\omega} - 1}\right) } 
= \pm k^{\text{ex}}.
\end{align}
Using results from \cite{bokil2018high}, we can rewrite this numerical dispersion relation as 
\begin{align}
\displaystyle 
\label{Dissemi4}
\frac{1}{2}\,h\, \Lambda_{2M} = \sum_{p=1}^{M} \frac{[(2p-3)!!]^2}{(2p-1)!} 
\sin^{2p-1} \left( \frac{k_{\text{FD},2M} h}{2} \right)
=\pm \frac{1}{2}K,
\end{align}
with $K:=k^{\text{ex}} h$.
In general, for any $M\geq1$, we will have $(4M-2)$ discrete wave numbers $k_{\text{FD},2M}$ that satisfy \eqref{Dissemi4}. In particular, when $M=1$, i.e. for the FD2 scheme,  the numerical dispersion relation \eqref{Dissemi4} is 
\begin{align}
\label{tay1}
\displaystyle
\sin \left( \frac{k_{\text{FD},2} h}{2} \right)
= \pm \frac{1}{2} K.
\end{align}
Thus, considering $K \ll 1$, and performing a Taylor expansion of \eqref{tay1} we obtain
	\begin{align}
	\label{eq:dis_FD_semi1}
	k_{\text{FD},2} =\pm  k^\text{ex} \left(  1+ \frac{1}{24} K^2 + \frac{3}{640} K^4 + \mathcal{O}( K^6) \right), 
	\end{align}
which indicates that the numerical dispersion error of the FD2 scheme is second order accurate in space.

\noindent 

For the case $M=2$,  i.e. for the FD4 scheme, the numerical dispersion relation \eqref{Dissemi4} becomes 
\begin{align}
\displaystyle 
\label{DisExact2}
\frac{1}{6}\sin^3 \left( \frac{k_{\text{FD},4} h}{2} \right)
+
\sin\left( \frac{k_{\text{FD},4} h}{2} \right)
= 
\pm \frac{1}{2} K. 
\end{align}

\noindent The Taylor expansions of all roots in equation \eqref{DisExact2} are given by 
\begin{subequations}
	\begin{align}
	\label{eq:dis_FD_semi2}
	k_{\text{FD}^{\text{ phys}},4}
	&= \pm k^\text{ex} \left( 1 + \frac{3}{640}K^4 - \frac{1}{3584} K^6 + \mathcal{O}(K^8) \right), \\
	\label{eq:dis_FD_semi2s1}
	k_{\text{FD}^{\text{ spur1}},4}
	&= \pm k^\text{ex} \left( i \frac{\arcsinh(2\sqrt{42})}{K} - \frac{1}{2\sqrt{7}} 
	+ i\frac{9}{1568}\sqrt{42}K + \mathcal{O}(K^2) \right),\\
	\label{eq:dis_FD_semi2s2}
	k_{\text{FD}^{\text{ spur2}},4} 
	&= \pm k^\text{ex} \left(- i \frac{\arcsinh(2\sqrt{42})}{K} - \frac{1}{2\sqrt{7}}
	- i\frac{9}{1568}\sqrt{42}K	+ \mathcal{O}(K^2) \right),
	\end{align}
\end{subequations}
\noindent where $k_{\text{FD}^\text{ phys},4}$ and $k_{\text{FD}^{\text{ spur1}},4}, k_{\text{FD}^{\text{ spur2}},4}$ are wave numbers corresponding to the physical modes and spurious modes, respectively, of the FD4 scheme. The physical modes indicate a fourth order accurate numerical dispersion error, while the leading terms in the spurious modes of $k$ are proportional to $\mathcal{O}(1/h)$, indicating an exponential increase or damping corresponding to the opposite sign in front. The existence of spurious, or non-physical modes for a variety of discretizations has been discussed in the literature, e.g., \cite{ainsworth2006dispersive, cohen1, cohen2}. The presence of spurious modes in not ideal, however in practice these have not shown to be serious issues for numerical methods. We would like to note, that to the best of the authors' knowledge, the existence of spurious modes for high order FD discretizations has not been analytically identified in the literature. Equations \eqref{eq:dis_FD_semi2} and \eqref{eq:dis_FD_semi2s2} provide explicit formulas for the spurious modes that we have not found in the literature.

 Below, we focus on  the physical modes, and prove that for the FD scheme of order $2M$ (FD2M), the dispersion error is of $2M$-th order. Thus, the dispersion error is of the same order as the local truncation error for the finite difference schemes.

\begin{theorem}\label{thm5}
	The physical modes of the dispersion relation \eqref{Dissemi4}, for the spatial semi-discrete finite difference method FD2M, results in the dispersion error identity
	\begin{align} \label{q1}
	\displaystyle
	k_{\text{FD}^{\text{ phys}},2M} = \pm k^{\text{ex}} \left( 1 + \varsigma_{_{2M}} \right),
	\end{align}
	\noindent for any $M\geq1$, where 
	\begin{align} \label{q0}
	\displaystyle
	\varsigma_{_{2M}}
	= \frac{[(2M-1)!!]^2}{2^{2M}(2M+1)!} K^{2M}
	+ \mathcal{O}(K^{2M+2}).
	\end{align}

\noindent In other words, the dispersion error of the FD2M scheme \eqref{eq:semi} is  of order $2M$. \\	
\end{theorem}

\begin{proof}
  Here, we only consider $k_{\text{FD}^{\text{ phys}},2M}$ with plus sign in front. Define $\varsigma_{_{2M}} := \displaystyle \frac{k_{\text{FD}^{\text{ phys}},2M}-k^{\text{ex}}}{k^{\text{ex}}}$. Then, substituting from \eqref{Dissemi3} for $k^{\text{ex}}$, rearranging, and (using results from \cite{bokil2018high}) we obtain the identity
	\begin{align*}
	\displaystyle 
	k^{\text{ex}} \varsigma_{_{2M}}
	&= k_{\text{FD}^{\text{ phys}},2M}
	- \frac{2}{h}\sum_{p=1}^{M} \frac{\lambda_{2p-1}^{2M}}{(2p-1)}
	\sin\left[\left(p - \frac{1}{2}\right) k_{\text{FD}^{\text{ phys}},2M} h\right], 
	\end{align*}
i.e, $k_{\text{FD}^{\text{ phys}},2M}$ satisfies \eqref{q1} for $\varsigma_{_{2M}}$ as defined in \eqref{q0}. Next, we prove the identity \eqref{q0}. Because we only consider the physical modes here, it is reasonable to assume that $k_{\text{FD}^{\text{ phys}},2M} = k^{\text{ex}}\left(1 + O(K^{\tau}) \right)$ for some $\tau>0$. Hence, when  $k_{\text{FD}^{\text{ phys}},2M}h=  K \left(1 + \mathcal{O}(K^{\tau})\right)$  is small enough, performing a Taylor expansion with respect to $k_{\text{FD}^{\text{ phys}},2M} h$ we get
	\begin{align*}
	\displaystyle 
	k^{\text{ex}}\varsigma_{_{2M}}
	&= k_{\text{FD}^{\text{ phys}},2M}
	- \frac{2}{h}\sum_{p=1}^{M} \frac{\lambda_{2p-1}^{2M}}{(2p-1)} 
	\sum_{\ell = 0}^{\infty} \frac{(-1)^{\ell}}{(2\ell + 1)!}
	\left[\frac{1}{2} \left(2p - 1\right) k_{\text{FD}^{\text{ phys}},2M} h\right]^{2\ell + 1} \notag \\
	&= k_{\text{FD}^{\text{ phys}},2M}
	- \frac{2}{h}\sum_{\ell = 0}^{\infty}\sum_{p=1}^{M} \frac{\lambda_{2p-1}^{2M}}{(2p-1)} 
	\frac{(-1)^{\ell}}{(2\ell + 1)!}
	\left[\frac{1}{2} \left(2p -1 \right) k_{\text{FD}^{\text{ phys}},2M} h\right]^{2\ell + 1} \notag \\
	&= k_{\text{FD}^{\text{ phys}},2M}
	- k_{\text{FD}^{\text{ phys}},2M} \sum_{\ell = 0}^{\infty}
	\left[\sum_{p=1}^{M} \lambda_{2p-1}^{2M}(2p-1)^{2\ell}\right]
	\frac{(-1)^{\ell}}{2^{2\ell}(2\ell + 1)!}
	\left( k_{\text{FD}^{\text{ phys}},2M} h \right)^{2\ell}.
	\end{align*}

	\noindent Based on the derivation of $\displaystyle \lambda_{2p-1}^{2M}$ as discussed in \cite{bokil2018high}, we have the following identities
	\begin{align*}
	\displaystyle
	& \sum_{p=1}^M \lambda_{2p-1}^{2M}  = 1, \\
	& \sum_{p=1}^M \lambda_{2p-1}^{2M}(2p-1)^{2\ell}  = 0,
	\quad \text{ for } \ell = 1,2,..,M-1, \\
	& \sum_{p=1}^M \lambda_{2p-1}^{2M}(2p-1)^{2M} = (-1)^{M+1} \left[(2M-1)!!\right]^2.
	\end{align*}

	\noindent Therefore, 
	\begin{align}
	\displaystyle
	\label{eq:qq4} 
	k^{\text{ex}}\varsigma_{_{2M}}
	&= k_{\text{FD}^{\text{ phys}},2M}-k_{\text{FD}^{\text{ phys}},2M}
	\left[ 1
	-\frac{1}{2^{2M}} \frac{[(2M-1)!!]^2}{(2M+1)!}\left( k_{\text{FD}^{\text{ phys}},2M} h \right)^{2M}
	+ \mathcal{O}\left(\left( k_{\text{FD}^{\text{ phys}},2M} h \right)^{2M+2}\right)
	\right] \notag \\
	&= -k_{\text{FD}^{\text{ phys}},2M}
	\left[
	-\frac{1}{2^{2M}} \frac{[(2M-1)!!]^2}{(2M+1)!}\left( K \right)^{2M}  
	+ \mathcal{O}\left(K^{2M+\tau}+K^{2M+2}\right)
	\right] \notag \\
	&= -k^{\text{ex}}
	\left[
	-\frac{1}{2^{2M}} \frac{[(2M-1)!!]^2}{(2M+1)!}\left( K \right)^{2M}  
	+ \mathcal{O}\left(K^{2M+\tau}+K^{2M+2}\right)
	\right],
	\end{align}

	\noindent which proves \eqref{q0}. Hence, with the assumption  $k_{\text{FD}^{\text{ phys}},2M} = k^{\text{ex}}\left(1 + \mathcal{O}(K^{\tau})\right)$  and $\tau>0$, we can deduce that  $k_{\text{FD}^{\text{ phys}},2M}= k^{\text{ex}}\left(1 + \mathcal{O}(K^{2M}) \right)$. 
\end{proof}

Next, we illustrate the relative phase errors of the FD2M scheme for \eqref{eq:semi}, $M=1,\ldots,5$, with the parameter set $\mathbf{p}$ fixed at values given in \eqref{eq:parameter}. The numerical wave number $k_{\text{FD},2M}$ is obtained by solving \eqref{Dissemi4} exactly or with the help of a Newton solver (we set the tolerance at $10^{-18}$). Since 
$$k^{\text{ex}}h = \WH{\omega}\, \sqrt{\epsilon(\WH{\omega};\mathbf{p})}\, \omega_{1}h,$$
then $k_{\text{FD},2M}h$ depends on $\WH{\omega}$, $\mathbf{p}$, $\omega_{1}h$, and $M$,  and so does the relative phase error
\begin{align}
\label{ERRTP2}
\displaystyle 
\Psi_{\text{FD},2M}(\WH{\omega})
:=
\left| \frac{k_{\text{FD},2M}(\WH{\omega})-k^{\text{ex}}(\WH{\omega})}{k^{\text{ex}}(\WH{\omega})}  \right|
= \left| \frac{k_{\text{FD},2M}(\WH{\omega})h -k^{\text{ex}}(\WH{\omega})h}{k^{\text{ex}}(\WH{\omega})h}\right|. 
\end{align}
First, we fix $\omega_{1} h=\pi/30$, and present the relative phase errors  as functions of $\WH{\omega}\in[0,3]$ in Figure \ref{Fig:FD}. 
Because the leading error term in the numerical wave number for the FD2M scheme is proportional to $K^{2M}$, we expect $2M$ order accuracy of the relative phase error with respect to $K$ at a fixed angular frequency. We observe that all schemes have significantly larger error around $\WH{\omega}=1$, while the error fades out near $\WH{\omega}=\sqrt{\epsilon_{s}/\epsilon_{\infty}}$, where $K$ is close to zero. As expected from analysis, higher order spatial accuracy does result in reduced relative phase errors.  We present the relative phase  errors at $\WH{\omega}=1$ with $\omega_1 h =\pi/30$  in the left plot in Figure \ref{Fig:FD}, while in the right plot we depict the $2M$ order convergence of relative phase errors with respect to $K$ for fixed $\WH{\omega}=1$. The slopes of phase errors in this plot are shown to be the same as those of reference lines with slope $2M$, indicating the $2M$th order of accuracy for each FD2M scheme, which agrees with the results in Theorem \ref{thm5}.   We note the presence of just one peak in these plots as compared to the presence of two peaks in analogous plots of phase errors for temporal discretizations presented in Section \ref{time}. 

\begin{figure}[htb]
	\centering
	\includegraphics[scale= 0.3] {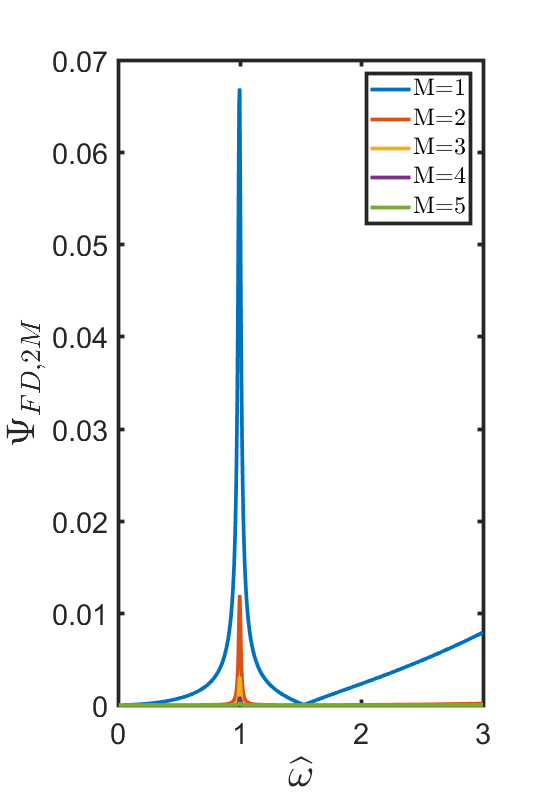}
	\includegraphics[scale= 0.3] {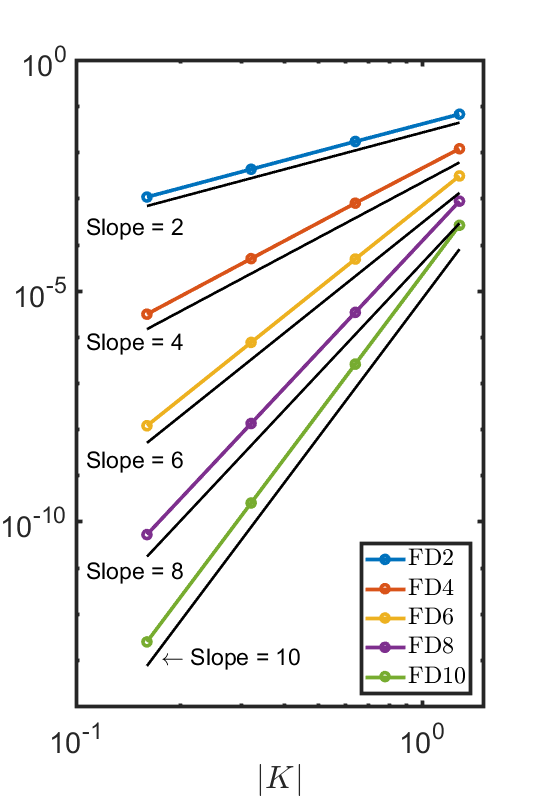}
	\caption{The relative phase error of  physical modes for the spatial discretization FD2M. Left: fix $\omega_{1}h=\pi/30$ with $\WH{\omega}\in[0,3]$; right: fix $\WH{\omega}=1$ with different $\omega_{1}h \in\left\{ \pi/30,\pi/60,\pi/120,\pi/240\right\} $.}
	\label{Fig:FD}
\end{figure}
\section{Fully discrete FDTD Methods} $\,$ 
\label{fdtd}
In this section, we consider the high order staggered spatial discretizatons \eqref{eq:semi} combined with either the leap-frog scheme in time \eqref{eq:LF} or the trapezoidal scheme in time \eqref{eq:TP} presented in Section \ref{time}. These fully discrete methods are second order accurate in time and $2M$-th order accurate in space, thus we denote them as $(2,2M)$ leap-frog FDTD schemes or $(2,2M)$ trapezoidal FDTD methods. In particular, the $(2,2)$ leap-frog method is the extension of the standard Yee FDTD method to Lorentz dispersive media. Finally, comparisons will be made among all finite difference schemes under considerations.

We first compute the dispersion relation for the fully discrete ($2,2M$) schemes. To do so, we assume the plane wave solutions
\begin{align} 
\label{disp} \displaystyle
X_j^n \equiv
X_0 e^{i\left( k_{\text{FD},2M}^{*} j h - \omega n \Delta t \right)},
\end{align}
where $*$ is either LF or TP. Substituting \eqref{disp} into the appropriate $(2,2M)$ FDTD method (which we have not explicitly written out here for brevity), 
we obtain the linear system 
\begin{align}
\label{DisLF2}
\mathcal{A}^{\text{*}}_{\text{FD},2M}\textbf{U}^{\text{*}}_{\text{FD}} = \textbf{0}, 
\end{align}
where the coefficient matrix for the two schemes will be discussed in the next two sections.

\subsection{ Fully discrete dispersion analysis: $(2,2M)$ leap-frog-FDTD schemes} 
We first consider the $(2,2M)$ leap-frog FDTD scheme. For the leap-frog temporal discretization the coefficient matrix in the linear system \eqref{DisLF2} is given as
\renewcommand{\arraystretch}{1.5}
\begin{align}
\mathcal{A}^{\text{LF}}_{\text{FD},2M} = \begin{pmatrix}
	\sin \left( \frac{W}{2} \right)  & \Lambda^{\text{LF}}_{2M} & 0 & 0\\
	\Lambda^{\text{LF}}_{2M} & \epsilon_\infty \sin \left( \frac{W}{2} \right) & 
	\sin \left( \frac{W}{2} \right) & 0 \\
	0 & 0 & i\sin \left( \frac{W}{2} \right) & 
	\frac{\Delta t}{2} \cos \left( \frac{W}{2} \right) \\
	0 & \frac{\Delta t}{2}\omega_p^2 \cos \left( \frac{W}{2} \right) &
	- \frac{\Delta t}{2}\omega_1^2 
	\cos \left( \frac{W}{2} \right)  & 
	i\sin \left( \frac{W}{2} \right) -\gamma\Delta t \cos \left( \frac{W}{2} \right) \\
\end{pmatrix},
\end{align}
\renewcommand{\arraystretch}{1}
with $$\Lambda^{\text{LF}}_{2M} = \frac{\Delta t}{h} \sum_{p=1}^{M} \frac{\lambda_{2p-1}^{2M}}{(2p-1)} 
\sin \left[ \left( p-\frac{1}{2} \right) k_{\text{FD},2M}^{\text{LF}} h \right]. $$
Based on previous discussions, we can derive the identity 
\begin{align}
\displaystyle 
\label{DisLF7}
\sum_{p=1}^{M} \frac{[(2p-3)!!]^2}{(2p-1)!} 
\sin^{2p-1} \left( \frac{k_{\text{FD},2M}^{\text{LF}} h}{2} \right)
 =  \frac{1}{2} k^{\text{LF}} h.
\end{align}

 For both the fully discrete methods, we focus our discussions on the physical modes. For the case of  $W\ll 1$ and $K\ll 1$, we analyze the Taylor expansion of the physical modes  for the fully discrete leap-frog FDTD schemes and observe the following pattern:

\begin{align}
\displaystyle 
k^{\text{LF}}_{\text{FD}^{\text{ phys}},2M}
&=  \pm k^{\text{ex}}\left(
1+ \frac{1}{12}\left( \frac{\delta(\WH{\omega};\textbf{p})}{\epsilon(\WH{\omega};\textbf{p})} - \frac{1}{2} \right)W^2+ \frac{[(2M-1)!!]^2}{2^{2M}(2M+1)!}K^{2M} 
+  \mathcal{O}(K^{2M+2} + K^{2M} W^2+W^4)  
\right), \quad M \geq 1. 	\label{eq:disp_FD_LF1}
\end{align}	
Furthermore, with the relation $\displaystyle K=\frac{\sqrt{\epsilon(\WH{\omega};\mathbf{p})}}{\nu\sqrt{\epsilon_{\infty}}}W$, we can treat $k_{\text{FD}^{\text{ phys}},2M}^{\text{LF}}$ as a function of $W$ and  $\ds \nu=\frac{\Delta t}{h\sqrt{\epsilon_{\infty}}}$, the CFL  (Courant-Friedrich-Lewy) number subject to the stability constraint for the $(2,2M)$ leap-frog FDTD scheme.  
Assuming $\nu=O(1),$ we have
\begin{align}
\displaystyle k^{\text{LF}}_{\text{FD}^{\text{ phys}},2M}
= 
\begin{cases}
\pm \displaystyle k^{\text{ex}}
\left(
1 + 
\frac{1}{12}\left( \frac{\delta(\WH{\omega};\textbf{p})}{\epsilon(\WH{\omega};\textbf{p})} - \frac{1}{2} + \frac{\epsilon(\WH{\omega};\textbf{p})}{2\epsilon_\infty\nu^2}\right)W^2
+ \mathcal{O}(W^4)
\right),  & M = 1, \\ \\
\displaystyle
\pm k^{\text{ex}}
\Big(
1 + 
\frac{1}{12}\left( \frac{\delta(\WH{\omega};\textbf{p})}{\epsilon(\WH{\omega};\textbf{p})} - \frac{1}{2} \right)W^2
+  \mathcal{O}(W^{4})  
\Big),  & M \geq 2. 
\end{cases}    \label{eq:disp_FD_LF2}
\end{align}

We can see that due to the second order time discretizations employed,
 the fully discrete scheme always results in a second order dispersion error. Particularly for all $M\ge 2,$ the leading term in the dispersion error is identical,   and independent of $\nu$ which comes solely from the temporal discretization. To compare the performance of the scheme for $M=1$ and $M\geq2$, we will focus on comparison of the coefficients of leading error terms in \eqref{eq:disp_FD_LF2}.

We first consider $\WHG=0$. We can make the following conclusions. 
\begin{itemize}
\item For $\WHO$ in the medium absorption band, i.e. $\WHO \in (1, \sqrt{\epsilon_{s}/\epsilon_{\infty}}),$ it is easy to check that 
$$\displaystyle 
\left| \frac{\delta(\WHO;\mathbf{p})}{\epsilon(\WHO;\mathbf{p})} -\frac{1}{2} + \frac{\epsilon(\WHO;\mathbf{p})}{2 \epsilon_{\infty}\nu^2} \right|
 \geq 
 \left| \frac{\delta(\WHO;\mathbf{p})}{\epsilon(\WHO;\mathbf{p})} -\frac{1}{2} \right|$$ based on the inequalities 
 $\displaystyle\frac{\delta(\WHO;\mathbf{p})}{\epsilon(\WHO;\mathbf{p})}\leq-1$ and $\displaystyle\frac{\epsilon(\WHO;\mathbf{p})}{2 \epsilon_{\infty}\nu^2}\leq0,$ which implies that  the high order schemes   reduce the dispersion error as one would expect. This is true independent of other parameter choices.
 \item For other $\WHO$ values, the outcome will depend on the parameters. We can show that the general condition for the higher order scheme ($M\ge2$) to be more accurate in its dispersion error is equivalent to the inequality
\begin{align}
\displaystyle\left(\frac{\epsilon_d}{\epsilon_\infty}\right)^2
	+ (1-2\nu^2)(\WH{\omega}^2 - 1)^2
	- 2 \left(\frac{\epsilon_d}{\epsilon_\infty}\right)\left(\WH{\omega}^2-1 + \nu^2(1-3\WH{\omega}^2)\right) \geq 0. \label{Sp333}
\end{align}
This is a quadratic inequality in $\WHO^2.$ We can conclude that with the CFL condition $\nu\leq1,$ which is a necessary condition to ensure the fully discrete $(2,2M)$ leap-frog-FDTD scheme is stable for any $M\geq1$ \cite{bokil2012,bokil2018high}, we have
\begin{itemize}
	\item if $0< \nu \leq \frac{1}{\sqrt{2}}$, the condition \eqref{Sp333} always holds for all $\WHO \geq 0$. 
	\item if $\frac{1}{\sqrt{2}}< \nu < 1$ and 
	\begin{itemize}
		\item if $0<\epsilon_{d}/\epsilon_{\infty}\leq 2\nu^2-1$, then the condition \eqref{Sp333} holds on
		\[\displaystyle
		\WHO_{L} \leq \WHO \leq \WHO_{R},
		\]
		\item if $\epsilon_{d}/\epsilon_{\infty} \geq 2\nu^2-1$, then the condition \eqref{Sp333} holds on
		\[\displaystyle
		0 \leq \WHO \leq \WHO_{R},
		\]
		 \end{itemize}
	 where
	  \end{itemize}	 
  \begin{align*}
  \WH{\omega}_{L} &= \sqrt{ \frac{-1 - \epsilon_{d}/\epsilon_{\infty} + 2 \nu^2 + 3 \epsilon_{d}/\epsilon_{\infty} \nu^2-\nu \sqrt{ -4 \epsilon_{d}/\epsilon_{\infty} - 4 (\epsilon_{d}/\epsilon_{\infty})^2 + 8 \epsilon_{d}/\epsilon_{\infty} \nu^2 + 9 (\epsilon_{d}\nu/\epsilon_{\infty})^2 }}{ 2 \nu^2 -1 } }, \\
\WH{\omega}_{R} &= \sqrt{ \frac{-1 - \epsilon_{d}/\epsilon_{\infty} + 2 \nu^2 + 3 \epsilon_{d}/\epsilon_{\infty} \nu^2+ \nu \sqrt{-4 \epsilon_{d}/\epsilon_{\infty}  - 4 (\epsilon_{d}/\epsilon_{\infty})^2 + 8 \epsilon_{d}/\epsilon_{\infty} \nu^2 + 9 (\epsilon_{d}\nu/\epsilon_{\infty})^2 }}{2 \nu^2 -1} }.
\end{align*}

 \end{itemize}

 The case of $\WHG>0$ is even more complicated. For low loss materials, in general we expect similar conclusions as in the lossless case.
We now perform a numerical study, and  compare the leading error terms in \eqref{eq:disp_FD_LF2} with $\nu=0.6$ (which is small enough to guarantee that the scheme is stable for arbitrary $M$ (see next section and \cite{bokil2018high}). The absolute values of coefficients of leading error terms are plotted in Figure \ref{Fig:coeff}, with $\epsilon_{\infty}=2.25$, $\epsilon_{s}=5.25$ and various $\WHG$ values. 
It is observed that we can not determine which method performs better for the general case. 
From   Figure \ref{Fig:coeff}, it's clear that higher order schemes  have smaller dispersion error for $\WHG=0, 0.01$ in the range $\WHO \in[0,3].$ This is no longer true for $\WHG=0.1, 1.$ The   discussion here reveals an interesting fact. For some parameter values, we can have counterintuitive results that the lower order scheme performs better than higher order scheme when numerical dispersion is present.

\begin{figure}[h]
	\centering
	\includegraphics[scale= 0.28] {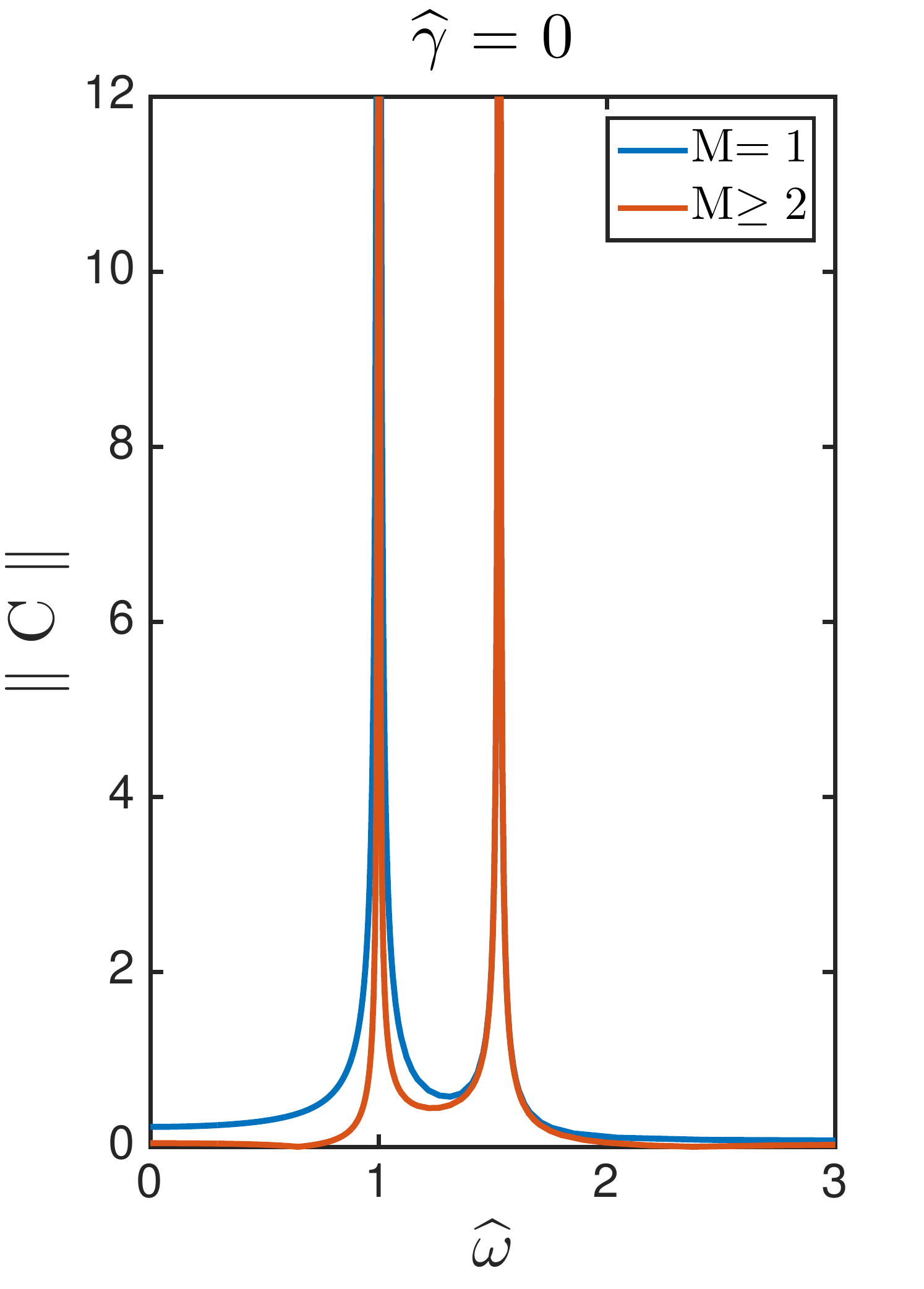}
	\includegraphics[scale= 0.28] {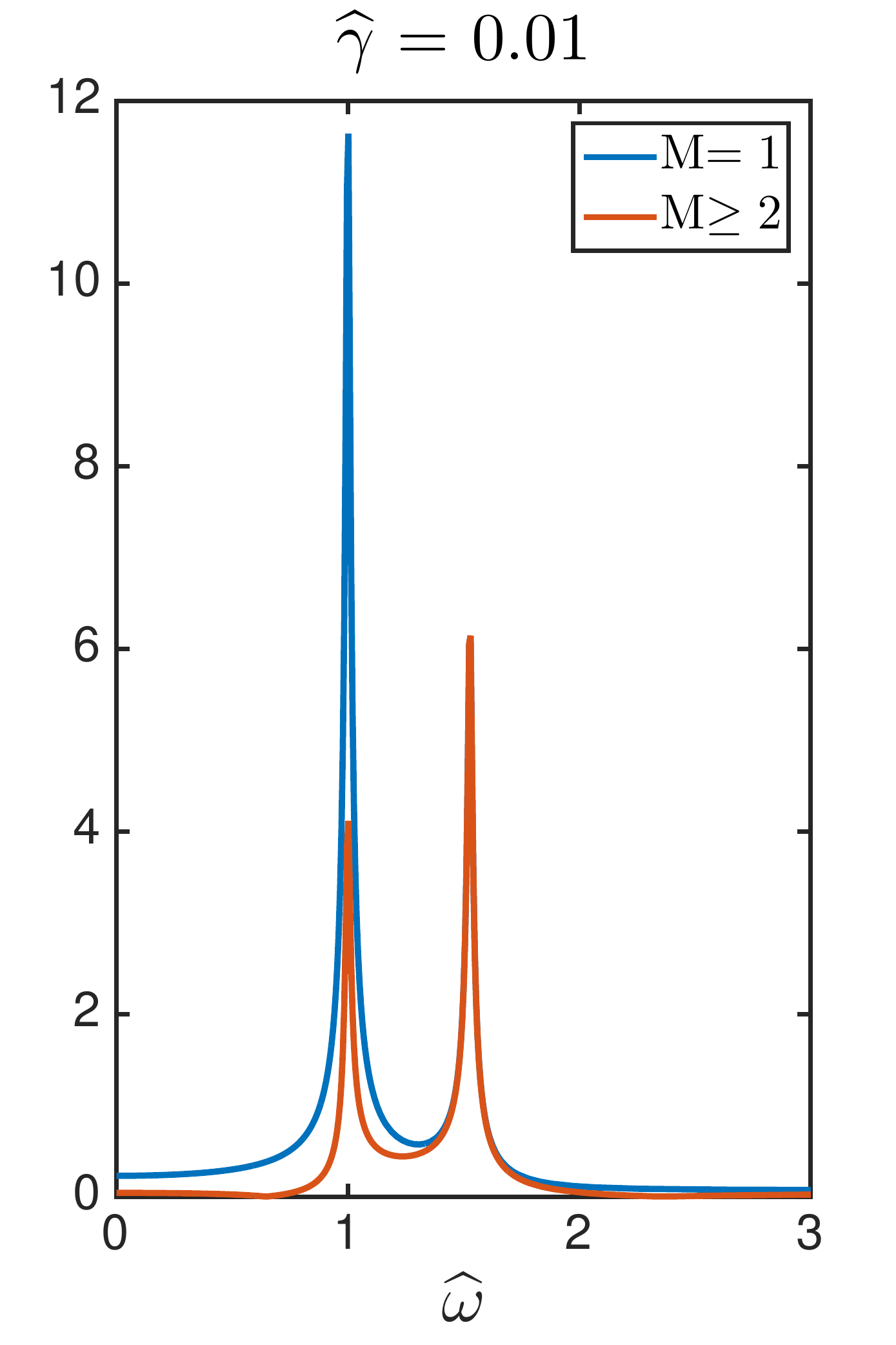}
	\includegraphics[scale= 0.28] {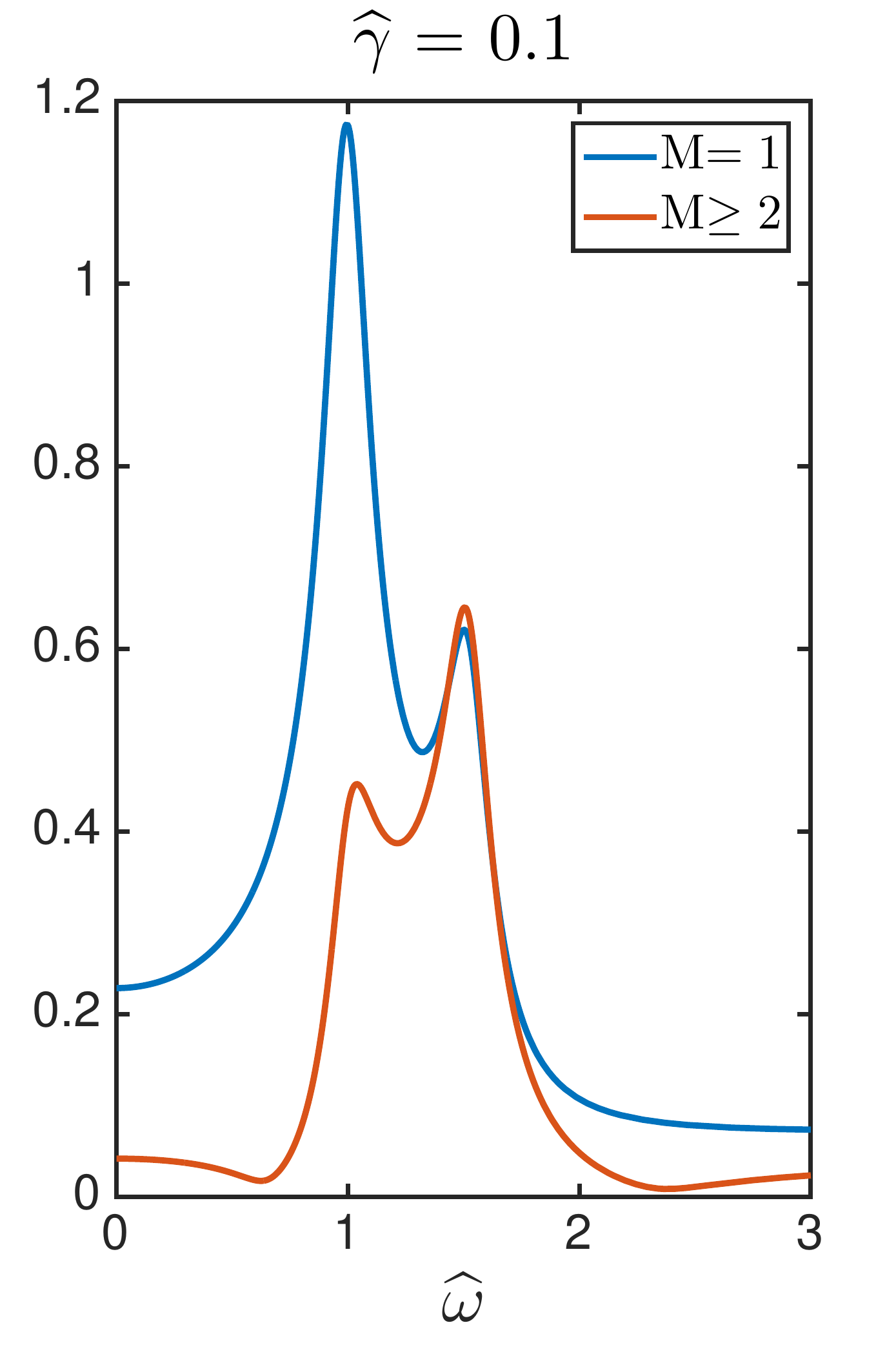}
	\includegraphics[scale= 0.28] {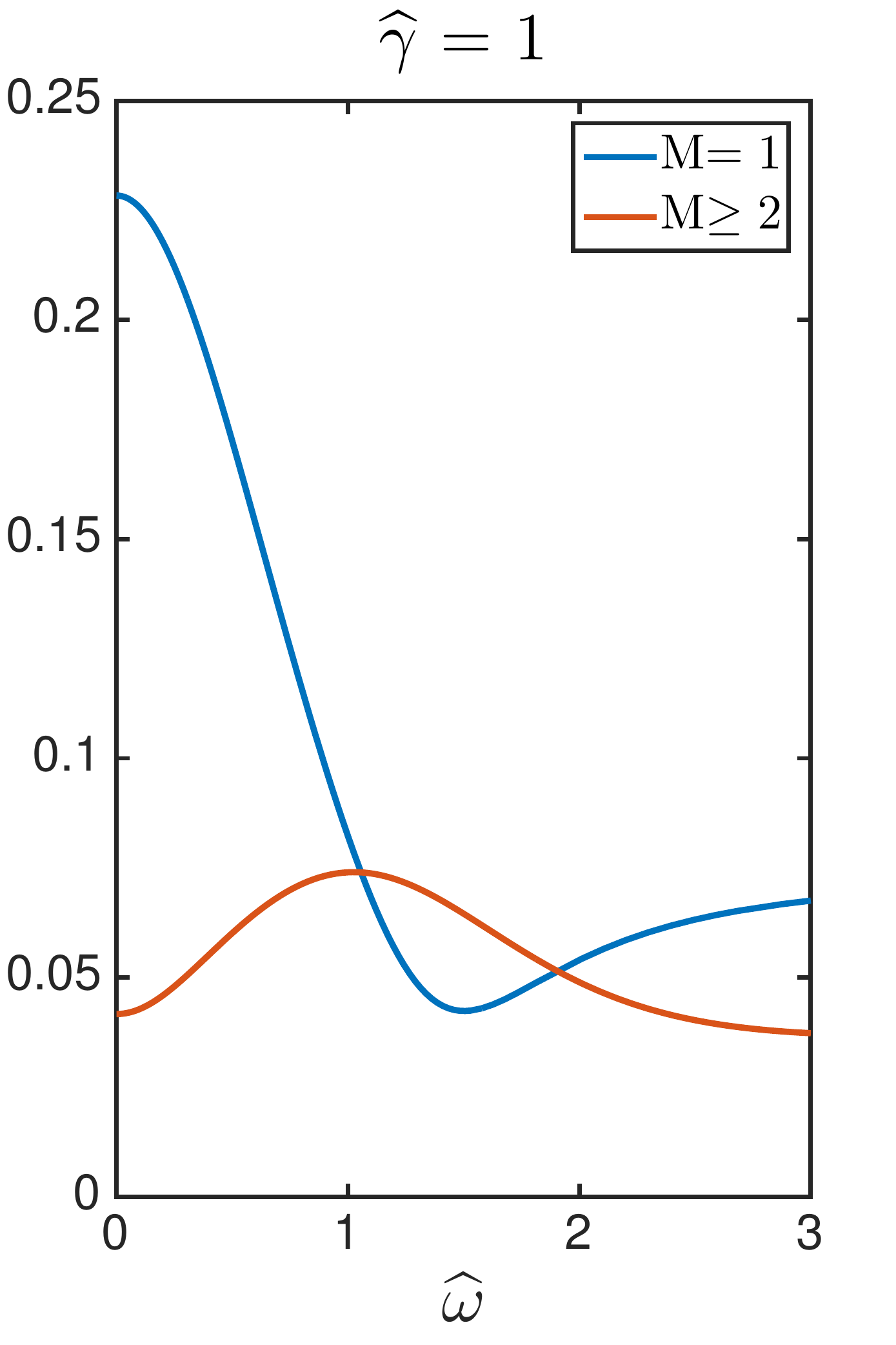}
	\caption{Absolute value of coefficients of leading error terms in \eqref{eq:disp_FD_LF2} (denoted by $C$) for the $(2,2M)$  leap-frog FDTD scheme.  }
	\label{Fig:coeff}
\end{figure}


\subsection{ Fully discrete dispersion analysis: $(2,2M)$ trapezoidal-FDTD schemes} 

We repeat the analysis done in the previous section for the fully discrete $(2,2M)$ trapezoidal FDTD schemes. 
We obtain the numerical dispersion relation for these schemes by setting the determinant of the matrix 
\begin{align}
\renewcommand{\arraystretch}{1.5}
\mathcal{A}_{\text{FD},2M}^{\text{TP}}=\begin{pmatrix}
	\sin \left( \frac{W}{2} \right) & \Lambda_{2M}^{\text{TP}} & 0 & 0 \\
	\Lambda_{2M}^{\text{TP}}  & \epsilon_\infty \sin \left( \frac{W}{2} \right) & 
	\sin \left( \frac{W}{2} \right) & 0 \\
	0 & 0 & i\sin \left( \frac{W}{2} \right)  
	& \frac{\Delta t}{2} \cos \left( \frac{W}{2} \right) \\
	0 & \frac{\Delta t}{2}\omega_p^2 \cos \left( \frac{W}{2} \right) & 
	- \frac{\Delta t}{2}\omega_1^2 \cos \left( \frac{W}{2} \right) & 
	i\sin \left( \frac{W}{2} \right) -\gamma \Delta t \cos \left( \frac{W}{2} \right) \\
\end{pmatrix}
\renewcommand{\arraystretch}{1}
\end{align}
to zero. In the above, we have
$$\ds \Lambda_{2M}^{\text{TP}}=\frac{\Delta t}{h} \sum_{p=1}^{M} \frac{\lambda_{2p-1}^{2M}}{(2p-1)} 
\sin \left[ \left( p-\frac{1}{2} \right) k_{\text{FD},2M}^{\text{TP}} h \right]
\cos \left( \frac{W}{2} \right).$$ 
The numerical dispersion is given by 
\begin{align}
\displaystyle 
\label{DisTP7}
\sum_{p=1}^{M} \frac{[(2p-3)!!]^2}{(2p-1)!} 
\sin^{2p-1} \left( \frac{k_{\text{FD},2M}^{\text{TP}} h}{2} \right)
= \frac{1}{2} k^\text{TP} h.
\end{align}

\noindent By requiring $W\ll 1$ and $K\ll 1$, we can obtain the physical modes   in the form  

\begin{align}
	\displaystyle	k^{\text{TP}}_{\text{FD}^{\text{ phys}},2M}
	&=  \pm k^{\text{ex}}\left(
	1+ \frac{1}{12}\left( \frac{\delta(\WH{\omega};\textbf{p})}{\epsilon(\WH{\omega};\textbf{p})} + 1 \right)W^2 + \frac{[(2M-1)!!]^2}{2^{2M}(2M+1)!}K^{2M} 
	+  \mathcal{O}(K^{2M+2} + K^{2M} W^2+W^4)  
	\right), \quad M \geq 1.    \label{eq:disp_FD_TP1}
	\end{align}
For $W\ll 1$ with   $\nu =\mathcal{O}(1)$, Taylor expansion gives us 
\begin{align}
\renewcommand{\arraystretch}{2}
\displaystyle k^{\text{TP}}_{\text{FD}^{\text{ phys}},2M}
= \left\{ \begin{array}{ll}
\displaystyle  \pm k^{\text{ex}}
\left(
1 + 
\frac{1}{12}\left( \frac{\delta(\WH{\omega};\textbf{p})}{\epsilon(\WH{\omega};\textbf{p})} +1 + \frac{\epsilon(\WH{\omega};\textbf{p})}{2\epsilon_\infty \nu^2}\right)W^2
+ \mathcal{O}(W^4)
\right),  & M = 1, \\ 
\displaystyle
\pm k^{\text{ex}}
\left(
1 + 
\frac{1}{12}\left( \frac{\delta(\WH{\omega};\textbf{p})}{\epsilon(\WH{\omega};\textbf{p})} +1 \right)W^2
+  \mathcal{O}(W^{4})  
\right),  & M \geq 2. 
\end{array}
\right. 
 \label{eq:disp_FD_TP2}
\renewcommand{\arraystretch}{1}
\end{align}

This shows  second order dispersion error in all cases. When $\WHG=0$, 
it is easy to check that 
$\ds \left| \frac{\delta(\WHO;\mathbf{p})}{\epsilon(\WHO;\mathbf{p})} +1 + \frac{\epsilon(\WHO;\mathbf{p})}{2\epsilon_{\infty}\nu^2} \right| \geq \left| \frac{\delta(\WHO;\mathbf{p})}{\epsilon(\WHO;\mathbf{p})} +1 \right|$ for any $\WHO\geq0$ and $\nu>0$. Hence, the high order FDTD schemes with $M \geq 2$ always have smaller dispersion error than the $(2,2)$ FDTD scheme.  
On the other hand, numerical tests comparing the coefficients of leading order error terms in the trapezoidal FDTD schemes are provided in Figure \ref{Fig:coeff2}, with various $\WHG$ values and the same parameters as used in Figure \ref{Fig:coeff}. The plots indicate that it is again difficult to determine which coefficient (for $M=1$ or $M\geq 2$) is larger when $\WHG$ is large.

\begin{figure}[h]
	\centering
	\includegraphics[scale= 0.28] {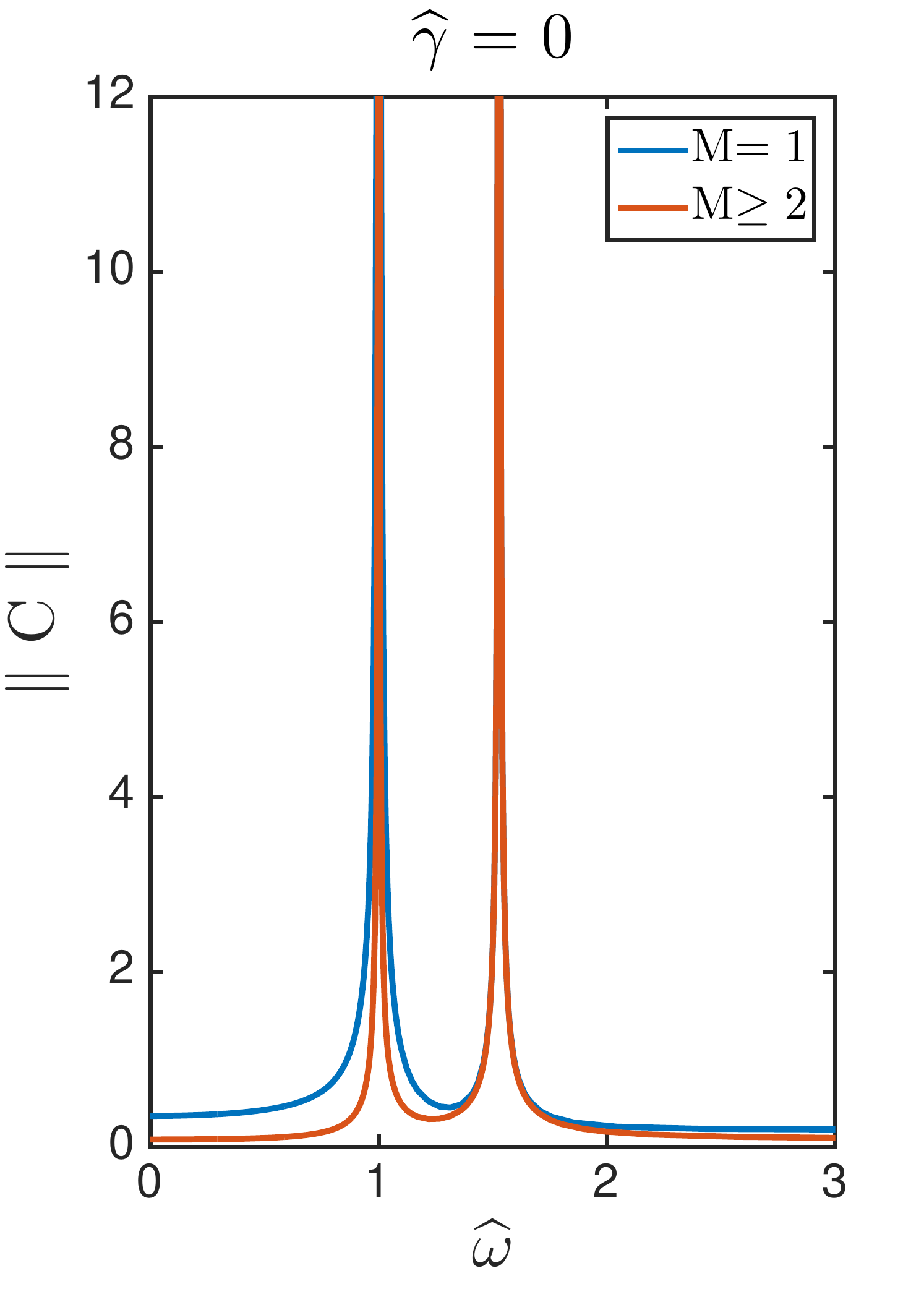}
	\includegraphics[scale= 0.28] {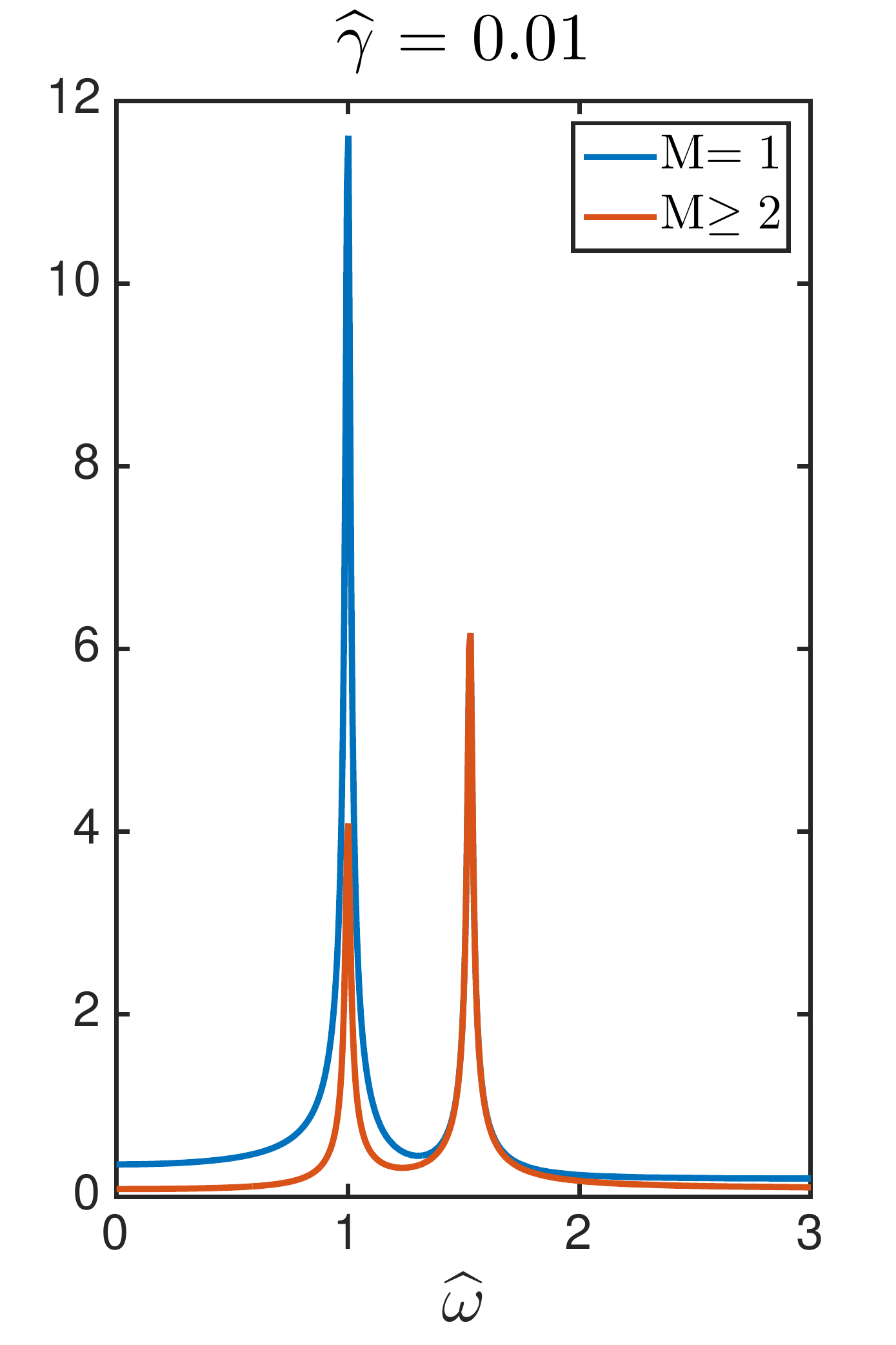}
	\includegraphics[scale= 0.28] {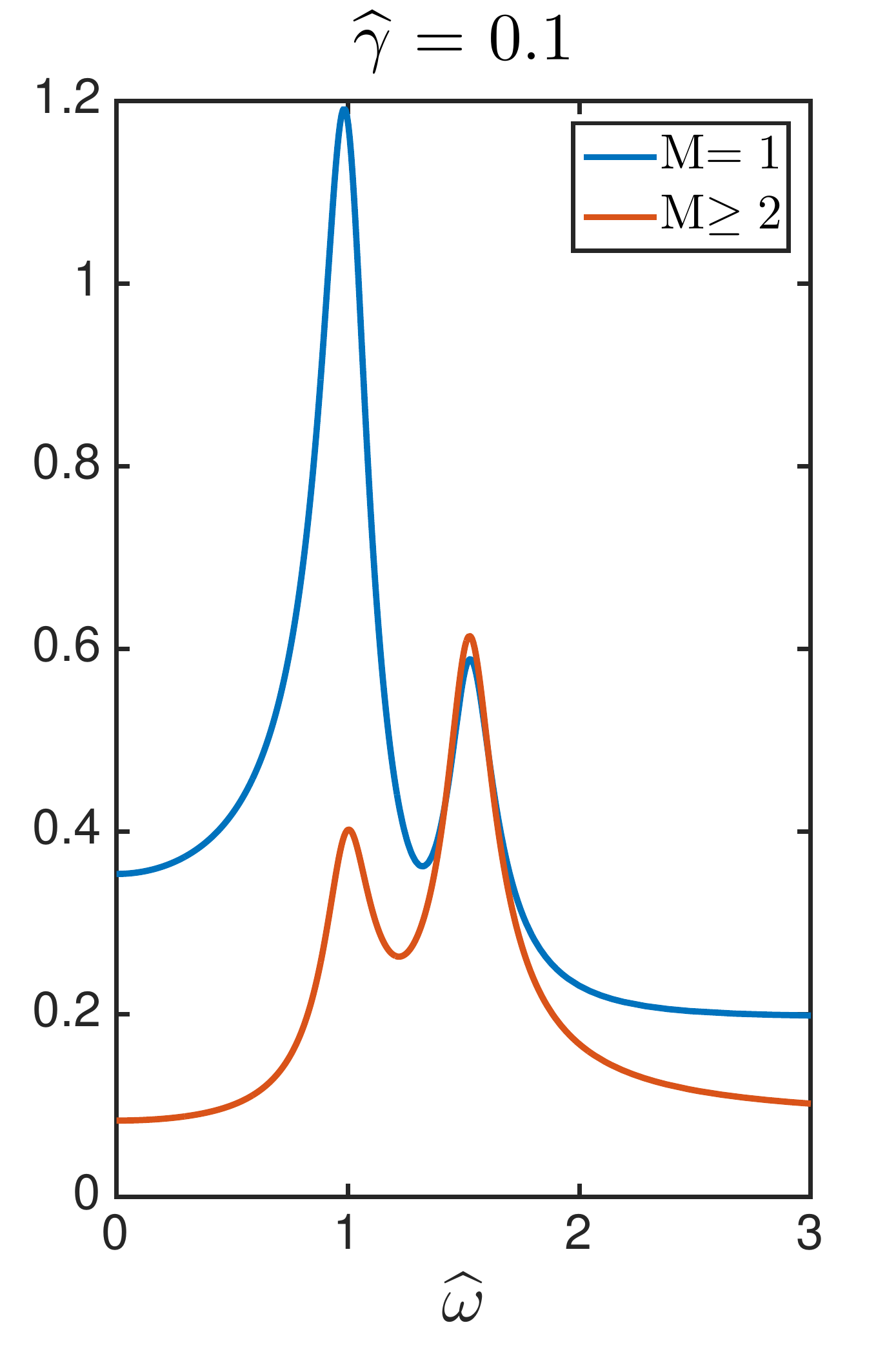}
	\includegraphics[scale= 0.28] {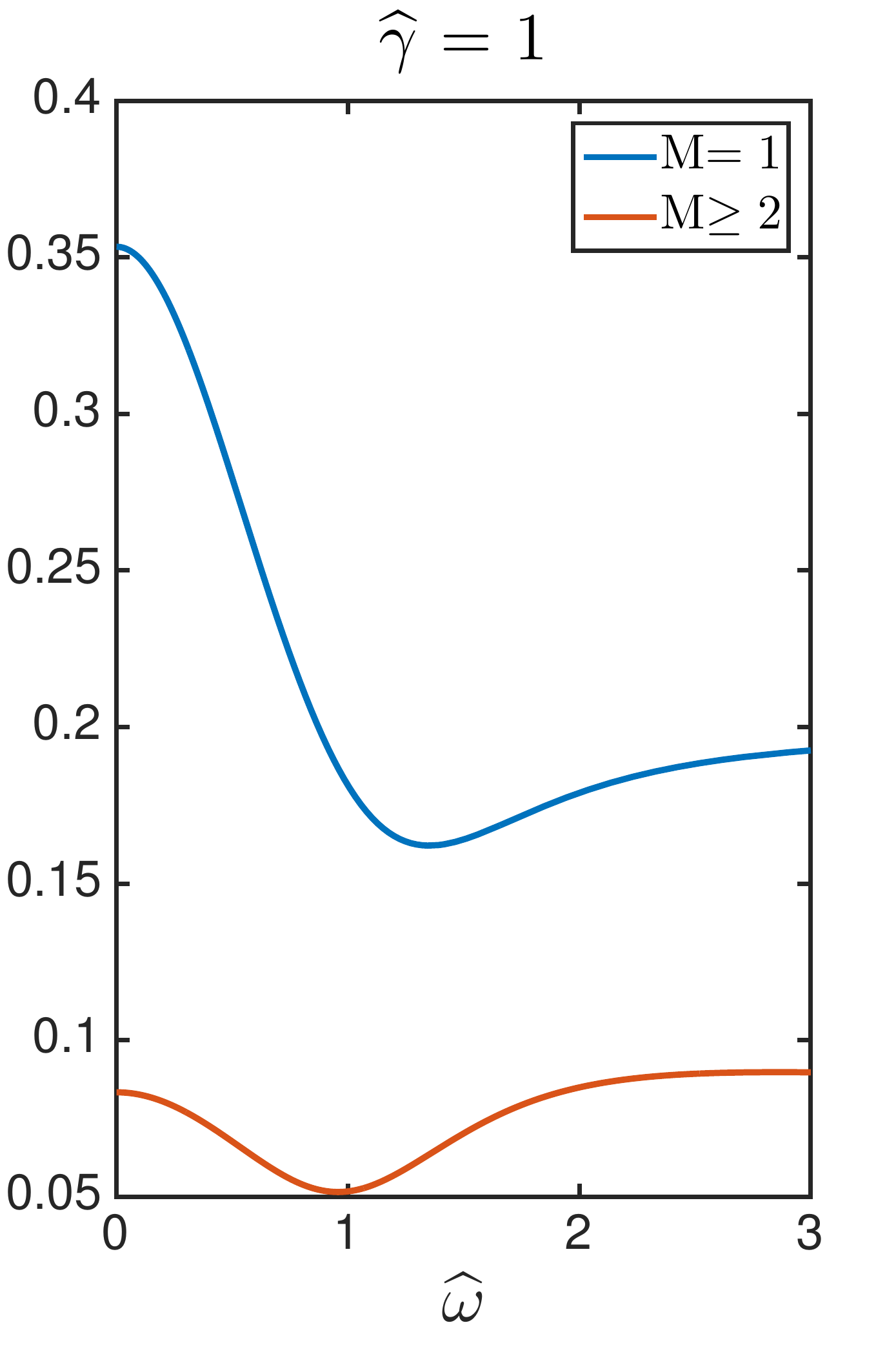}
	\caption{Absolute value of coefficients of leading error terms in \eqref{eq:disp_FD_TP2} (denoted by $C$) for the $(2,2M)$ trapezoidal FDTD scheme.  }
	\label{Fig:coeff2}
\end{figure}

\subsection{Comparison among fully discrete FDTD schemes}

Here, we will present comparisons of the relative phase error  for both the leap-frog and trapezoidal FDTD schemes using the parameters values fixed as in \eqref{eq:parameter}. For the fully discrete schemes, $\omega_{1}\Delta t$ and $\omega_{1}h$ are needed to determine $k^{*}_{\text{FD},2M}$. As shown in \cite{bokil2017energy, bokil2018high}, the schemes based on the trapezoidal rule are unconditionally stable, while the leap-frog schemes are conditionally stable, with the stability condition  as $\nu \leq \nu^{2M}_{max}$, with $\nu_{max}^{2M}$ defined as the largest CFL number of the $(2,2M)$ leap-frog-FD scheme, given by the formula \cite{bokil2018high, bokil2012}
\begin{align}
\label{eq:CFL_max}
\displaystyle 
\nu_{max}^{2M} = \frac{1}{\displaystyle \sum_{p=1}^{M} \frac{[(2p-3)!!]^2}{(2p-1)!}}.
\end{align}
We note that as $M$ increases, $\nu_{max}^{2M}$ decreases but is bounded from below by $\nu_{max}^{\infty}=2/\pi$, i.e. in the limiting case ($M\rightarrow\infty$), $\nu_{max}^{2M}$ approaches $2/\pi$ \cite{bokil2012}.

First, we will consider the schemes with  a normalized CFL number $\ds \nu/\nu^{2M}_{max} = 0.7$   for both types of temporal discretizations. Relative phase errors are plotted in the range $\WHO\in[0,3]$. 
In Figure \ref{Fig:Phase_Error_FD_fully1}, we show errors of LF$(2,2M)$ and TP$(2,2M)$    with $W_1 = \pi/30$. The fully discrete schemes do give two peaks near $\WHO=1$ and $\WHO=\sqrt{\epsilon_{s}/\epsilon_{\infty}}$. As seen in Section \ref{time}, the phase errors for schemes based on semi-discretizations in time have two peaks in this range, while only one peak in observed for the semi-discrete spatial schemes as seen in Section \ref{semi}. Thus, it is reasonable to believe that the second peak results from time discretization, while the first one is associated with both space and time discretization. 
Comparing FDTD schemes with the same time discretization, phase error for the scheme with $M=2$ is smaller than that of the second order scheme for $M=1$. However, there is no significant difference among the phase errors with $M\geq2$  indicating that dispersion errors are dominated by time discretizations  when $M\geq2$. These observations   are consistent with our analysis. 
On the other hand, difference in phase error plots between LF$(2,2M)$ and TP$(2,2M)$ is similar to the results obtained for the semi-discrete in time schemes as seen in the second plot of Figure \ref{Fig:semi_time0}. 

\begin{figure}[h]
	\centering
	\includegraphics[scale= 0.3] {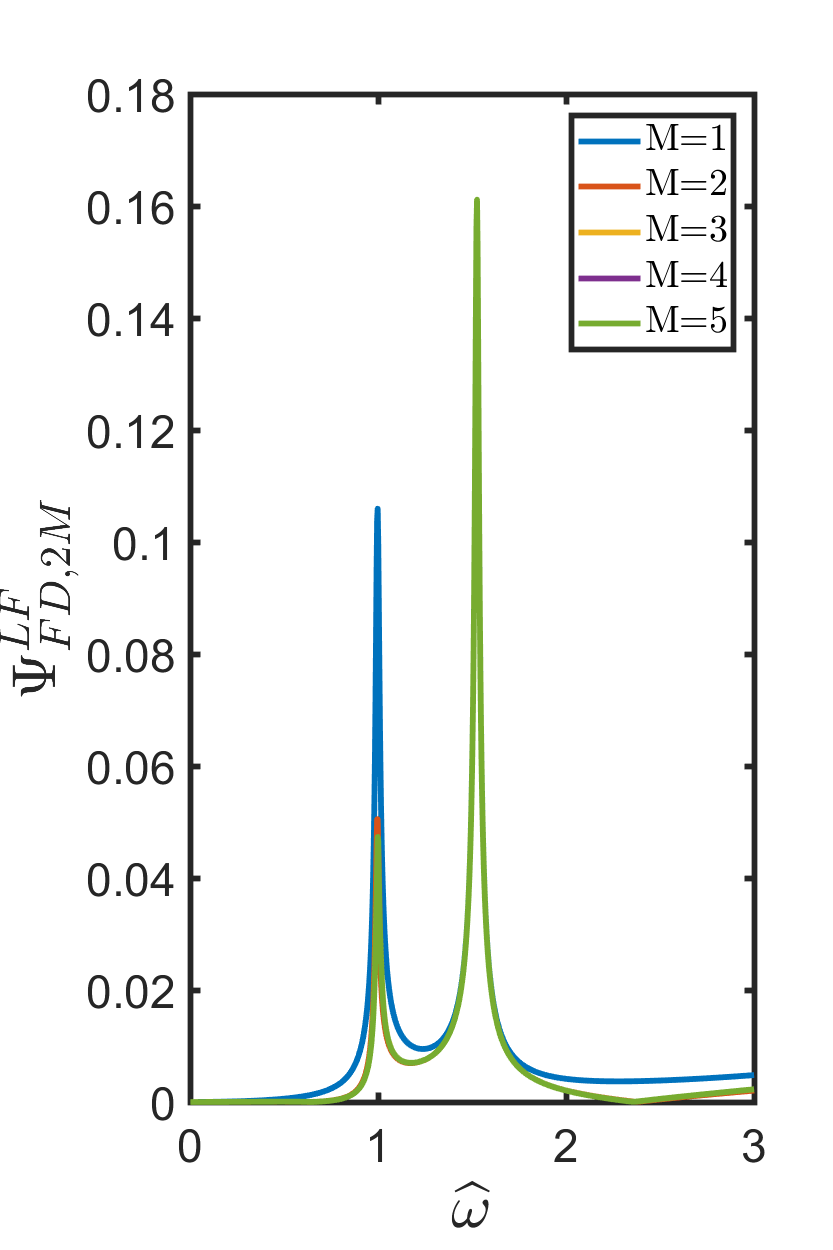}
	\includegraphics[scale= 0.3] {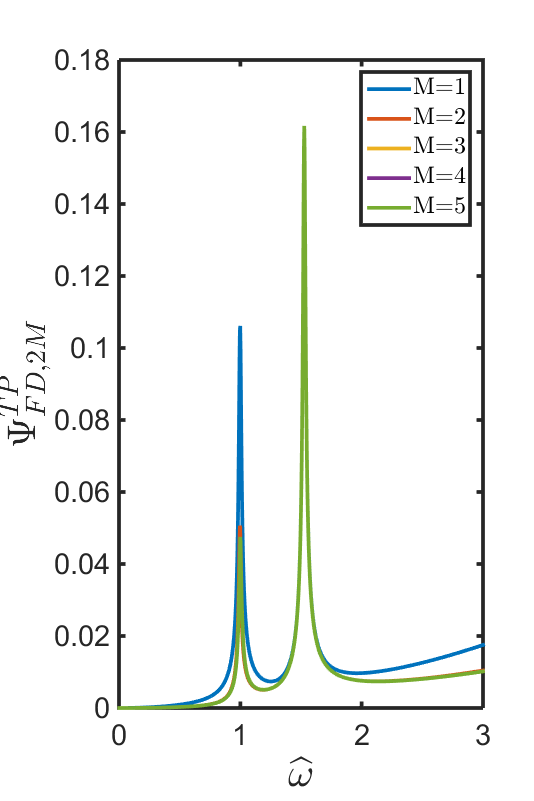}
	\caption{The relative phase error of fully discrete FDTD schemes for the physical modes with $\nu/\nu_{max}^{2M}=0.7$ and $W_1=\pi/30$. Left: the leap-frog scheme; right: the trapezoidal scheme.}
	\label{Fig:Phase_Error_FD_fully1}
\end{figure}

In the second experiment, we will consider the fully discrete trapezoidal FDTD scheme with various CFL numbers. We give the contour plots of the dispersion error at $\WHO=1$ in Figure \ref{Fig:Phase_Error_FD_fully2}, with $W_{1}\in[0.05,0.3]$ and $\omega_{1} h \in[0.01,0.1]$. 
 Here, the vertical coordinate is $W=\WHO\,W_{1}$  and the horizontal coordinate is $\ds |K| =|\WHO\sqrt{\epsilon(\WHO;\mathbf{p})} \, \omega_{1} h|$.  
In this coordinate system, for the range of values considered, the dispersion error of TP(2,2) can be improved by both taking smaller time steps and/or refining the spatial grid. With  fixed time step and spatial grid, we can also reduce the phase error by increasing the scheme to fourth order.
The contour lines in Figure \ref{Fig:Phase_Error_FD_fully2} of higher order ($M=2, 3$) schemes are horizontal, and the contours for TP(2,4) and TP(2,6) have no visible difference. Neither decreasing space mesh size nor increasing spatial order can reduce the phase error, which also illustrates the dominant role of temporal errors.

Both our analysis and figures demonstrate that FDTD schemes with  $M\geq3$ do not improve the phase error of fully discrete schemes significantly beyond that achieved for $M=2$. Hence, LF$(2,4)$ and TP$(2,4)$ seem to be the ``best'' schemes to work with from this perspective for most parameter choices (except for materials with large loss, or low-loss materials with certain range of frequencies as shown in Figures \ref{Fig:coeff} and \ref{Fig:coeff2}).

\begin{figure}
	\centering
	\includegraphics[scale= 0.23] {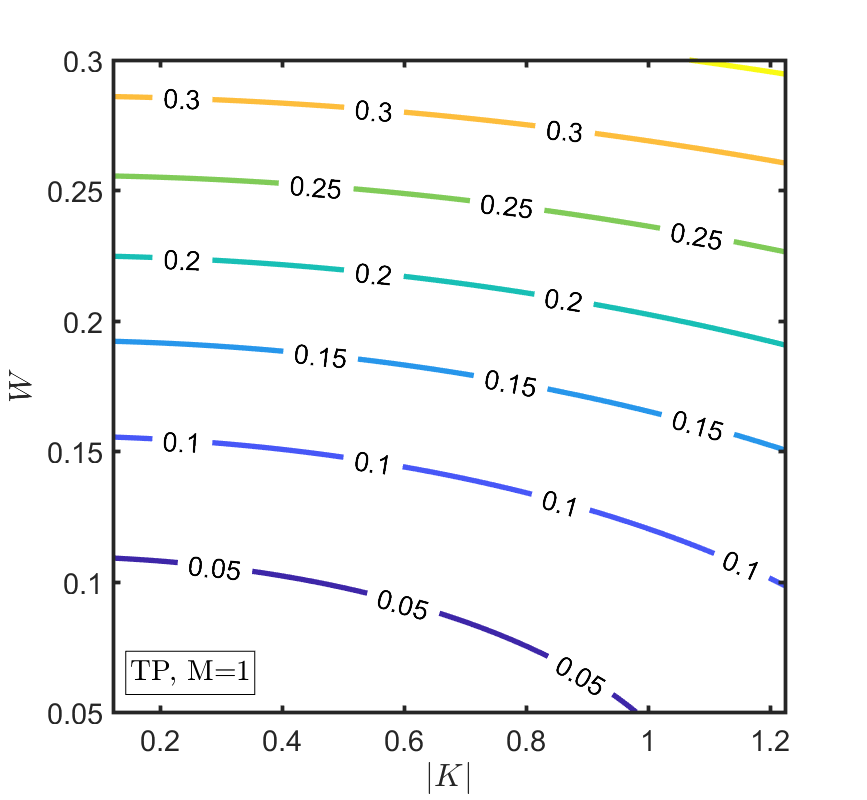} 
	\includegraphics[scale= 0.23] {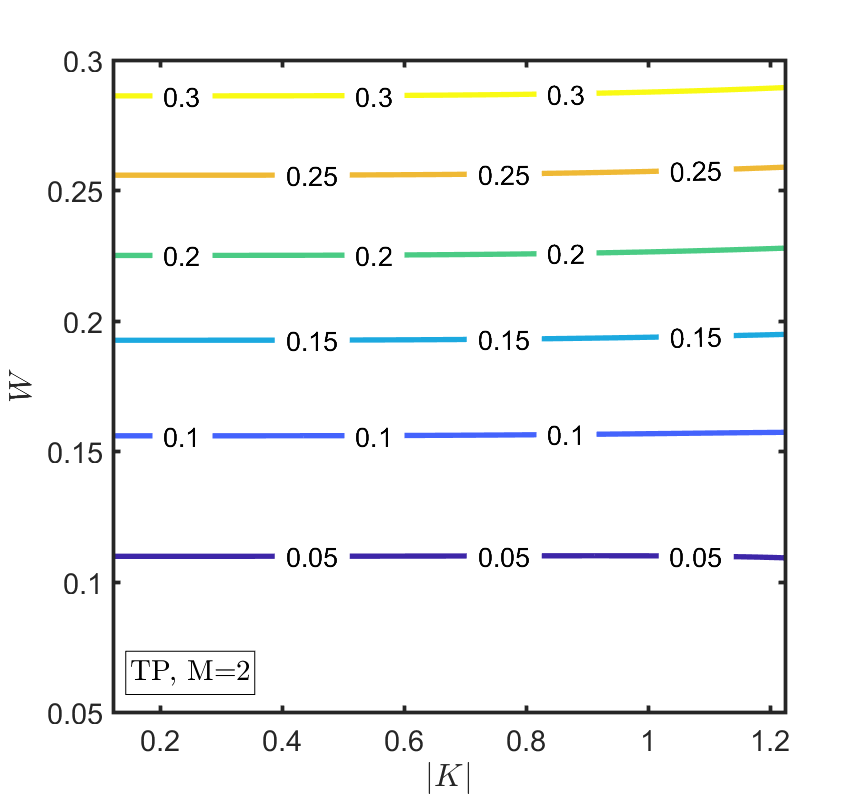} 
	\includegraphics[scale= 0.23] {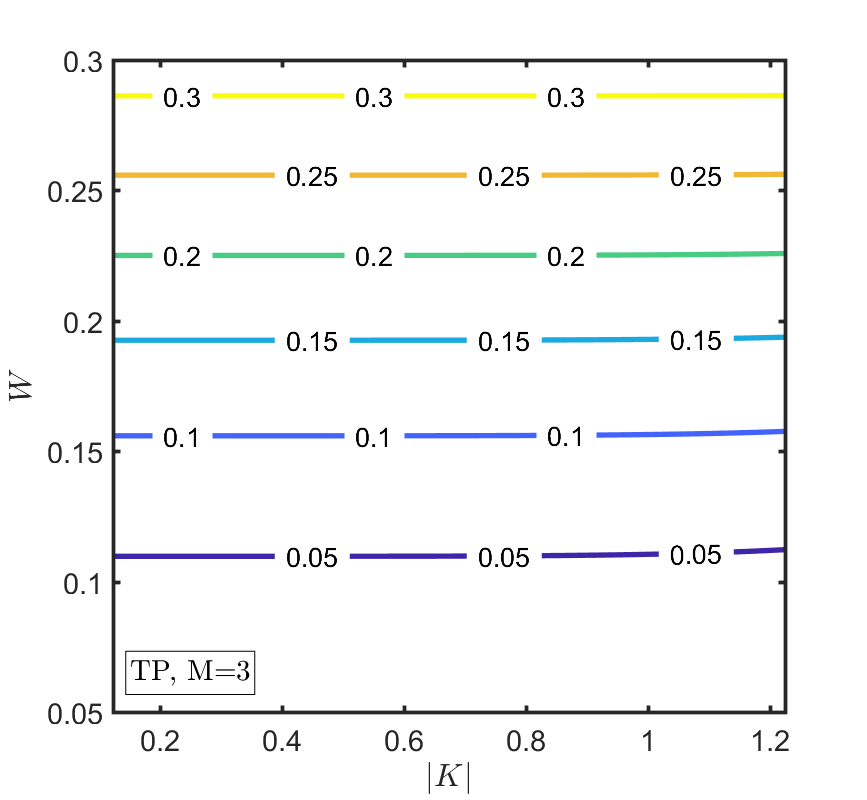} 	
	\caption{The contour plots of relative phase error of fully discrete FD schemes for the physical modes with trapezoidal scheme. $\WH{\omega}=1$. }
	\label{Fig:Phase_Error_FD_fully2}
\end{figure}

\section{Spatial Discretization: Discontinuous Galerkin Schemes}
\label{semidg}

In this section, similar to Section \ref{semifdtd}, we perform semi-discrete and fully discrete analysis when the spatial variable is discretized by DG schemes.
Here, we define the grid as
$x_{j+1/2}=(j+1/2)h,  \, j\in\mathbb{Z},$
with uniform mesh size $h$. Let $I_j=[x_{j-1/2},x_{j+1/2}]$ be a mesh element, with $x_j=\frac{1}{2}(x_{j-\frac{1}{2}}+x_{j+\frac{1}{2}})$ as its center. 
We now define a finite dimensional discrete space,
\begin{equation}\label{ldg:vhk}
V_h^p=\{v : v|_{I_j} \in P^p(I_j), \, j\in \mathbb{Z} \},
\end{equation}
which consists of piecewise polynomials of degree up to $p$ with respect to the mesh.
For any $v\in V_h^p$, let $v^+_{j+\frac{1}{2}}$ (resp. $v^-_{j+\frac{1}{2}}$) denote the limit value of $v$ at $x_{j+ \frac{1}{2}}$ from the element $I_{j+1}$ (resp. $I_j$), $[v]_{j+\frac{1}{2}}=v^+_{j+\frac{1}{2}} - v^-_{j+\frac{1}{2}}$  denote its jump, and $\{v\}_{j+\frac{1}{2}}=\frac{1}{2}(v^+_{j+\frac{1}{2}}+v^-_{j+\frac{1}{2}})$ be its average, again at $x_{j+\frac{1}{2}}$.

The semi-discrete DG method for the system \eqref{eq:sys} is formulated as follows: find $H_h(t,\cdot)$, $D_h(t,\cdot)$, $E_h(t,\cdot)$,  $P_h(t,\cdot)$, $J_h(t,\cdot)\in V_h^p$, such that  $\forall j,$
\begin{subequations}
	\label{eq:1d:sch}
	\begin{align}
	&\int_{I_j}\dd{t}{H_h}\phi dx +\int_{I_j} E_h\dd{x}\phi dx- (\widehat{E_h}\phi^-)_{j+1/2}
	+ (\widehat{E_h}\phi^+)_{j-1/2}=0,\quad \forall \phi\in V_h^p, \label{eq:sch1}\\
	&\int_{I_j}\dd{t}{D_h} \phi dx +\int_{I_j} H_h\dd{x}\phi dx- (\widetilde{H_h}\phi^-)_{j+1/2}
	+ (\widetilde{H_h}\phi^+)_{j-1/2}=0,\quad \forall \phi\in V_h^p, \label{eq:sch2}\\
	& \dd{t}{P_h}=J_h,  \label{eq:sch3}\\
	& \dd{t}{J_h}= -2\gamma J_h -\omega_1^2P_h+\omega_p^2E_h, \label{eq:sch4}\\
	& D_{h}=\epsilon_{\infty} E_{h} + P_{h}. \label{eq:sch5}
	\end{align}
\end{subequations}
Both the terms $\widehat{E_h}$ and $\widetilde{H_h}$ are numerical fluxes, and they are single-valued functions
defined on the cell interfaces and should be designed to ensure numerical stability and accuracy. In
the present work, we consider the following general form of numerical fluxes similar to the ones introduced in \cite{cheng2017L2},
\begin{subequations}
	\label{eq:flux}
	\begin{align}
	& \widehat{E_{h}} = \{E_{h}\} + \alpha[E_{h}] + \beta_1[H_{h}], \label{eq:flux_E} \\
	& \widetilde{H_h} = \{H_{h}\} - \alpha[H_{h}] + \beta_2[E_{h}]. \label{eq:flux_H}
	\end{align}
\end{subequations}
Here, $\alpha$, $\beta_1$ and $\beta_2$ are constants that are taken to be $\mathcal{O}(1)$, with $\beta_1$ and $\beta_2$ being non-negative for stability. For example, if we take $\alpha=\beta_1=\beta_2=0$, we have the central flux
\begin{align}
\label{eq:flux:c}
\widehat{E_{h}}=\{E_{h}\}, \quad \widetilde{H_h}=\{H_{h}\};
\end{align}
if $\alpha=\pm1/2$ and $\beta_1=\beta_2=0$, we have the alternating flux
\begin{align}
\label{eq:flux:a}
\widehat{E_{h}}=E_{h}^{-}, \quad\widetilde{H_h}=H_{h}^{+}; \quad \text{or} \quad
\widehat{E_{h}}=E_{h}^{+}, \quad\widetilde{H_h}=H_{h}^{-};
\end{align}
and if $\alpha=0$, $\beta_{1}=1/(2\sqrt{\epsilon_\infty})$, and $\beta_{2}=\sqrt{\epsilon_{\infty}}/2$, we have the ``upwind" flux for the   Maxwell's equations neglecting Lorentz dispersion
\begin{align}
\label{eq:flux:u}
\widehat{E_{h}}=\{E_{h}\}+\frac{1}{2\sqrt{\epsilon_{\infty}}}[H_{h}], \quad \widetilde{H_h}=\{H_{h}\}+\frac{\sqrt{\epsilon_{\infty}}}{2}[E_{h}].
\end{align}

In particular, when using the alternating flux with $p=0$, it is easy to check that the DG scheme is equivalent to FD2 discretization.

\subsection{Semi-discrete in space dispersion analysis}

In order to carry out the dispersion analysis, 
for piecewise $P^p$ polynomials, we  choose the basis functions on each element $I_{j}$ to be Lagrange polynomials $\phi_m^j(x)=\phi_{m}(\xi)$, $\xi=(x-x_{j})/h$:
$$\phi_{m}(\xi_{n})=\delta_{m,n}=
\left\{
\begin{array}{ll}
1, & n=m\\
0, & n\neq m\\
\end{array}
\right. ,\quad m,n=0,\ldots,p, \quad
\text{with} \quad 
\xi_{n}=
\left\{
\begin{array}{ll}
0, & p=0,\\
\frac{n}{p}-\frac{1}{2}, & \text{otherwise}.\\
\end{array}
\right. $$
Then, the numerical solution on $I_{j}$ can be written as 
\begin{align}
X_{h}=\sum_{m=0}^{p}\mathcal{X}^{m}_{j}\phi_{m}^j.
\end{align}
Here, $X$ can be $H$, $D$, $E$, $P$ and $J$. In particular, $\mathcal{X}^{m}_{j}$ is the point value of $X_{h}$ at $x_{j}+\xi_{m}h$. We define the vector $\textbf{X}_{j}=[\mathcal{X}^{0}_{j}, \cdots, \mathcal{X}^{p}_{j} ]^{T}$. Then, the semi-discrete scheme \eqref{eq:1d:sch} can be transformed into
\begin{subequations}
	\label{eq:DG_matrix}
	\begin{align}
	& \mc{M} (\textbf{H}_{j})_{t} +\mc{V}\textbf{E}_{j} +\mc{Q}_{-1}(\alpha)\textbf{E}_{j-1} +\mc{Q}_{0}(\alpha)\textbf{E}_{j} +\mc{Q}_{1}(\alpha)\textbf{E}_{j+1} +\mc{S}_{-1}(\beta_{1})\textbf{H}_{j-1} +\mc{S}_{0}(\beta_{1})\textbf{H}_{j} +\mc{S}_{1}(\beta_{1})\textbf{H}_{j+1} = 0\\
	& \mc{M} (\textbf{D}_{j})_{t} + \mc{V}\textbf{H}_{j} +\mc{Q}_{-1}(-\alpha)\textbf{H}_{j-1} +\mc{Q}_{0}(-\alpha)\textbf{H}_{j} +\mc{Q}_{1}(-\alpha)\textbf{H}_{j+1} +\mc{S}_{-1}(\beta_{2})\textbf{E}_{j-1} +\mc{S}_{0}(\beta_{2})\textbf{E}_{j} +\mc{S}_{1}(\beta_{2})\textbf{E}_{j+1} = 0 \\
	& (\textbf{P}_{j})_{t}=\textbf{J}_{j}\\
	& (\textbf{J}_{j})_{t}=-2\gamma\textbf{J}_{j} -\omega_{1}^{2} \textbf{P}_{j} +\omega_{p}^{2}\textbf{E}_{j}\\
	& \textbf{D}_{j}=\epsilon_{\infty}\textbf{E}_{j} +\textbf{P}_{j}
	\end{align}
\end{subequations}
where, $\mc{M}$, $\mc{V}$, $\mc{Q}_{*}$ and $\mc{S}_{*}$ are $(p+1)\times(p+1)$ matrices, with $*$ being $\pm1$ or 0. $\mc{M}$ is the element mass matrix, 
$$(\mathcal{M})_{m,n}=h \int_{-1/2}^{1/2} \phi_{m}(\xi) \phi_{n}(\xi) d\xi.$$
$\mathcal{V}$ is the element stiffness matrix,
$$(\mathcal{V})_{m,n}=\int_{-1/2}^{1/2} \phi'_{m}(\xi) \phi_{n}(\xi) d\xi.$$
 $\mathcal{Q}_{*}$ and $\mathcal{S}_{*}$ are related to the numerical flux, 
 \begin{align*}
 (\mathcal{Q}_{-1}(z))_{m,n}=\left\{\begin{array}{ll}
 \frac{1}{2}-z, & m=0,n=p,\\
 0, & \text{otherwise},\\
 \end{array}
 \right. && 
 (\mathcal{S}_{-1}(z))_{m,n} =\left\{\begin{array}{ll}
 -z, & m=0,n=p,\\
 0, & \text{otherwise},\\
 \end{array}
 \right. 
\end{align*}
\begin{align*}
 (\mathcal{Q}_{1}(z))_{m,n} =\left\{\begin{array}{ll}
 -\frac{1}{2}-z, & m=p,n=0,\\
 0, & \text{otherwise},\\
 \end{array}
 \right. &&
 (\mathcal{S}_{1}(z))_{m,n} =\left\{\begin{array}{ll}
 -z, & m=p,n=0,\\
 0, & \text{otherwise},\\
 \end{array}
 \right.
 \end{align*}
 \begin{align*}
(\mathcal{Q}_{0}(z))_{m,n} =\left\{\begin{array}{ll}
\frac{1}{2}+z, & m=n=0,\\
-\frac{1}{2}+z, & m=n=p,\\
0, & \text{otherwise},\\
\end{array}
\right.  &&
(\mathcal{S}_{0}(z))_{m,n} =\left\{\begin{array}{ll}
z, & m=n=0,\\
z, & m=n=p,\\
0, & \text{otherwise}.\\
\end{array}
\right.
 \end{align*}

Following standard practice for dispersion analysis for DG schemes, we  formulate the linear system resulting from \eqref{eq:DG_matrix}. With the assumption that 
$$\mathcal{X}^{m}_{j}(t)=X_{0}^{m}e^{i (k_{\text{DG},p} j h -\omega t)}, \quad m=0, \ldots, p,$$
we can obtain  
\begin{align}
\mathcal{A}_{\text{DG},p}\textbf{U}_{\text{DG}} = \textbf{0}, 
\end{align}
with $\textbf{U}_{\text{DG}}=[H^{0}_{0},\cdots,H^{p}_{0}, E^{0}_{0},\cdots,E^{p}_{0}, P^{0}_{0},\cdots, P^{p}_{0}, J^{0}_{0},\cdots,J^{p}_{0}]^T$, and
\begin{align}
\mathcal{A}_{\text{DG},p}=\begin{pmatrix}
	-i\omega \mc{M}+\mc{R} & \mc{P} & 0  & 0\\
	\widetilde{\mc{P}} & -i\omega \epsilon_{\infty}\mc{M}+\widetilde{\mc{R}} & -i\omega\mc{M} & 0\\
	0 & 0 & -i\omega \mathcal{I} & -\mathcal{I}\\
	0 & -\omega_{p}^{2}\mathcal{I} & \omega_{1}^{2}\mathcal{I} & (-i\omega+2\gamma) \mathcal{I} \\
	\end{pmatrix}.
\end{align}
Here, $\mc{I}$ is the $(p+1)\times(p+1)$ identity matrix, and
\begin{align*}
\displaystyle 
& \mc{P}= \mc{V} +\mc{Q}_{-1}(\alpha)e^{-i k_{\text{DG},p} h} +\mc{Q}_{0}(\alpha) +\mc{Q}_{1}(\alpha)e^{i k_{\text{DG},p} h} , && 
\mc{R}= \mc{S}_{-1}(\beta_1)e^{-i k_{\text{DG},p} h} +\mc{S}_{0}(\beta_1) +\mc{S}_{1}(\beta_1)e^{i k_{\text{DG},p} h} ,\\  
& \widetilde{\mc{P}}= \mc{V} +\mc{Q}_{-1}(-\alpha)e^{-i k_{\text{DG},p} h} +\mc{Q}_{0}(-\alpha) +\mc{Q}_{1}(-\alpha)e^{i k_{\text{DG},p} h} , && 
\widetilde{\mc{R}}=  \mc{S}_{-1}(\beta_2)e^{-i k_{\text{DG},p} h} +\mc{S}_{0}(\beta_2) +\mc{S}_{1}(\beta_2)e^{i k_{\text{DG},p} h} .
\end{align*}
Then, we can derive the corresponding numerical dispersion relation for the DG
methods by solving $\det(\mc{A}_{\text{DG},p})=0$.

Due to the dependence on the flux parameters $\alpha, \beta_1, \beta_2$ and the coupling of the local degrees of freedom, the dispersion relation is more complicated than that of the FD scheme.  We have the following theorem, which characterizes the dispersion relation satisfied by $k_{\text{DG},p}$.

\begin{theorem} \label{thm1}
	Consider the DG scheme \eqref{eq:1d:sch} with $V^{p}_{h}$ as the discrete space, then $k_{\text{DG},p}$ are the roots of a quartic polynomial equation in terms of $\xi=e^{ik_{\text{DG},p}h}$ if $\alpha^2+\beta_{1}\beta_{2}\ne1/4$, and $k_{\text{DG},p}$ are the roots of a quadratic polynomial equation in terms of $\xi=e^{ik_{\text{DG},p}h}$   when $\alpha^2+\beta_{1}\beta_{2}=1/4.$
	\end{theorem}

\begin{proof} 
	For $p=0$, we can obtain  
		\begin{align}
	\label{eq:p0}
	\det(\mc{A}_{\text{DG},0})
	=&	\left(e^{-2i k_{\text{DG},0} h}+e^{2i k_{\text{DG},0} h}\right)  \left(-1+4(\alpha^2+\beta_{1}\beta_{2})\right) \nonumber\\
	& + 4 \left(e^{-i k_{\text{DG},0} h}+e^{i k_{\text{DG},0} h}\right) \left(-4 (\alpha^{2} +\beta_{1}\beta_{2}) + i  \left( \beta_{1}\, \epsilon(\WH{\omega}; \mathbf{p}) +\beta_{2} \right) \omega h \right) \nonumber\\
	&  +  2 \left(1 + 12 (\alpha^{2} +\beta_{1}\beta_{2}) - 4 i  \left( \beta_{1}\, \epsilon(\WH{\omega}; \mathbf{p}) +\beta_{2} \right) \omega h - 
	2  (k^{\text{ex}} h)^2 \right).
	\end{align}
	Hence, the conclusion is straightforward.

	For $p\geq1$, note that the term $e^{ik_{\text{DG},p}h}$ only appears in $\mc{P}_{p0}$, $\widetilde{\mc{P}}_{p0}$, $\mc{R}_{p0}$ and $\widetilde{\mc{R}}_{p0}$, and $e^{-ik_{\text{DG},p}h}$ only appears in $\mc{P}_{0p}$, $\widetilde{\mc{P}}_{0p}$, $\mc{R}_{0p}$ and $\widetilde{\mc{R}}_{0p}$. Hence, by the properties of determinant under row or colume operations, $\det(\mc{A}_{\text{DG},p})$ is in the form of $C_{0} +C_{1}e^{ik_{\text{DG},p}h} +C_{2} e^{2ik_{\text{DG},p}h} + C_{-1} e^{-ik_{\text{DG},p}h} + C_{-2} e^{-2ik_{\text{DG},p}h}$, where $C_{i}$, $i=-2,-1,0,1,2$, do not depend on $k_{\text{DG},p}$. Hence, $k_{\text{DG},p}$ is the root of $C_{-2} + C_{-1}\xi+C_{0}\xi^2 +C_{1}\xi^3 +C_{2}\xi^4=0 $ with $\xi=e^{ik_{\text{DG},p}h}$. 		
	Furthermore, if $\alpha^2+\beta_{1}\beta_{2}=1/4$, we have the following two cases: 
	\begin{itemize}
		\item Case 1: $\alpha=\pm1/2$ and at least one of $\beta_{1}$ and $\beta_{2}$ is zero. Without loss of generality, we assume $\alpha=1/2$ and $\beta_{1}=0$. It is easy to check that $e^{ik_{\text{DG},p}h}$ can only appear in $\mc{P}_{p0}$, and $\widetilde{\mc{R}}_{p0}$, and $e^{-ik_{\text{DG},p}h}$ can only appear in $\widetilde{\mc{P}}_{0p}$ and $\widetilde{\mc{R}}_{0p}$. Therefore, the determinant $\det(\mc{A}_{\text{DG},p})$ is in the form of $C_{0} +C_{1}e^{ik_{\text{DG},p}h}+ C_{-1} e^{-ik_{\text{DG},p}h}$. 
		
		\item Case 2: $\alpha\neq\pm1/2$, $\beta_{1}\neq0$ and $\beta_{2}\neq0$. Then, all of $\mc{P}$, $\widetilde{\mc{P}}$, $\mc{R}$ and $\widetilde{\mc{R}}$ include $e^{ik_{\text{DG},p}h}$ and $e^{-ik_{\text{DG},p}h}$. However, with the help of the fact that
		$$\frac{1/2+\alpha}{\beta_{1}} = \frac{\beta_{2}}{1/2-\alpha},$$
		we can check that the new matrix $\left(\mc{D}\,\mc{A}_{\text{DG},p}\right)$ only has $e^{-ik_{\text{DG},p}h}$ in its first row and $e^{ik_{\text{DG},p}h}$ in its $(p+1)$-th row, where the matrix $\mc{D}$ is defined as following:
		\begin{align}
		\mc{D}_{m,n}=\left\{ \begin{array}{ll}
		1, & m=n ,\\
		-(1/2+\alpha)/\beta_{1}, & m=p+2, n=1,\\
		(1/2+\alpha)/\beta_{1}, & m=2p+2, n=p+1,\\
		0, & \text{otherwise},\\
		\end{array}
		\right.
		\qquad m, n=1,\cdots 4(p+1).
		\end{align}
		Hence, we can obtain that $\det(\mc{D}\,\mc{A}_{\text{DG},p})$ should be in the form of $C_{0} +C_{1}e^{ik_{\text{DG},p}h}+ C_{-1} e^{-ik_{\text{DG},p}h}$. On the other hand, it is easy to check that $\det(\mc{D})=1$.
		Therefore, the determinant $\det(\mc{A}_{\text{DG},p})=\det(\mc{D}\,\mc{A}_{\text{DG},p})$ . 
	
	\end{itemize}	
	 In both cases, we can reach the conclusion that $k_{\text{DG},p}$ is the root of $C_{0} +C_{1}\xi +C_{2}\xi^2=0 $ with $\xi=e^{ik_{\text{DG},p}h}$.   
\end{proof}

By Theorem \ref{thm1}, we can see that for the DG scheme \eqref{eq:1d:sch} employing the central flux \eqref{eq:flux:c} ($\alpha=\beta_{1}=\beta_{2}=0$), there are four discrete wave numbers $k_{\text{DG},p},$ corresponding to two physical modes and two spurious modes. While for the alternating fluxes \eqref{eq:flux:a} and the upwind flux \eqref{eq:flux:u} ($\alpha^2+\beta_{1}\beta_{2}=1/4$), there are only two discrete wave numbers $k_{\text{DG},p}$, corresponding to the physical modes. This conclusion holds for arbitrary $p.$ Unlike the FD scheme, when we increase the order of the accuracy of the scheme, the number of modes won't change when the dispersion relation is expressed by representing the discrete wavenumber as a function of the angular frequency.

Unfortunately, we can not obtain the analytical dispersion relation formula or formula with closed form for general $p\geq0$. Instead, in the following, we will discuss the cases of $p=0, \ldots, 3$   based on the small wave number limit $K \rightarrow 0$, while for higher order cases the dispersion relation becomes more cumbersome and is not included in this paper. 
In the following, we write
\begin{align}
\label{eq:b}
b=\omega \left( \beta_{1}\, \epsilon(\WH{\omega}; \mathbf{p}) +\beta_{2} \right),
\quad \text{and} \quad B=b\, h.
\end{align}
Note that, $b(\omega)=0$ if and only if $\beta_{1}=\beta_{2}=0$, and in the Taylor expansion, we assume $B\ll 1$ as well.
The results are given as follows.

\begin{itemize}
	\item When $\alpha=\beta_{1}=\beta_{2}=0$, there are four discrete wave numbers. Two of them correspond to the physical modes
	\renewcommand{\arraystretch}{1.5}
	\begin{align}
	\label{eq:dis_DG_semi1}
	k_{\text{DG}^{\text{phys}},p} =& \left\{ \begin{array}{ll}
	\pm k^{\text{ex}} \left(  1 +\frac{1}{6} K^2 +\mc{O}\left( K^4 \right) \right), & p=0, \\
	\pm k^{\text{ex}} \left(  1 -\frac{1}{48} K^2 +\mc{O}\left( K^4 \right) \right), & p=1, \\
	\pm k^{\text{ex}} \left(  1 +\frac{1}{16800} K^6 +\mc{O}\left( K^8 \right) \right), & p=2, \\
	\pm k^{\text{ex}} \left( 1  -\frac{1}{806400} K^6  +\mc{O}\left( K^8 \right) \right), & p=3. \\
	\end{array}\right.
	\end{align}
	\renewcommand{\arraystretch}{1}	
	The other two are the spurious modes
	\renewcommand{\arraystretch}{1.5}
	\begin{align}
	k_{\text{DG}^{\text{spur}}, p} =& \left\{ \begin{array}{ll}
	\pm k^{\text{ex}} \left( -\frac{\pi}{K} + 1 +\frac{1}{6}K^2 +\mc{O}( K^4 ) \right)  , & p=0, \\
	\pm k^{\text{ex}} \left( \frac{1}{3} +\frac{5}{1296}K^2 + \mc{O}\left( K^4 \right) \right), & p=1, \\
	\pm k^{\text{ex}} \left( -\frac{\pi}{K} + \frac{1}{5} +\frac{1}{375}K^2 +\mc{O}\left( K^4 \right) \right) , & p=2, \\
	\pm k^{\text{ex}} \left( \frac{1}{7} +\frac{4}{5145}K^2 +\mc{O}\left( K^4 \right) \right), & p=3. \\
	\end{array}\right.
	\end{align}
	\renewcommand{\arraystretch}{1}
	These formulas show that, when using the central flux, the physical modes have a dispersion error with order
	\begin{align}
	\left\{\begin{array}{ll}
		2p+2, & \text{if $p$ is even},\\
		2p, & \text{if $p$ is odd}.\\
		\end{array}
	\right.
	\end{align}
	 Moreover, the relative phase errors do not rely on other model parameters except $K$. 
    	 When $p$ is odd, the spurious modes  consists of two waves with wave length $(2p+1)$ times the actual wave length. And when $p$ is even, the spurious modes $k_{\text{DG}, s1, s2}$ will be inversely proportional to $h$, similar to the FD case.

	\item When $\alpha^2+\beta_{1}\beta_{2}=1/4$,  we have two physical modes:
	\renewcommand{\arraystretch}{1.5}
	\begin{align}
	\label{eq:dis_DG_semi2}
	k_{\text{DG}^{\text{phys}},p} =& \left\{ \begin{array}{ll}
	\pm k^{\text{ex}}  \left(  1 +\frac{1}{2} i B +\frac{1}{24}\left( K^2-9B^2\right)  +\mc{O}( i K^2 B + i B^3) \right), & p=0, \\
	\pm k^{\text{ex}} \left( 1 +\frac{1}{72} i K^2 B +\frac{1}{1080} \left(K^4-5 K^2 B^2\right) +\mc{O}\left( i K^2 B^3 \right) \right), & p=1, \\
	\pm k^{\text{ex}} \left( 1 +\frac{1}{7200} i K^4 B +\frac{1}{252000}\left( K^6-7 K^4 B^2 \right) +\mc{O}(i K^6 B + i K^4 B^3) \right), & p=2, \\
	\pm k^{\text{ex}} \left( 1 +\frac{1}{1411200} i K^6 B  +\frac{1}{88905600} \left( K^8-9 K^6 B^2\right) +\mc{O}(i K^8 B + i K^6 B^3) \right), & p=3. \\
	\end{array}\right.
	\end{align}
	\renewcommand{\arraystretch}{1}
	 Therefore, when $B=0$, i.e. with the alternating fluxes, the scheme has a dispersion error of order $(2p+2)$. In particular, the dispersion errors for $\alpha=1/2$ and $\alpha=-1/2$ are the same.
	On the other hand, for the upwind flux ($B\neq0$), we can observe a $(2p+1)$-th order dispersion error, 
	which is related to $K$ and $B$ at the same time.  
\end{itemize}

It is clear that the order of dispersion error for DG scheme is   higher than that of the $L^2$ convergence. This is an advantage of DG schemes, and differs from FD schemes significantly.

To verify the results above, in Figure \ref{Fig:semi_DG}, we study the relative phase error of the physical modes of the semi-discrete DG scheme \eqref{eq:1d:sch} with parameters in \eqref{eq:parameter} using the alternating flux (DG-AL) (only the result of one version of the alternating fluxes is shown, because they are identical to each other), the central flux (DG-CE) and the upwind flux (DG-UP).  The numerical wave number $k_{\text{DG},p}$ is obtained by solving $\det(\mathcal{A}_{\text{DG},p})=0$ exactly.  First, we fix $\omega_{1} h=\pi/30$, and plot the dependence of relative phase error as a function of $\WHO$, see the first row of Figure \ref{Fig:semi_DG}. It is clear that DG-AL always gives smallest error when the same discrete space is used. This can also be verified by comparing the orders and coefficients in \eqref{eq:dis_DG_semi1} and \eqref{eq:dis_DG_semi2}. All schemes have significant larger errors around $\WHO=1$. 
For DG-AL and DG-CE, the phase errors  approach zero near $\WHO=\sqrt{\epsilon_{s}/\epsilon_{\infty}}$ where $K$ is close to zero. For   DG-UP with $p=0$, the error is dominated by $B$ (see equation \eqref{eq:dis_DG_semi2}). Therefore, the ``zero'' point would shift to the ``zero'' point of $B$, which is about $\sqrt{1+\epsilon_{d}/(2\epsilon_{\infty})}\approx1.291$. Comparing FD2 ( Figure \ref{Fig:FD}) and DG-AL with $P^{0}$, they have the same performance. However, once we increase the order to $p=1$, DG-AL has significantly smaller error than FD4, resulting from the smaller coefficients in leading error terms, see \eqref{eq:dis_FD_semi1}, \eqref{eq:dis_FD_semi2} and \eqref{eq:dis_DG_semi2}. 
In the second row of Figure \ref{Fig:semi_DG}, we  present the errors at $\WHO_{1}=1$ with mesh refinement. Slopes indicate the order of accuracy for each scheme, which agree with our analysis in \eqref{eq:dis_DG_semi1} and \eqref{eq:dis_DG_semi2}.  
\begin{figure}[h]
	\centering
	\includegraphics[scale= 0.28] {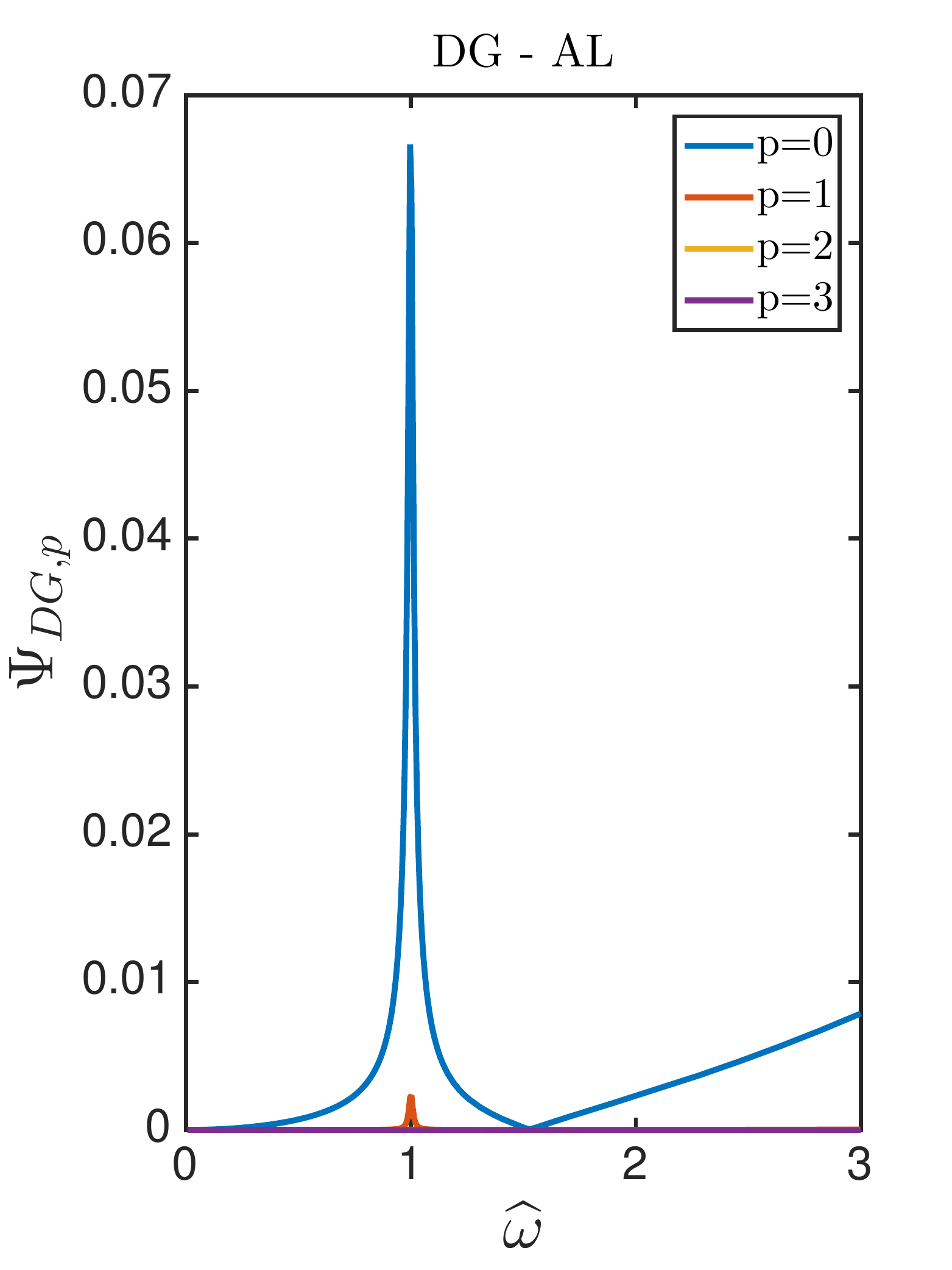}
	\includegraphics[scale= 0.28] {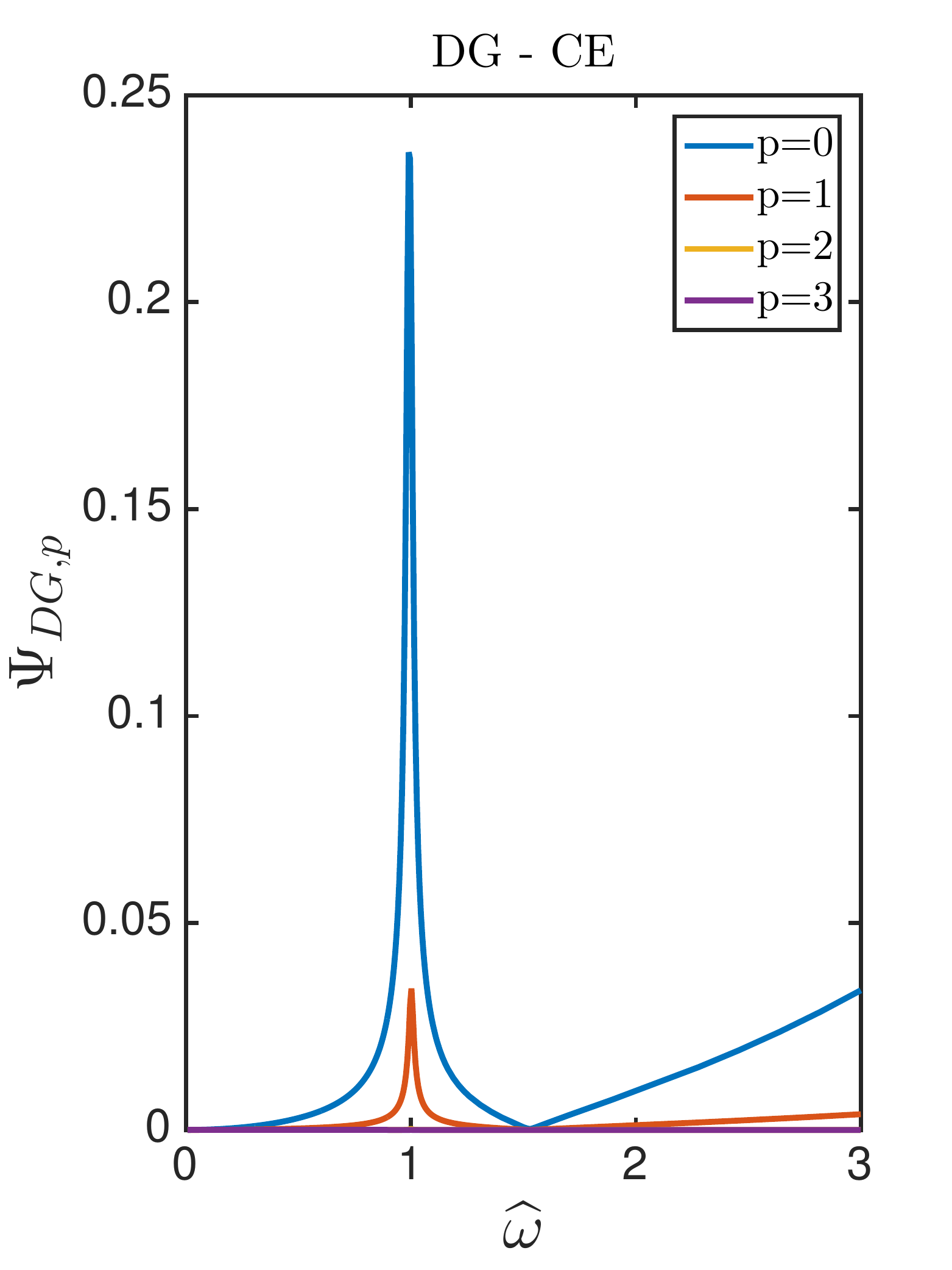}
	\includegraphics[scale= 0.28] {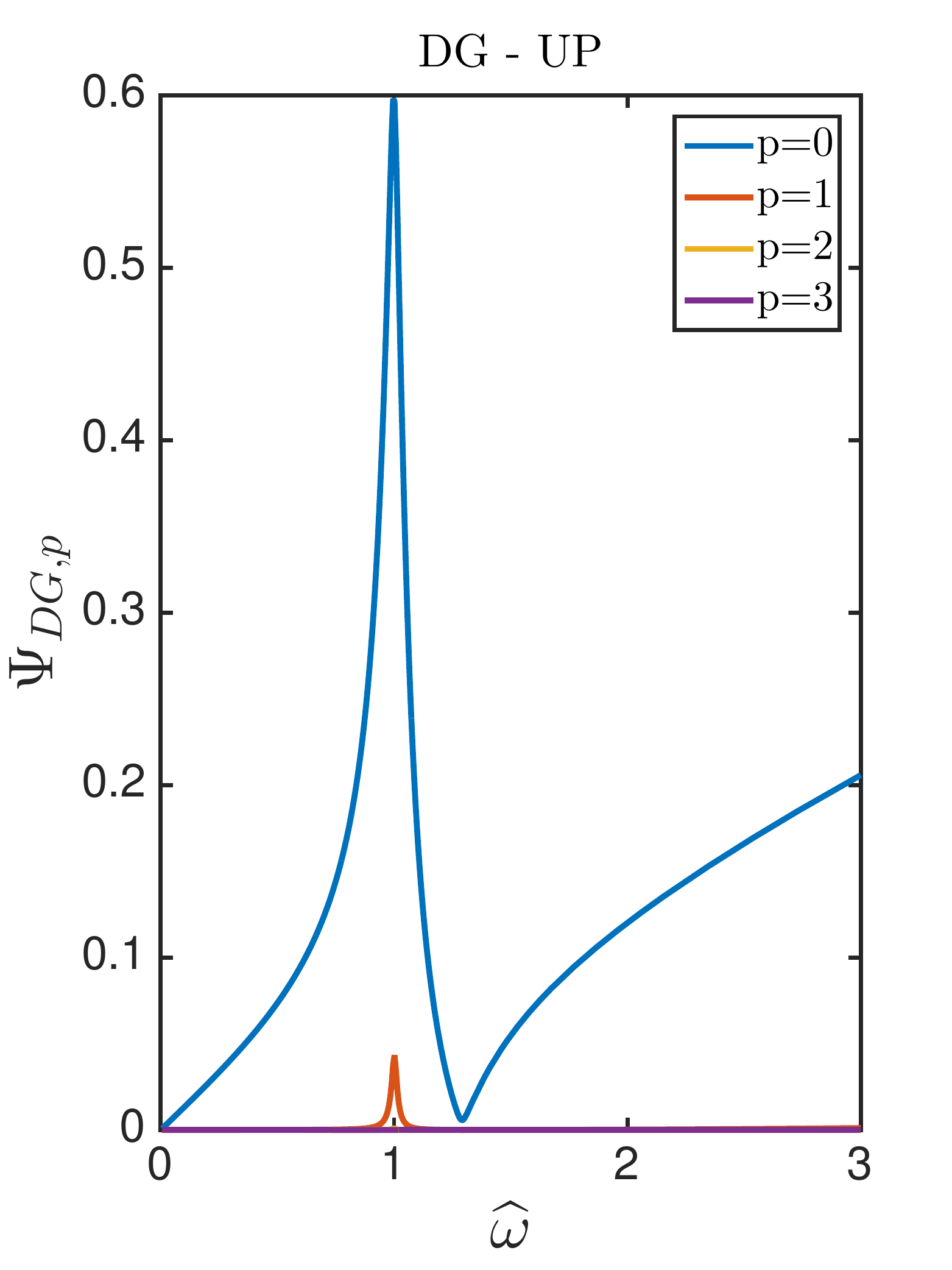}\\
	\includegraphics[scale= 0.28] {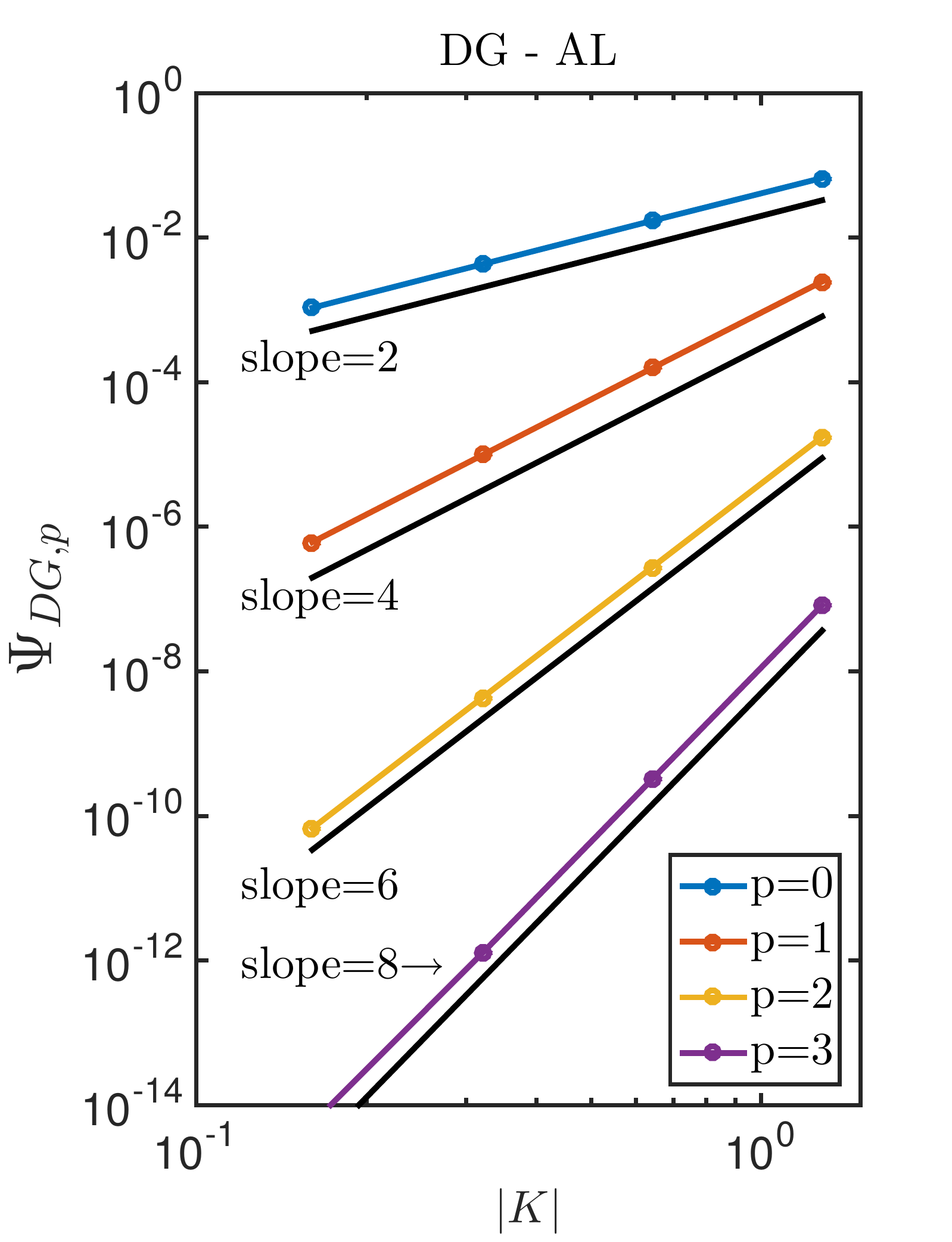}
	\includegraphics[scale= 0.28] {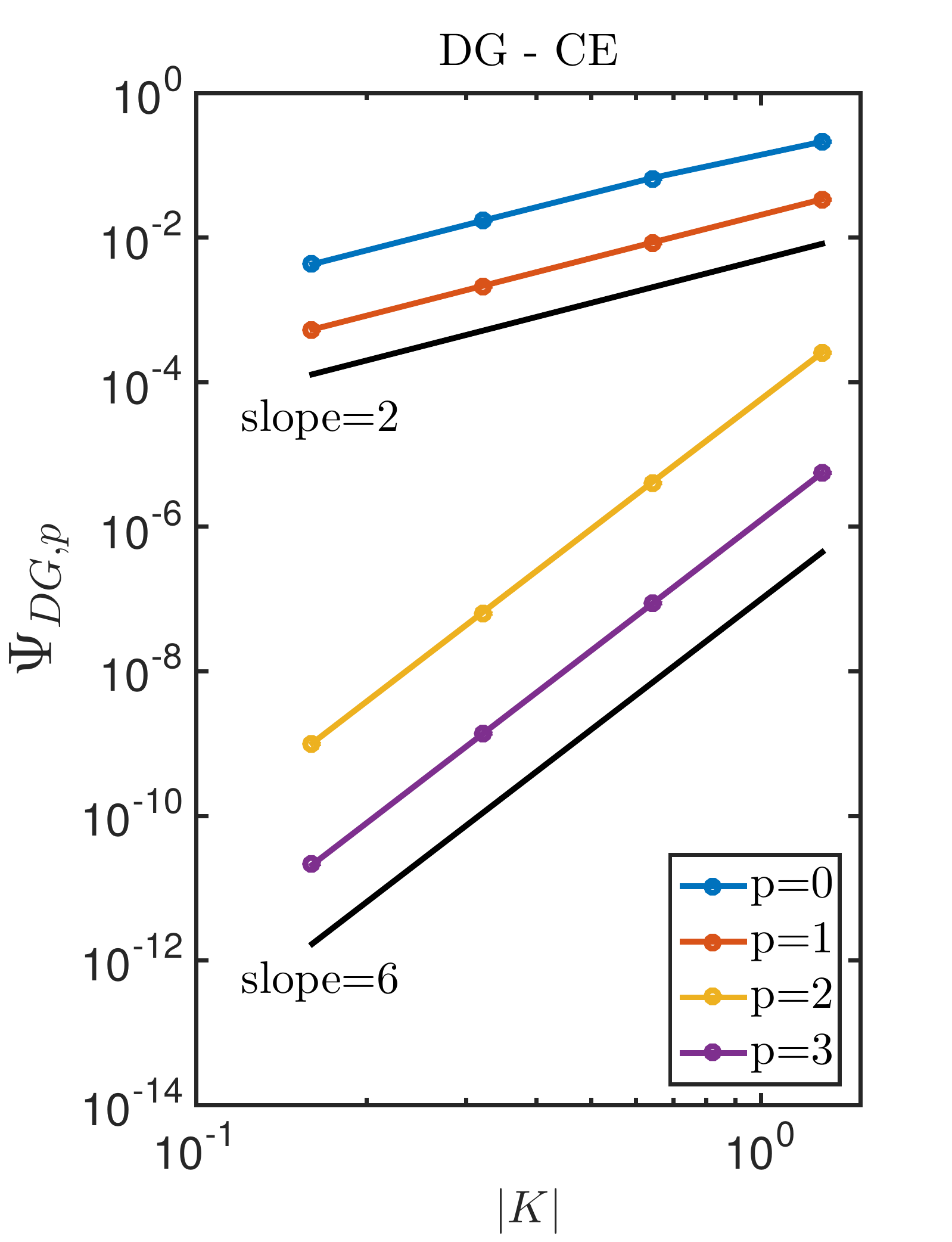}
	\includegraphics[scale= 0.28] {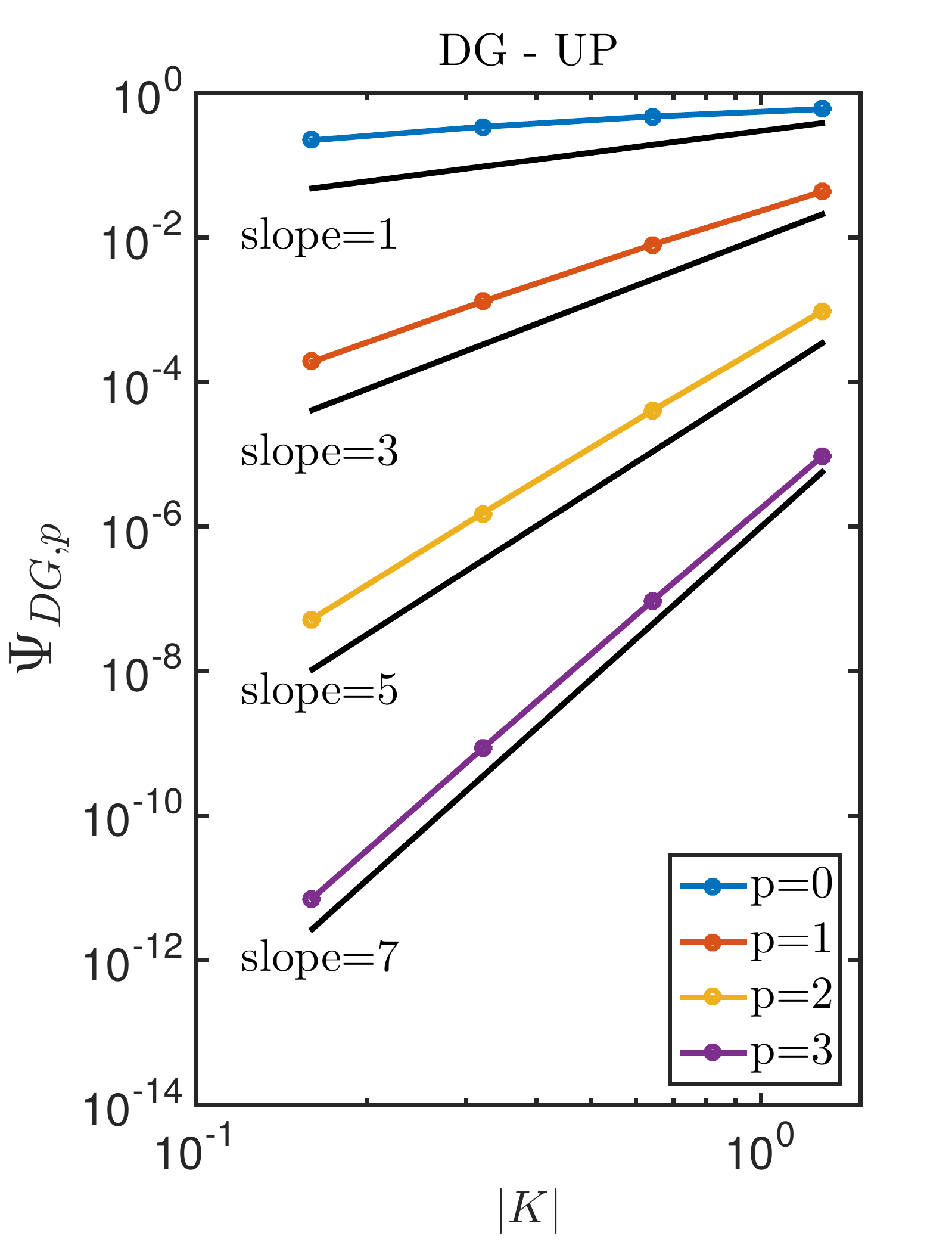}
	\caption{The relative phase error of semi-discrete DG scheme for the physical modes. First row: fix $\omega_{1}h=\pi/30$ with $\WHO_{1}\in[0,3]$; second row: fix $\WHO_{1}=1$ with different $\omega_{1}h\in\{ \frac{\pi}{30},\frac{\pi}{60},\frac{\pi}{120},\frac{\pi}{240} \}$.}
	\label{Fig:semi_DG}
\end{figure}

\section{Fully discrete Discontinuous Galerkin Methods} $\,$ 
\label{dgtd}
Here, we consider the DG scheme \eqref{eq:1d:sch} coupled with the leap-frog time discretization \eqref{eq:LF} and the trapezoidal time discretization \eqref{eq:TP}.
To analyze dispersion relation for those fully discrete   schemes, we assume  numerical solutions in the form of
$$(\mathcal{X}^{m})^{n}_{j}=X^{m}_{0}e^{i(k_{\text{DG},p}^{\text{*}} j h -\omega n\Delta t)}, \quad m=0, \ldots, p.$$
Then the fully discrete system will yield a linear system
\begin{align}
\mathcal{A}_{\text{DG},p}^{\text{*}}\textbf{U}_{\text{DG}}=0,
\end{align}
where $*$ can be LF and TP.

\subsection{Fully discrete dispersion analysis: leap-frog-DG schemes}

From simple algebra, $\mathcal{A}^{\text{LF}}_{\text{DG},p}$ is given by
\renewcommand{\arraystretch}{1.5}
\begin{align*}
\begin{small}
\begin{pmatrix}
-i\sin(\frac{W}{2}) \mc{M}+\frac{\Delta t}{2}\cos(\frac{W}{2})\mc{R} & \frac{\Delta t}{2}\mc{P} & 0  & 0\\
\frac{\Delta t}{2}\widetilde{\mc{P}} & -i\epsilon_{\infty}\sin(\frac{W}{2})\mc{M}+\frac{\Delta t}{2} \cos(\frac{W}{2})\widetilde{\mc{R}} & -i\sin(\frac{W}{2})\mc{M} & 0\\
0 & 0 & i\sin(\frac{W}{2}) \mathcal{I} & \frac{\Delta t}{2}\cos(\frac{W}{2})\mathcal{I}\\
0 & \omega_{p}^{2}\frac{\Delta t}{2}\cos(\frac{W}{2})\mathcal{I} & -\omega_{1}^{2}\frac{\Delta t}{2}\cos(\frac{W}{2})\mathcal{I} & (i\sin(\frac{W}{2})-\gamma\Delta t \cos(\frac{W}{2})) \mathcal{I} \\
\end{pmatrix}.
\end{small}
\end{align*}
\renewcommand{\arraystretch}{1}

Since $\mc{P}$, $\widetilde{\mc{P}}$, $\mc{R}$ and $\widetilde{\mc{R}}$ are multiplied by different factors, we can not cancel $e^{ik^{\text{LF}}_{\text{DG},p}h}$ or $e^{-i k^{\text{LF}}_{\text{DG},p}h}$ in $\mathcal{A}^{\text{LF}}_{\text{DG},p}$ as what we did in case 2 of the proof of Theorem \ref{thm1}. Hence,  the dispersion relation satisfied by $k^{\text{LF}}_{\text{DG},p}$ will satisfy the following theorem which differs from Theorem \ref{thm1}. 

\begin{theorem} \label{thm2}
	Consider the fully discrete leap-frog-DG scheme with $V^{p}_{h}$ as the discrete space, if $\alpha=\pm\frac{1}{2}$ and $\beta_{1}=\beta_{2}=0$, then $k^{\text{LF}}_{\text{DG},p}$ are the roots of a quadratic polynomial equation in terms of $\xi=e^{ik^{\text{LF}}_{\text{DG},p}h}$. Otherwise, $k^{\text{LF}}_{\text{DG},p}$ are the roots of a quartic polynomial equation in terms of $\xi=e^{ik^{\text{LF}}_{\text{DG},p}h}$.
\end{theorem}

Below, we analyze numerical dispersion property of the physical modes when $W\ll 1$. Note that $\displaystyle B=bh=\frac{b/\omega}{\sqrt{\epsilon_{\infty}}\, \nu} W\ll 1$ with a fixed CFL number $\nu$, with $B$ given in \eqref{eq:b}.

\begin{itemize}
	\item When using the central flux, i.e. $\alpha=\beta_{1}=\beta_{2}=0$, we   have four solutions, and two of them correspond to the physical modes,  
	\begin{align}
	\label{eq:disp_LF_DG_CE1}
	\renewcommand{\arraystretch}{2}
	k^{\text{LF}}_{\text{DG}^{\text{ phys}},p}
	= \left\{\begin{array}{ll}
	\displaystyle \pm k^{\text{ex}}\left( 1 + \frac{1}{12}\left(  \frac{\delta(\WHO;\mathbf{p})}{\epsilon(\WHO;\mathbf{p})} -\frac{1}{2}\right) W^2 +\frac{1}{6}K^2 +\mc{O}\left( W^4 + W^2 K^2 + K^4 \right)\right) , & p=0, \\
	\displaystyle \pm k^{\text{ex}}\left( 1 + \frac{1}{12}\left(  \frac{\delta(\WHO;\mathbf{p})}{\epsilon(\WHO;\mathbf{p})} -\frac{1}{2} \right) W^2 -\frac{1}{48}K^2 +\mc{O}\left( W^4 + W^2 K^2 + K^4\right)\right) , & p=1, \\
	\displaystyle \pm k^{\text{ex}}\left( 1 + \frac{1}{12}\left(  \frac{\delta(\WHO;\mathbf{p})}{\epsilon(\WHO;\mathbf{p})} -\frac{1}{2} \right) W^2  +\mc{O}\left( W^4 +K^{6}\right)\right) , & p=2,3,\\
	\end{array}
	\right.
	\renewcommand{\arraystretch}{1}
	\end{align}
	in the case of $K\ll 1$ and $W\ll 1$. They can be further written as 
	\begin{align}
	\label{eq:disp_LF_DG_CE2}
	\renewcommand{\arraystretch}{2}
	k^{\text{LF}}_{\text{DG}^{\text{ phys}},p}
	= \left\{\begin{array}{ll}
	\displaystyle \pm k^{\text{ex}}\left( 1 + \frac{1}{12}\left(  \frac{\delta(\WHO;\mathbf{p})}{\epsilon(\WHO;\mathbf{p})} -\frac{1}{2} + \frac{2\epsilon(\WHO;\mathbf{p})}{\epsilon_{\infty}\nu^2} \right) W^2  +\mc{O}\left( W^4 \right)\right) , & p=0, \\
	\displaystyle \pm k^{\text{ex}}\left( 1 + \frac{1}{12}\left(  \frac{\delta(\WHO;\mathbf{p})}{\epsilon(\WHO;\mathbf{p})} -\frac{1}{2} - \frac{\epsilon(\WHO;\mathbf{p})}{4\epsilon_{\infty}\nu^2}\right) W^2  +\mc{O}\left( W^4 \right)\right) , & p=1, \\
	\displaystyle \pm k^{\text{ex}}\left( 1 + \frac{1}{12}\left(  \frac{\delta(\WHO;\mathbf{p})}{\epsilon(\WHO;\mathbf{p})} -\frac{1}{2} \right) W^2  +\mc{O}\left( W^4 \right)\right) , & p=2, 3, \\
	\end{array}
	\right.
	\renewcommand{\arraystretch}{1}
	\end{align}
	 with $W\ll 1$ and a fixed CFL number $\nu$.

	\item When using the alternating flux, i.e. $\alpha=\pm1/2$ and $\beta_{1}=\beta_{2}=0$, there are only two solutions, corresponding to the physical modes,
	\begin{align}
	\label{eq:disp_LF_DG_AL1}	
	\renewcommand{\arraystretch}{2} 
	k^{\text{LF}}_{\text{DG}^{\text{ phys}},p}
	= \left\{\begin{array}{ll}
	\displaystyle \pm k^{\text{ex}}\left( 1 + \frac{1}{12}\left(  \frac{\delta(\WHO;\mathbf{p})}{\epsilon(\WHO;\mathbf{p})} -\frac{1}{2} \right) W^2 +\frac{1}{24}K^2 +\mc{O}\left( W^4 + W^2K^2 + K^4\right)\right) , & p=0, \\
	\displaystyle \pm k^{\text{ex}}\left( 1 + \frac{1}{12}\left(  \frac{\delta(\WHO;\mathbf{p})}{\epsilon(\WHO;\mathbf{p})} -\frac{1}{2} \right) W^2  +\mc{O}\left( W^4 + K^{2p+2}\right)\right) , & p=1, 2, 3,\\
	\end{array}
	\right.
	\renewcommand{\arraystretch}{1}
	\end{align}
	in  the case of $K\ll 1$ and $W\ll 1$, or
	\begin{align}
	\label{eq:disp_LF_DG_AL2}
	\renewcommand{\arraystretch}{2} 
	k^{\text{LF}}_{\text{DG}^{\text{ phys}},p}
	= \left\{\begin{array}{ll}
	\displaystyle \pm k^{\text{ex}}\left( 1 + \frac{1}{12}\left(  \frac{\delta(\WHO;\mathbf{p})}{\epsilon(\WHO;\mathbf{p})} -\frac{1}{2} + \frac{\epsilon(\WHO;\mathbf{p})}{2\epsilon_{\infty}\nu^2} \right) W^2  +\mc{O}\left( W^4 \right)\right) , & p=0, \\
	\displaystyle \pm k^{\text{ex}}\left( 1 + \frac{1}{12}\left(  \frac{\delta(\WHO;\mathbf{p})}{\epsilon(\WHO;\mathbf{p})} -\frac{1}{2} \right) W^2  +\mc{O}\left( W^4 \right)\right) , & p=1, 2, 3.\\
	\end{array}
	\right.
	\renewcommand{\arraystretch}{1}
	\end{align}
	 with $W\ll 1$ and a fixed CFL number $\nu$.

	\item When using the upwind flux, i.e. $\alpha=0$, $\ds \beta_{1}=\frac{1}{2\sqrt{\epsilon_{\infty}}}$ and $\ds \beta_{2}=\frac{\sqrt{\epsilon_{\infty}}}{2}$, there are four solutions. 
	The two physical modes are
	\begin{align}
	\label{eq:disp_LF_DG_UP1}
	\renewcommand{\arraystretch}{2} 
	k^{\text{LF}}_{\text{DG}^{\text{ phys}},p}
	= \left\{\begin{array}{ll}
	\displaystyle \pm k^{\text{ex}} \left(  1 +\frac{1}{2} i B +\mc{O}\left(W^2 + K^2 \right) \right) , & p=0, \\
	\displaystyle \pm k^{\text{ex}} \left( 1 + \frac{1}{12}\left(  \frac{\delta(\WHO;\mathbf{p})}{\epsilon(\WHO;\mathbf{p})} -\frac{1}{2} \right) W^2 +\mc{O}\left(W^4 +i K^{2p} B \right) \right), & p=1, 2, 3, \\
	\end{array}
	\right.
	\renewcommand{\arraystretch}{1}
	\end{align}
	which can also be written as
	\begin{align}
	\label{eq:disp_LF_DG_UP2}
	\renewcommand{\arraystretch}{2} 
	k^{\text{LF}}_{\text{DG}^{\text{ phys}},p}
	= \left\{\begin{array}{ll}
	\displaystyle \pm k^{\text{ex}} \left( 1 +i \frac{b/\omega}{2\sqrt{\epsilon_{\infty}}\,\nu} W +\mc{O}\left(W^2 \right) \right) , & p=0, \\
	\displaystyle \pm k^{\text{ex}} \left( 1 + \frac{1}{12}\left(  \frac{\delta(\WHO;\mathbf{p})}{\epsilon(\WHO;\mathbf{p})} -\frac{1}{2} \right) W^2 +\mc{O}\left( i W^3 \right) \right), & p=1, \\
	\displaystyle \pm k^{\text{ex}} \left( 1+ \frac{1}{12}\left(  \frac{\delta(\WHO;\mathbf{p})}{\epsilon(\WHO;\mathbf{p})} -\frac{1}{2} \right) W^2 +\mc{O}\left( W^4 \right) \right), & p=2, 3,\\
	\end{array}
	\right.
	\renewcommand{\arraystretch}{1}
	\end{align} 
	with $W\ll 1$ and a fixed CFL number $\nu$. 
\end{itemize}

The formulations above demonstrate that all fully discrete schemes are second order accurate in numerical dispersion, except for the upwind flux with $P^{0}$, for which the error is of first order. Comparing the leading error terms   with the same flux but with different $p$ values, we can see that the temporal error would be dominant when $p\geq2$ by upwind flux or central flux and $p\geq1$ by alternating fluxes.  
Moreover, it is observed that the leading error terms for high order schemes, when the temporal error dominates, are the same with the FD schemes \eqref{eq:disp_FD_LF1} and \eqref{eq:disp_FD_LF2}, which come from the time discretization \eqref{eq:dis_LF}. 
 In particular, the leading terms of DG scheme with alternating flux are the same as those of FD scheme. Hence, we can also have counterintuitive results that the lower order scheme performs better than higher order scheme when numerical dispersion is concerned for some given dispersive media and discretization parameters. {Similar results are observed for other numerical fluxes as well as for the  trapzoidal-DG schemes in next section. 

\subsection{Fully discrete dispersion analysis: trapezoidal-DG schemes}

We can deduce that $\mathcal{A}^{\text{TP}}_{\text{DG},p}$ is given by
\renewcommand{\arraystretch}{1.5}
\begin{align*}
\begin{small}
\begin{pmatrix}
-i\sin(\frac{W}{2}) \mc{I}+\frac{\Delta t}{2}\cos(\frac{W}{2})\mc{R} & \frac{\Delta t}{2}\cos(\frac{W}{2})\mc{P} & 0  & 0\\
\frac{\Delta t}{2}\cos(\frac{W}{2})\widetilde{\mc{P}} & -i\epsilon_{\infty}\sin(\frac{W}{2})\mc{I}+\frac{\Delta t}{2} \cos(\frac{W}{2})\widetilde{\mc{R}} & -i\sin(\frac{W}{2})\mc{I} & 0\\
0 & 0 & i\sin(\frac{W}{2}) \mathcal{I} & \frac{\Delta t}{2}\cos(\frac{W}{2})\mathcal{I}\\
0 & \frac{\Delta t}{2}\omega_{p}^{2}\cos(\frac{W}{2})\mathcal{I} & -\frac{\Delta t}{2}\omega_{1}^{2}\cos(\frac{W}{2})\mathcal{I} & (i\sin(\frac{W}{2})-\gamma\Delta t \cos(\frac{W}{2})) \mathcal{I} \\
\end{pmatrix}.
\end{small}
\end{align*}
\renewcommand{\arraystretch}{1}

Because we multiply $\mc{P}$, $\widetilde{\mc{P}}$, $\mc{R}$ and $\widetilde{\mc{R}}$ by the same factor $\cos(\frac{W}{2})$, it is easy to check that  a similar theorem as Theorem \ref{thm1}  holds for this fully discrete scheme. 

\begin{theorem} \label{thm3}
	Consider the fully discrete trapezoidal-DG scheme with $V^{p}_{h}$ as the discrete space, then $k^{\text{TP}}_{\text{DG},p}$ are the roots of a quartic polynomial equation in terms of $\xi=e^{ik^{\text{TP}}_{\text{DG},p}h}$ if $\alpha^2+\beta_{1}\beta_{2}\ne1/4$, and $k_{\text{DG},p}$ are the roots of a quadratic polynomial equation in terms of $\xi=e^{ik^{\text{TP}}_{\text{DG},p}h}$ when $\alpha^2+\beta_{1}\beta_{2}=1/4.$
\end{theorem}

Below, we list the physical modes $k^{\text{TP}}_{\text{DG},p}$ for $p\geq0$, and perform an asymptotic analysis when $W\ll 1$ and $K\ll 1.$ 
\begin{itemize}
	\item When using the central flux, i.e. $\alpha=\beta_{1}=\beta_{2}=0$, we  have four solutions, and two of them correspond to the physical modes.  When $W\ll1$ and $K\ll 1$, the physical solutions have the form as
	\begin{align}
	\label{eq:disp_TP_DG_CE1}
	\renewcommand{\arraystretch}{2} 
	k^{\text{TP}}_{\text{DG}^{\text{ phys}},p}
	= \left\{\begin{array}{ll}
	\displaystyle \pm k^{\text{ex}}\left( 1 + \frac{1}{12}\left(  \frac{\delta(\WHO;\mathbf{p})}{\epsilon(\WHO;\mathbf{p})} +1 \right) W^2 +\frac{1}{6}K^2 +\mc{O}\left( W^4 + W^2K^2 + K^4\right)\right) , & p=0, \\
	\displaystyle \pm k^{\text{ex}}\left( 1 + \frac{1}{12}\left(  \frac{\delta(\WHO;\mathbf{p})}{\epsilon(\WHO;\mathbf{p})} +1 \right) W^2 - \frac{1}{48}K^2  +\mc{O}\left( W^4 + W^2K^2 + K^4 \right)\right) , & p=1, \\
	\displaystyle \pm k^{\text{ex}}\left( 1 + \frac{1}{12}\left(  \frac{\delta(\WHO;\mathbf{p})}{\epsilon(\WHO;\mathbf{p})} +1 \right) W^2  +\mc{O}\left( W^4 +K^{6}\right)\right) , & p=2, 3,\\
	\end{array}
	\right.
	\renewcommand{\arraystretch}{1}
	\end{align}
	and can be further rewritten into	 
	\begin{align}
	\label{eq:disp_TP_DG_CE2}
	\renewcommand{\arraystretch}{2}
	k^{\text{TP}}_{\text{DG}^{\text{ phys}},p}
	= \left\{\begin{array}{ll}
	\displaystyle \pm k^{\text{ex}}\left( 1 + \frac{1}{12}\left(  \frac{\delta(\WHO;\mathbf{p})}{\epsilon(\WHO;\mathbf{p})} +1 + \frac{2\epsilon(\WHO;\mathbf{p})}{\epsilon_{\infty}\nu^2} \right) W^2  +\mc{O}\left( W^4 \right)\right) , & p=0, \\
	\displaystyle \pm k^{\text{ex}}\left( 1 + \frac{1}{12}\left(  \frac{\delta(\WHO;\mathbf{p})}{\epsilon(\WHO;\mathbf{p})} +1 - \frac{\epsilon(\WHO;\mathbf{p})}{4\epsilon_{\infty}\nu^2} \right) W^2  +\mc{O}\left( W^4 \right)\right) , & p=1, \\
	\displaystyle \pm k^{\text{ex}}\left( 1 + \frac{1}{12}\left(  \frac{\delta(\WHO;\mathbf{p})}{\epsilon(\WHO;\mathbf{p})} +1 \right) W^2  +\mc{O}\left( W^4 \right)\right) , & p=2, 3.\\
	\end{array}
	\right.
	\renewcommand{\arraystretch}{1}
	\end{align}
	in the case of $W\ll 1$ and with a fixed CFL number $\nu.$

	\item When using the alternating flux, i.e. $\alpha=\pm1/2$ and $\beta_{1}=\beta_{2}=0$, there are only two solutions corresponding to the physical modes,
	\renewcommand{\arraystretch}{2} 
	\begin{align}
	\label{eq:disp_TP_DG_AL1}
	k^{\text{TP}}_{\text{DG}^{\text{ phys}},p}
	= \left\{\begin{array}{ll}
	\ds \pm k^{\text{ex}}\left( 1 + \frac{1}{12}\left(  \frac{\delta(\WHO;\mathbf{p})}{\epsilon(\WHO;\mathbf{p})} +1 \right) W^2 +\frac{1}{24}K^2 +\mc{O}\left( W^4 + W^2 K^2 + K^4 \right)\right) , & p=0, \\
	\ds \pm k^{\text{ex}}\left( 1 + \frac{1}{12}\left(  \frac{\delta(\WHO;\mathbf{p})}{\epsilon(\WHO;\mathbf{p})} +1 \right) W^2  +\mc{O}\left( W^4 + K^{2p+2}\right)\right) , & p=1, 2, 3,\\
	\end{array}
	\right.
	\end{align}
	\renewcommand{\arraystretch}{1}
	in the case of $W\ll 1$ and $K\ll 1$, or
	\renewcommand{\arraystretch}{2} 
	\begin{align}
	\label{eq:disp_TP_DG_AL2}
	k^{\text{TP}}_{\text{DG}^{\text{ phys}},p}
	= \left\{\begin{array}{ll}
	\ds \pm k^{\text{ex}}\left( 1 + \frac{1}{12}\left(  \frac{\delta(\WHO;\mathbf{p})}{\epsilon(\WHO;\mathbf{p})} +1 + \frac{\epsilon(\WHO;\mathbf{p})}{2\epsilon_{\infty}\nu^2} \right) W^2  +\mc{O}\left( W^4 \right)\right) , & p=0, \\
	\ds \pm k^{\text{ex}}\left( 1 + \frac{1}{12}\left(  \frac{\delta(\WHO;\mathbf{p})}{\epsilon(\WHO;\mathbf{p})} +1 \right) W^2  +\mc{O}\left( W^4 \right)\right) , & p=1, 2, 3.\\
	\end{array}
	\right.
	\end{align}
	\renewcommand{\arraystretch}{1}
	in the case of $W\ll 1$ and with a fixed CFL number $\nu.$

	\item When using the upwind flux, i.e. $\alpha=0$, $\beta_{1}=\frac{1}{2\sqrt{\epsilon_{\infty}}}$ and $\beta_{2}=\frac{\sqrt{\epsilon_{\infty}}}{2}$, there are only two solutions corresponding to the physical modes,
	\renewcommand{\arraystretch}{2} 
	\begin{align}
	\label{eq:disp_TP_DG_UP1}
	k^{\text{TP}}_{\text{DG}^{\text{ phys}},p}
	= \left\{\begin{array}{ll}
	\ds \pm k^{\text{ex}} \left(  1 +\frac{1}{2} i B +\mc{O}\left(W^2 + K^2 \right) \right) , & p=0, \\
	\ds \pm k^{\text{ex}} \left( 1 + \frac{1}{12}\left(  \frac{\delta(\WHO;\mathbf{p})}{\epsilon(\WHO;\mathbf{p})} +1 \right) W^2 +\mc{O}\left(W^4 +i K^{2p} B \right) \right), & p=1, 2, 3, \\
	\end{array}
	\right.
	\end{align}
	\renewcommand{\arraystretch}{1}
	which can be rewritten as 
	\renewcommand{\arraystretch}{2} 
	\begin{align}
	\label{eq:disp_TP_DG_UP2}
	k^{\text{TP}}_{\text{DG}^{\text{ phys}},p}
	= \left\{\begin{array}{ll}
	\ds \pm k^{\text{ex}} \left( 1 +i \frac{b/\omega}{2\sqrt{\epsilon_{\infty}}\nu} W +\mc{O}\left(W^2 \right) \right) , & p=0, \\
	\ds \pm k^{\text{ex}} \left( 1 + \frac{1}{12}\left(  \frac{\delta(\WHO;\mathbf{p})}{\epsilon(\WHO;\mathbf{p})} +1 \right) W^2 +\mc{O}\left( i W^3 \right) \right), & p=1, \\
	\ds \pm k^{\text{ex}} \left( 1+ \frac{1}{12}\left(  \frac{\delta(\WHO;\mathbf{p})}{\epsilon(\WHO;\mathbf{p})} +1 \right) W^2 +\mc{O}\left( W^4 \right) \right), & p=2, 3,\\
	\end{array}
	\right.
	\end{align}
	\renewcommand{\arraystretch}{1} 
	in the case of $W\ll 1$ and with a fixed CFL number $\nu.$ 
\end{itemize}

We have  similar conclusions as those for the fully discrete leap-frog DG schemes, except that the leading error terms for high order schemes come from the fully implicit time discretization \eqref{eq:dis_TP}.

\subsection{Comparison among fully discrete DG schemes}

In our previous work \cite{bokil2017energy}, we have proved that the  fully discrete DG schemes based on the trapezoidal rule is unconditionally stable and the leap-frog schemes are conditionally stable. Following the proof in \cite{bokil2017energy}, we can find $\nu^{p}_{max}$ such that under the condition $\nu\leq \nu^{p}_{max}$, the leap-frog schemes using $P^{p}$ space   are stable .  Those $\nu^{p}_{max}$ values for $p=0,\ldots,3, \infty$ are listed in Table 6.1. For comparison, we also list the CFL condition for FD scheme for various $M$ values in the same table. We can see that the CFL number for DG scheme is much smaller than that for FD scheme, particularly for high order case.

\renewcommand{\arraystretch}{1.5}
\begin{table}[H]
	\label{tab:CFL}
	\centering
	\caption{$\nu^{p}_{max}$ for DG scheme and $\nu^{2M}_{max}$ for FD scheme.}
	\begin{tabular}{c|cccccc}
		& $p=0$ & $p=1$ & $p=2$ &  $p=3$ & &   $p\rightarrow\infty$ \\ \hline
		DG-CE	 & 1 &  0.211325 & 0.101287  &  0.0605268 & &  0 \\
		DG-AL & 1 &  0.192450 & 0.089115  &  0.0521629  & &  0\\
		DG-UP		 & 1 &  0.211325 & 0.101287  &  0.0605268 &  & 0 \\\hline\hline
		&  $M=1$  & $M=2$  & $M=3$ & $M=4$ & $M=5$ & $M\rightarrow\infty$  \\\hline
		FD & 1  &  0.857143  & 0.805369  &  0.777418 & 0.759479 &  0.636620  \\
	\end{tabular}
\end{table}
\renewcommand{\arraystretch}{1}

 In Figure \ref{Fig:Phase_Error_DG_fully1}, we plot the relative phase errors of fully discrete DG schemes with leap-frog discretization for   $W_1= \pi/30, \pi/300$   using    material parameters \eqref{eq:parameter}. We can observe that the overall behavior of the plot with  $W_1= \pi/30$ is quite different from the FD plots and the plots obtained with $W_1= \pi/300$, and the magnitude of the errors is very large. This phenomenon results from $\ds\omega_1 h = \frac{1}{\sqrt{\epsilon_{\infty}}\,\nu}W_1$ and the tiny CFL numbers restricted by the stability condition which makes the mesh size $h$   extra large. We conclude that the small CFL number is one disadvantage of high order DG schemes.  When comparing the figures obtained with $W_1= \pi/300$ using the three numerical fluxes, it is clear that the alternating flux has the smallest error, while the error obtained by the upwind flux is the largest. The overall dependence of the error on $\WH{\omega}$ is very similar to those from the FD schemes.

\begin{figure}[h]
	\centering
	\includegraphics[scale= 0.28] {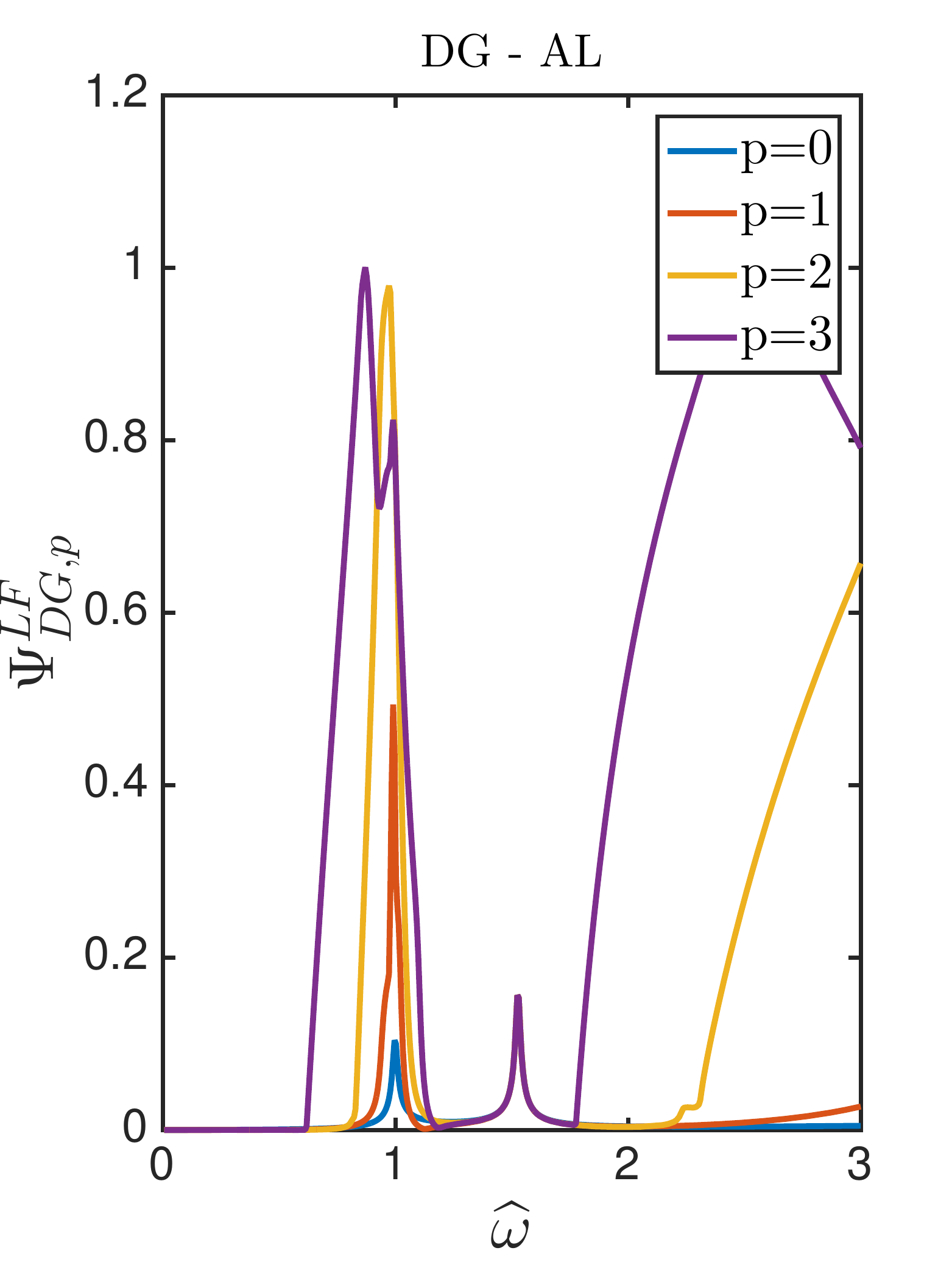}
	\includegraphics[scale= 0.28] {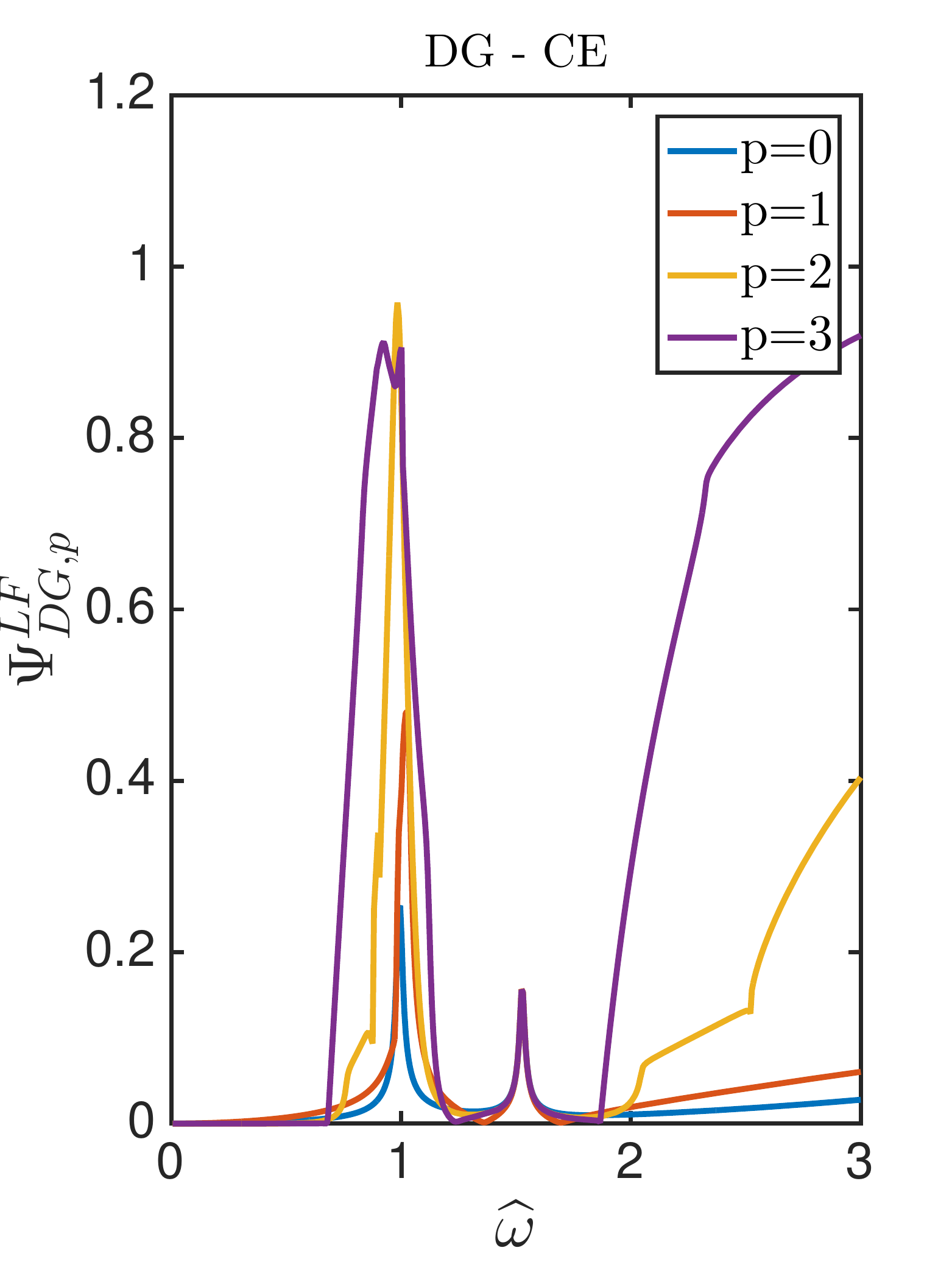}
	\includegraphics[scale= 0.28] {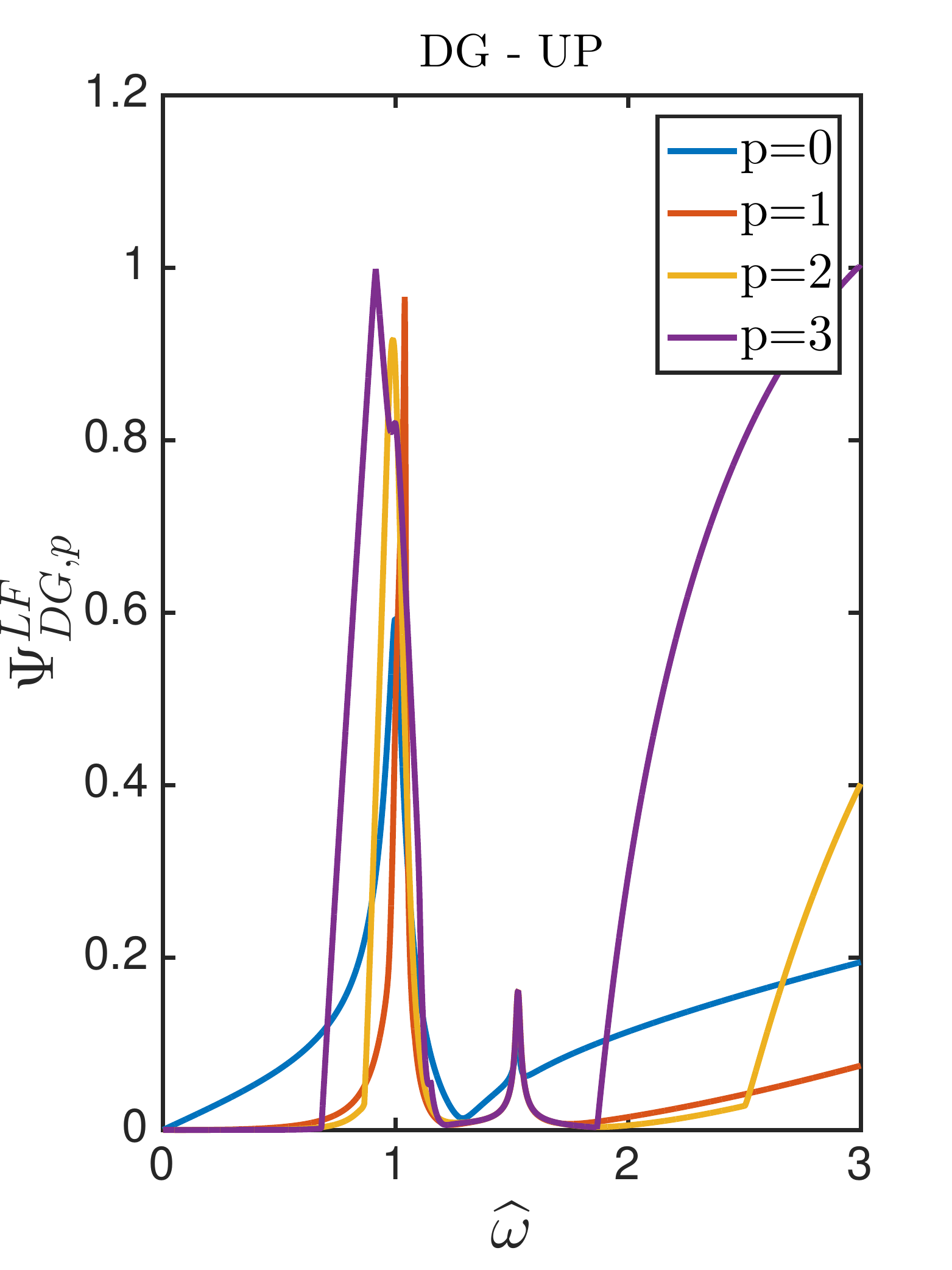}\\
	\includegraphics[scale= 0.28] {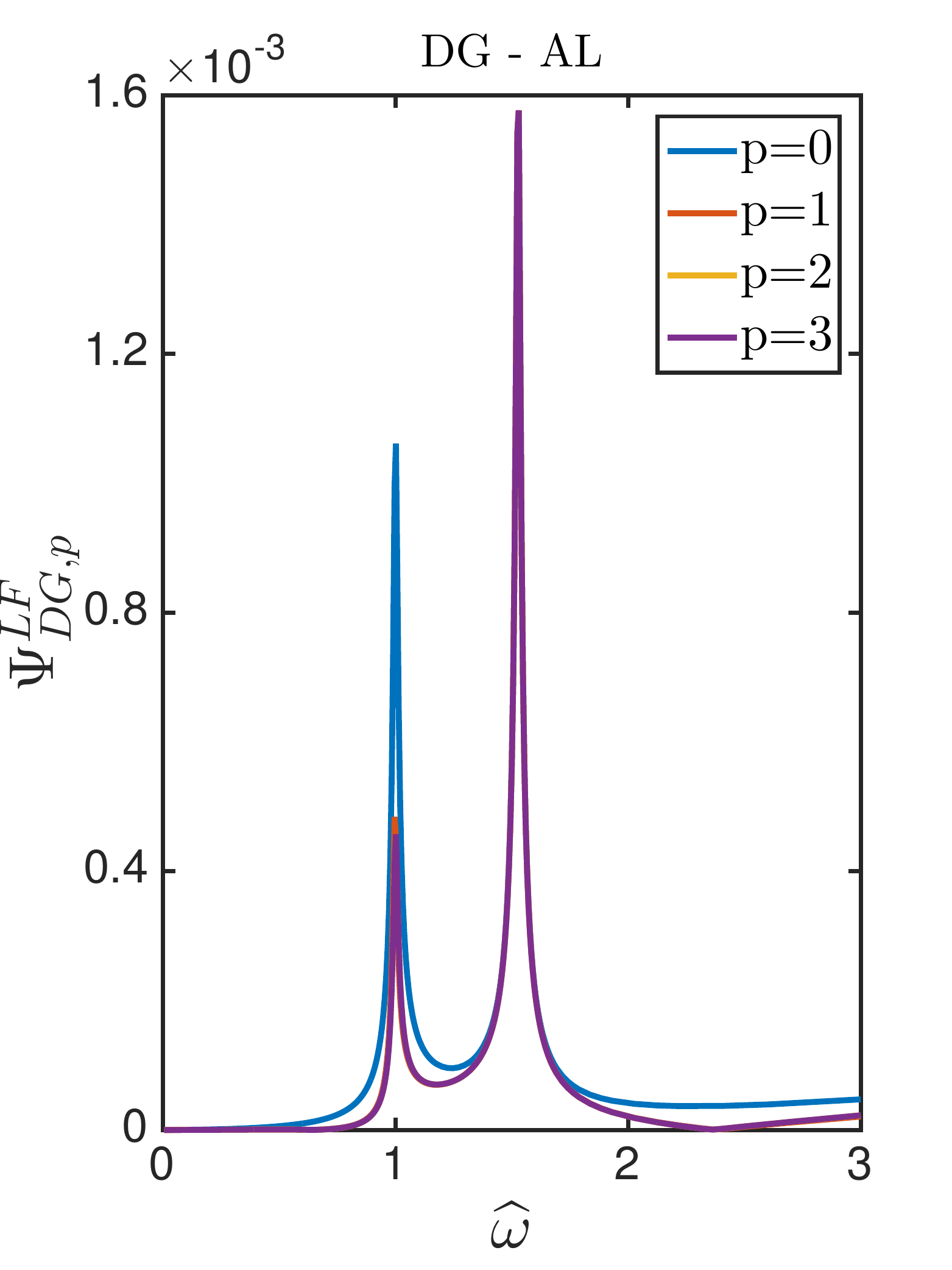}
	\includegraphics[scale= 0.28] {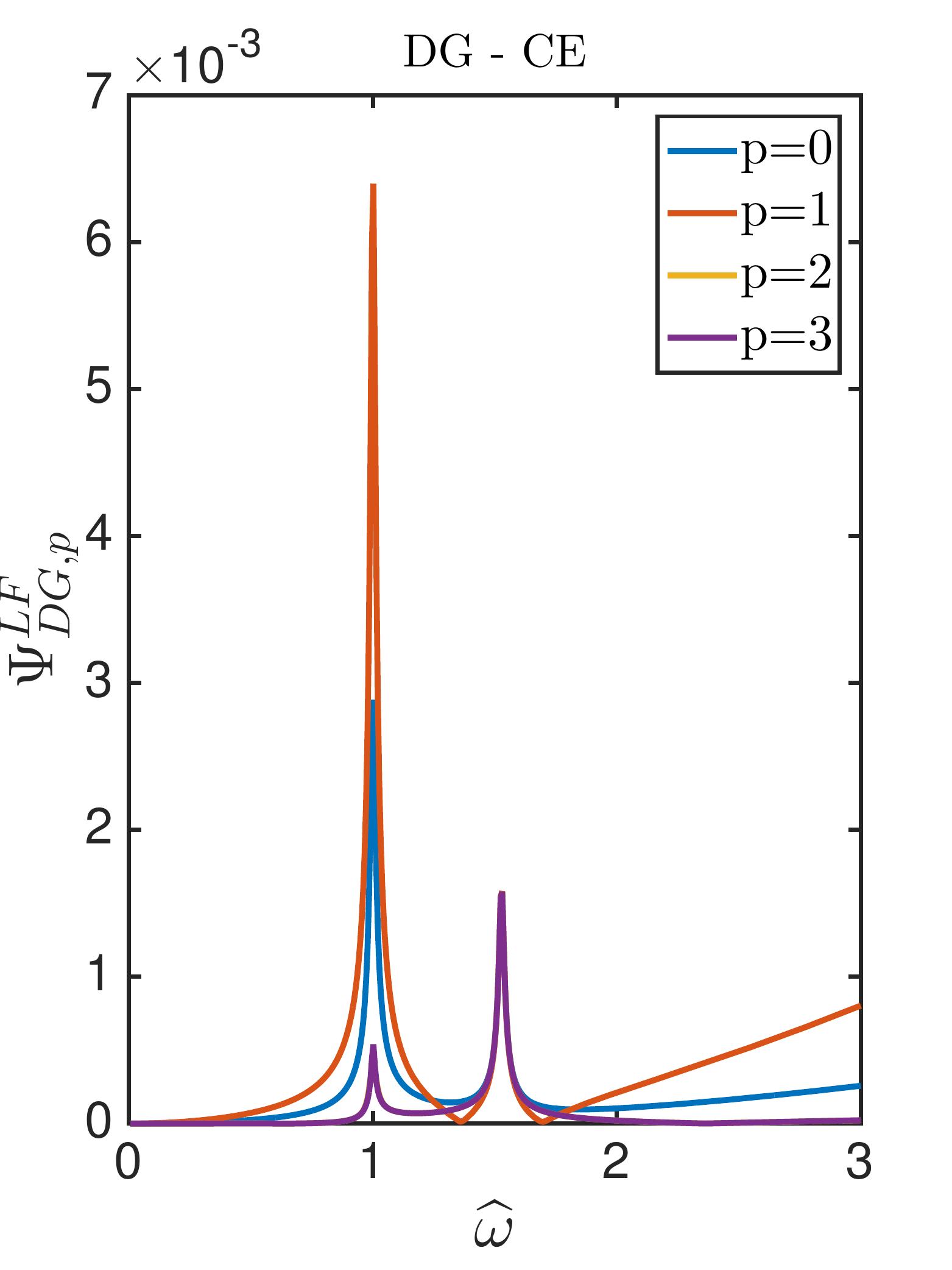}
	\includegraphics[scale= 0.28] {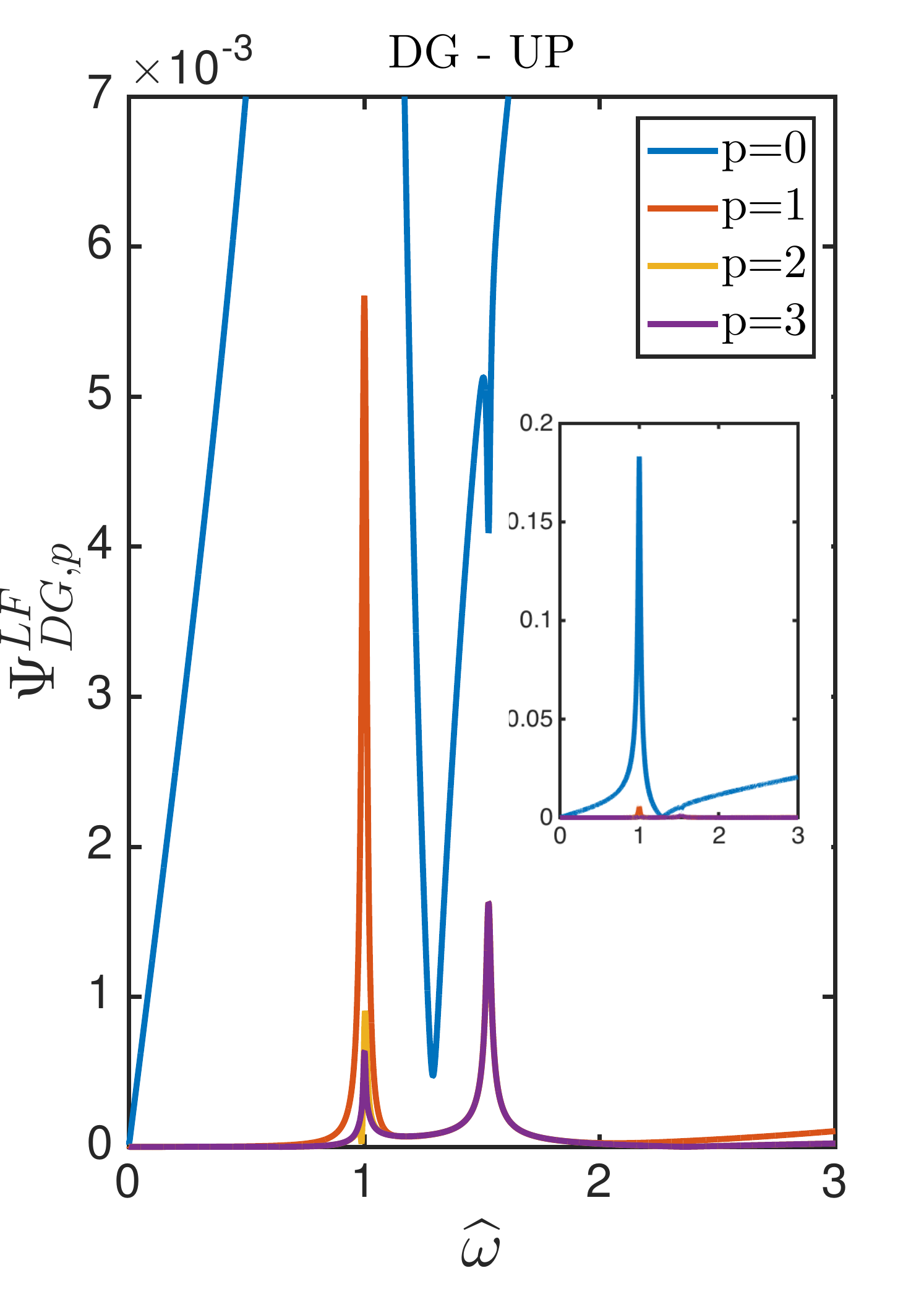}
	\caption{The relative phase error in physical modes of the fully discrete DG schemes with leap-frog time discretization, using $\nu/\nu_{max}^{p}=0.7$. First row: $W_1=\pi/30$ ; second row:   $W_1=\pi/300.$ }
	\label{Fig:Phase_Error_DG_fully1}
\end{figure}

Next, we consider the unconditionally stable DG scheme with trapezoidal rule and varying CFL numbers. Figure \ref{Fig:Phase_Error_DG_fully2} shows the contour plots of the dispersion error at $\WH{\omega}=1$ with $(W_1,\omega_{1} h) \in[0.05,0.3]\times[0.01,0.1]$.  
It is observed that DG-AL with $p\geq1$ have horizontal contour lines, indicating dispersion errors are dominated by temporal ones. 
In comparison, DG-UP and DG-CE have horizontal contour lines when $p\geq2$. The values of numerical dispersion errors obtained by high order DG schemes are very similar, which also illustrates the dominant role of temporal errors. This observation is consistent with our theoretical analysis.

\begin{figure}
	\centering	
	\includegraphics[scale= 0.18] {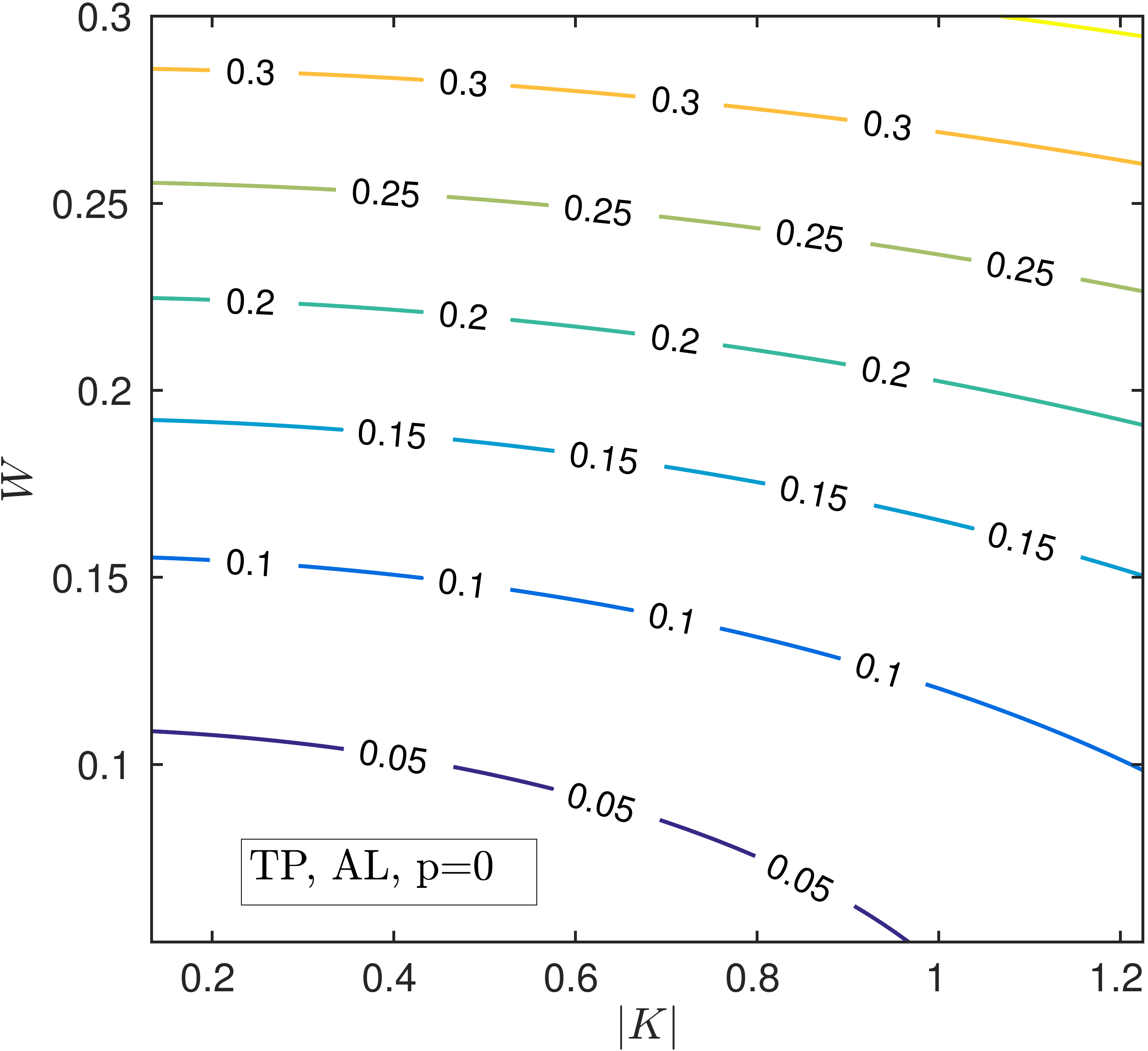}
	\includegraphics[scale= 0.18] {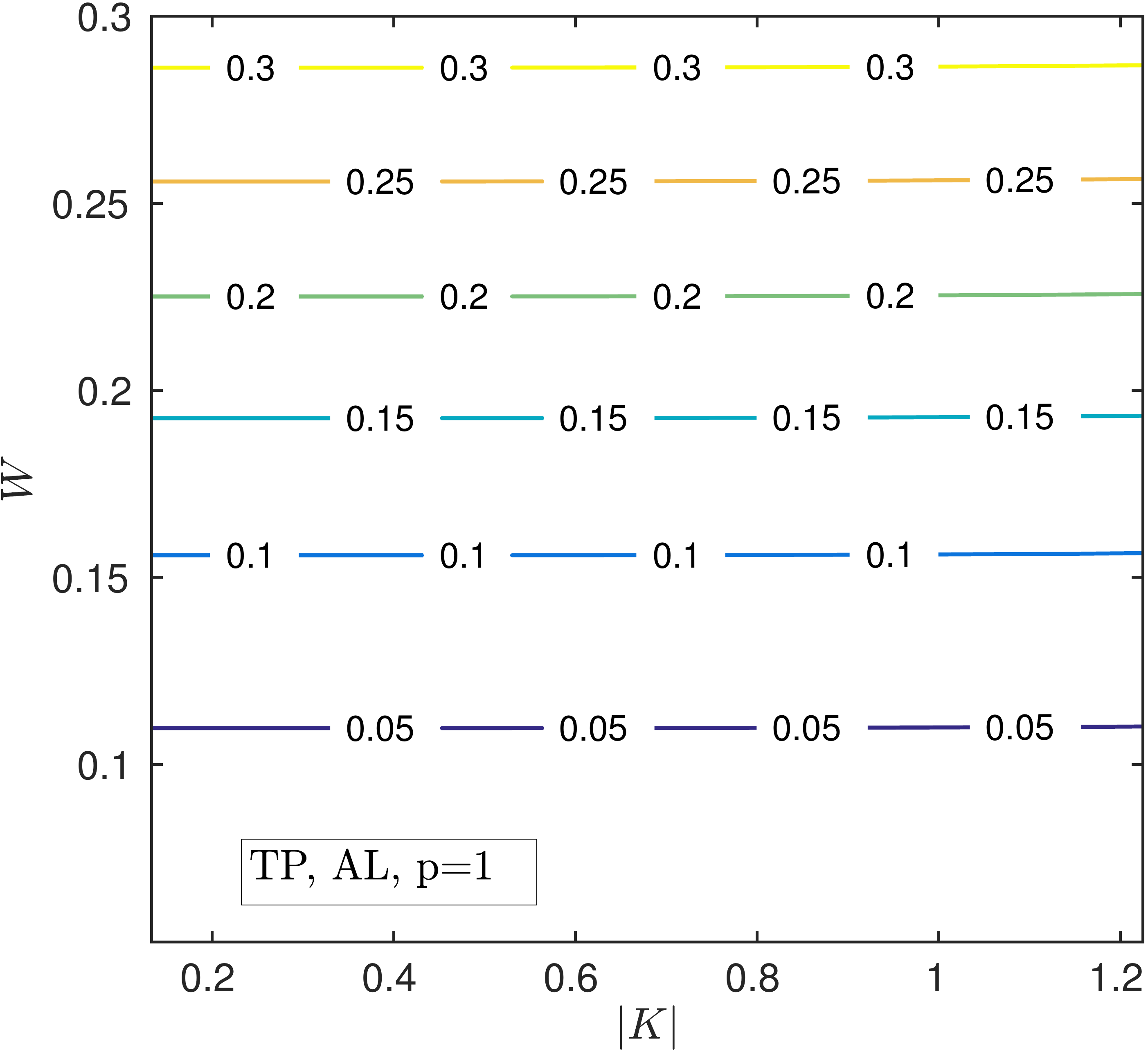}
	\includegraphics[scale= 0.18] {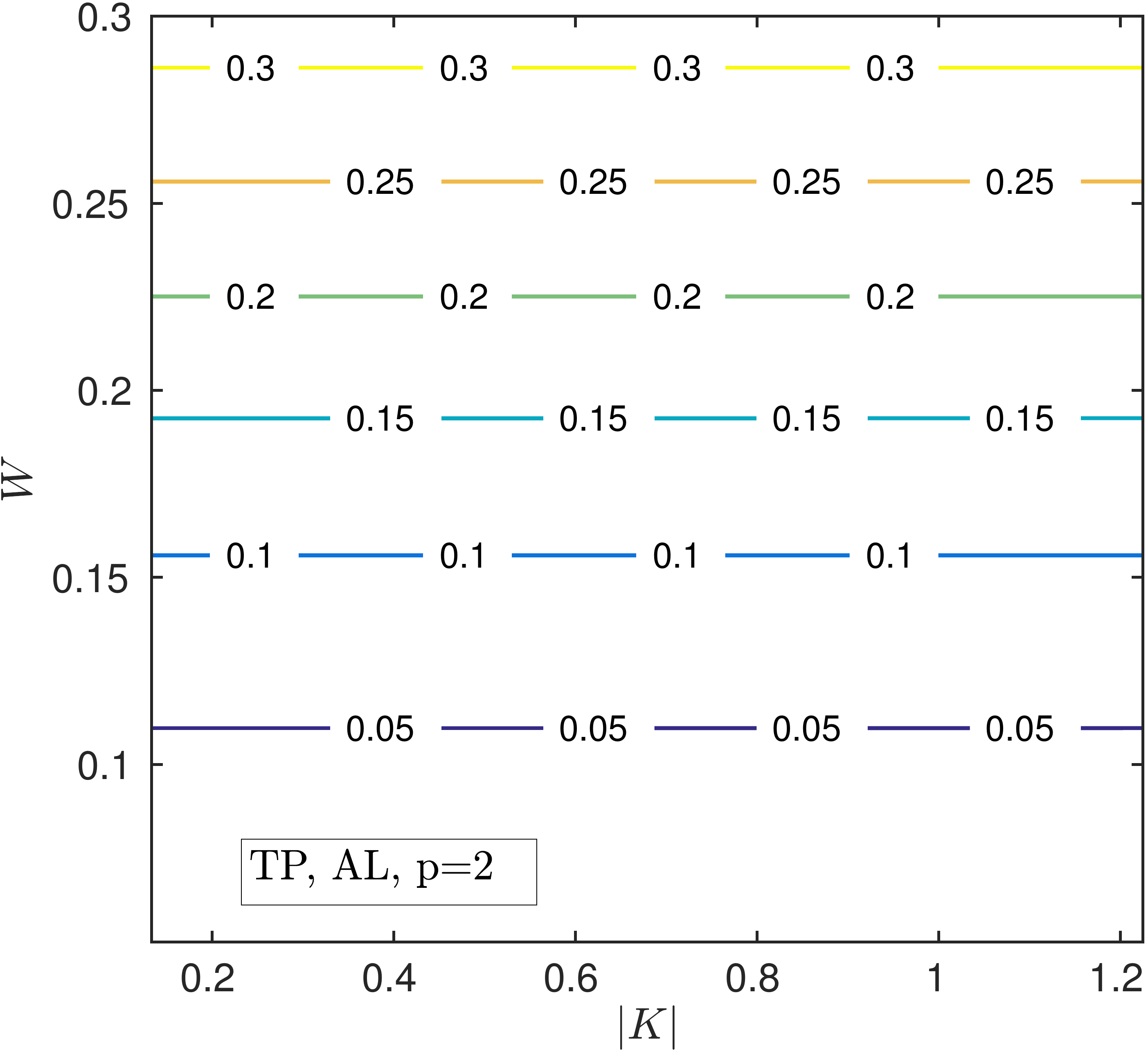}\\
	\includegraphics[scale= 0.18] {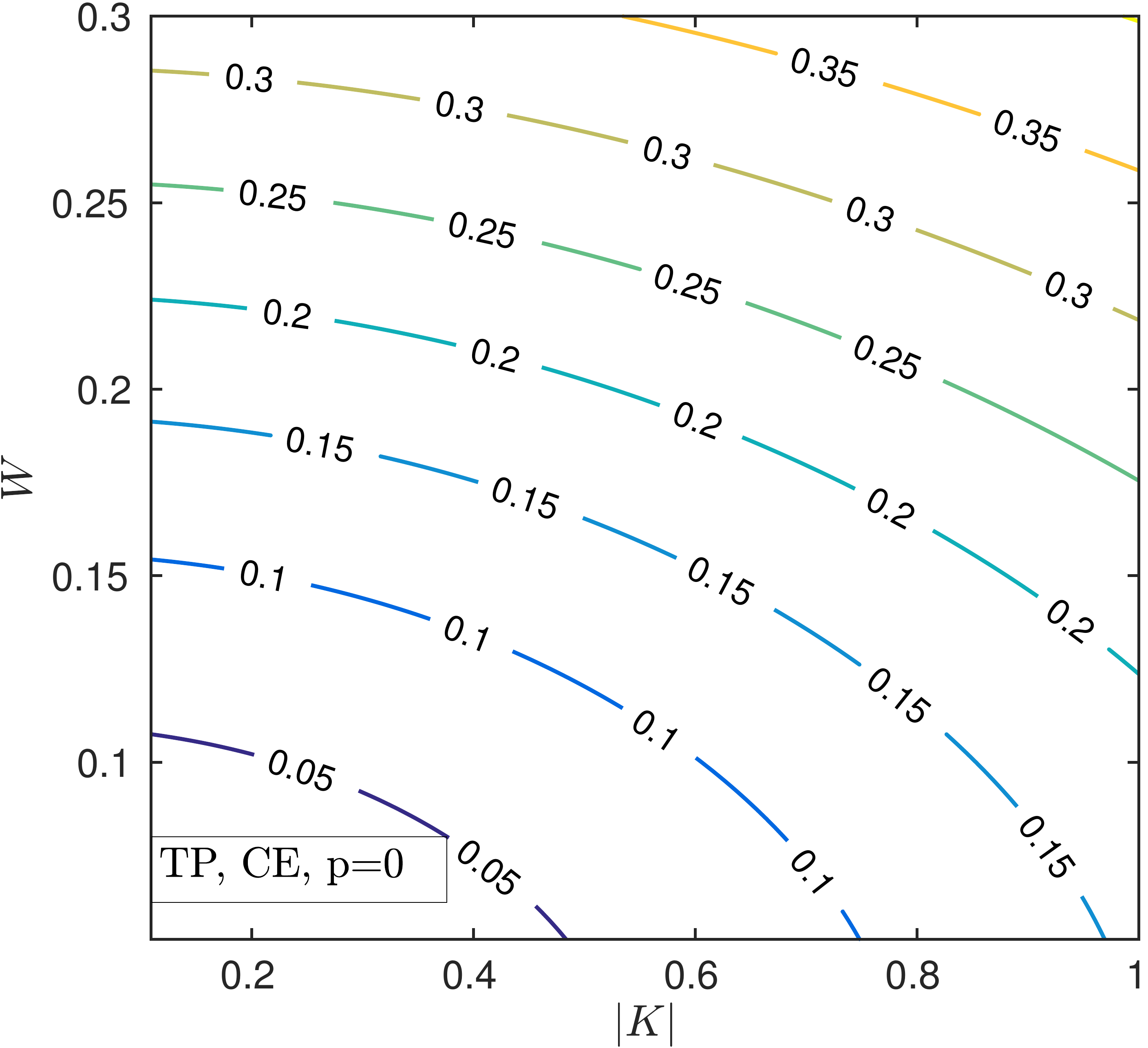}
	\includegraphics[scale= 0.18] {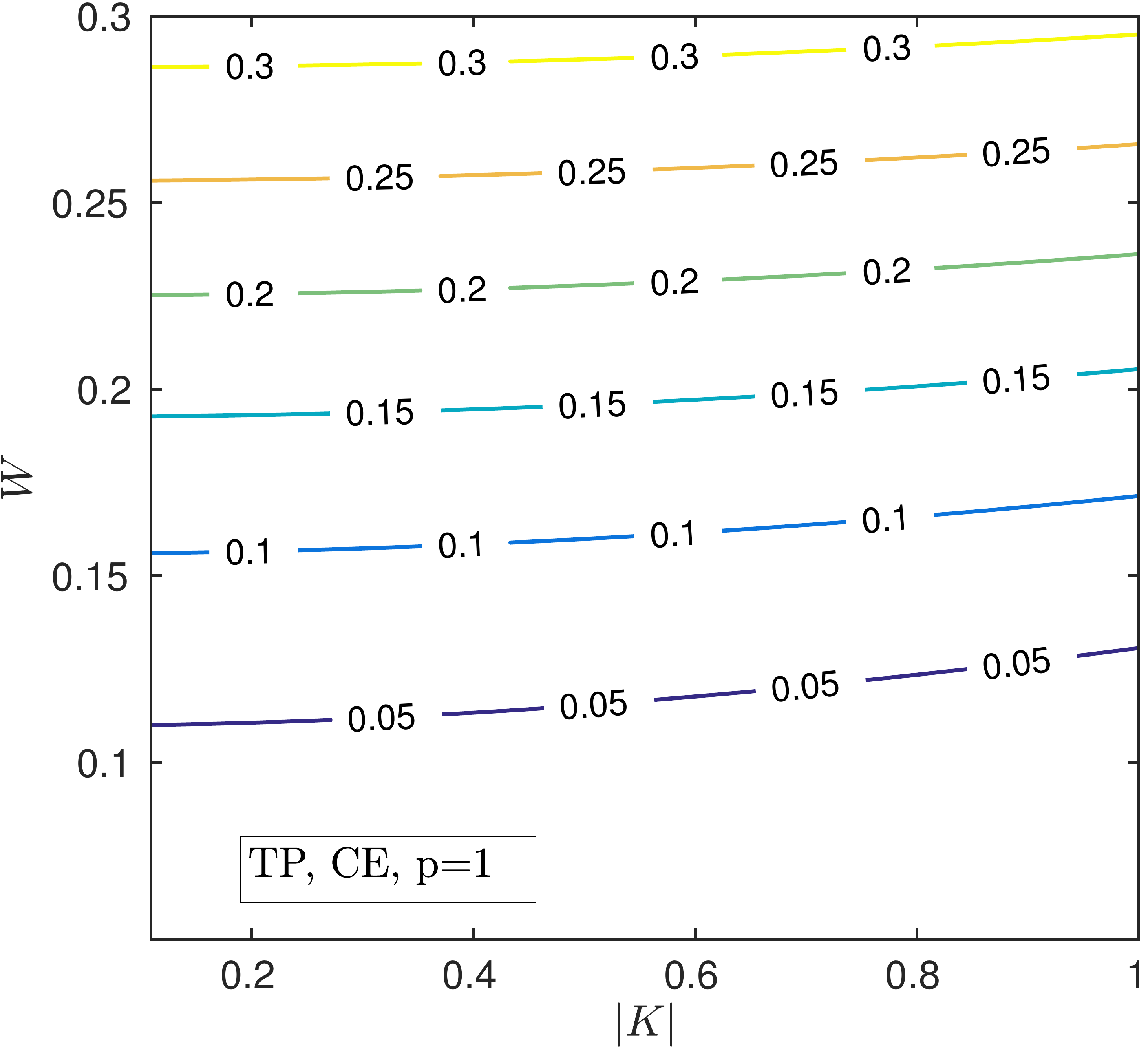}
	\includegraphics[scale= 0.18] {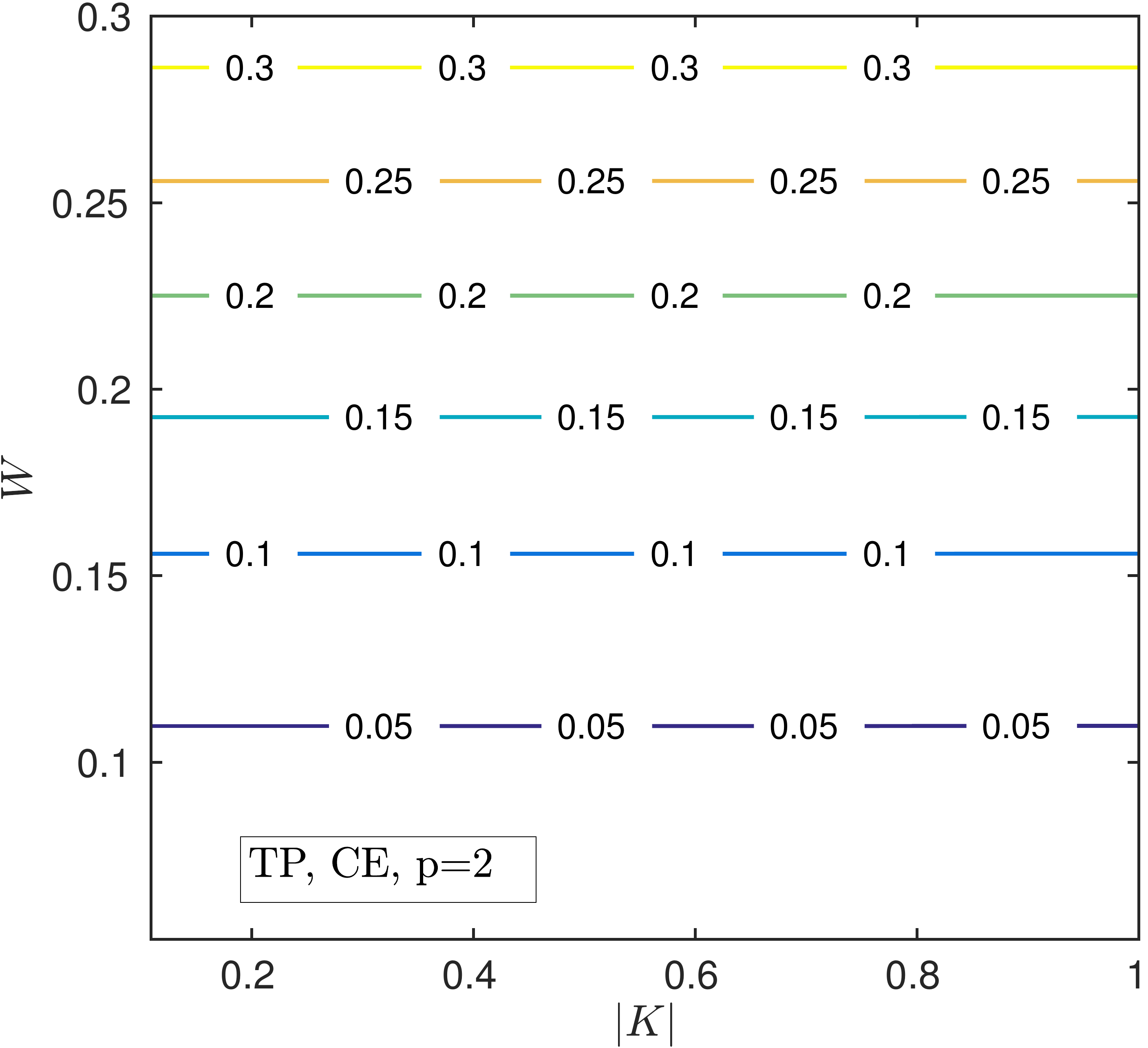}\\
	\includegraphics[scale= 0.18] {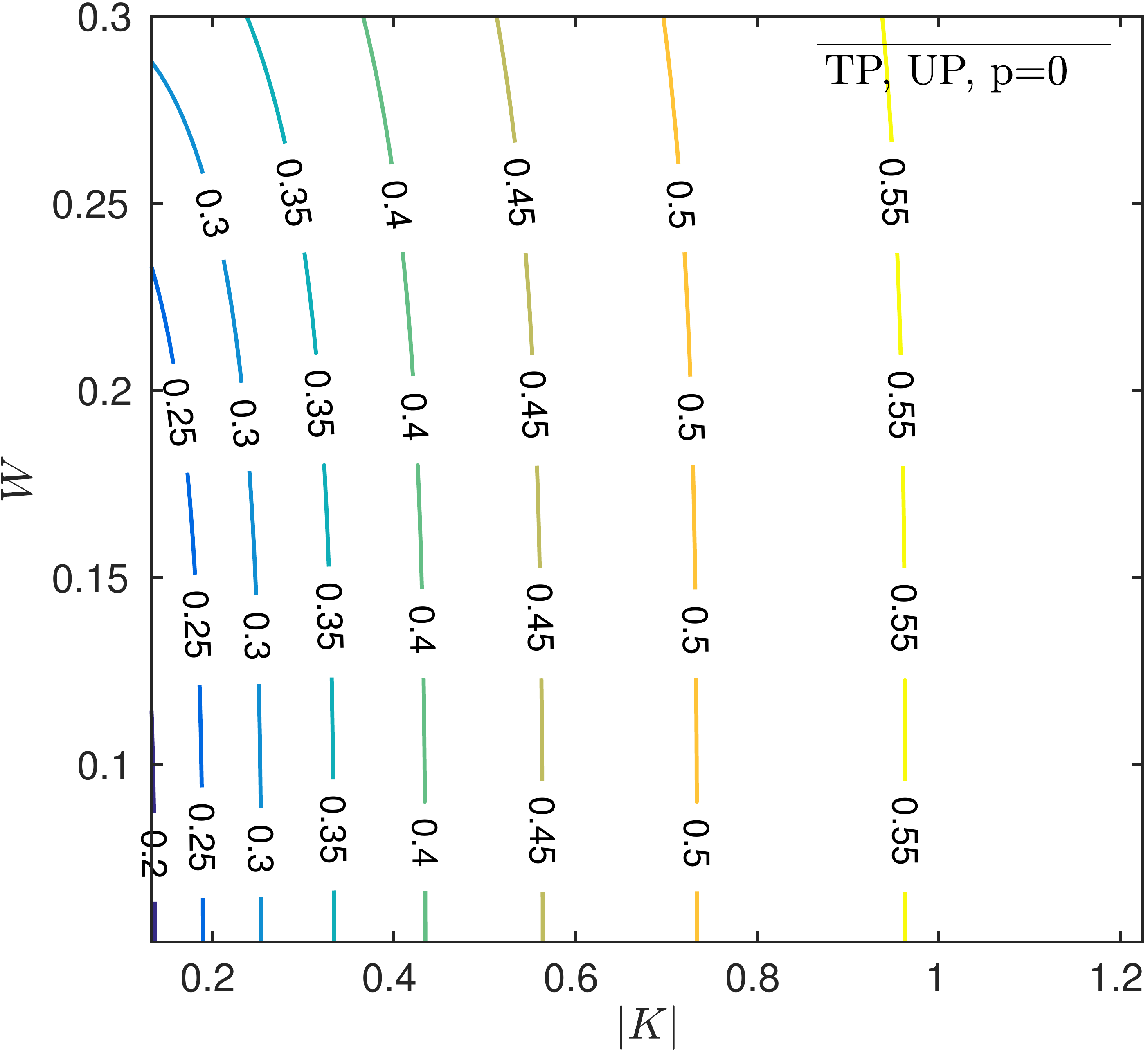}
	\includegraphics[scale= 0.18] {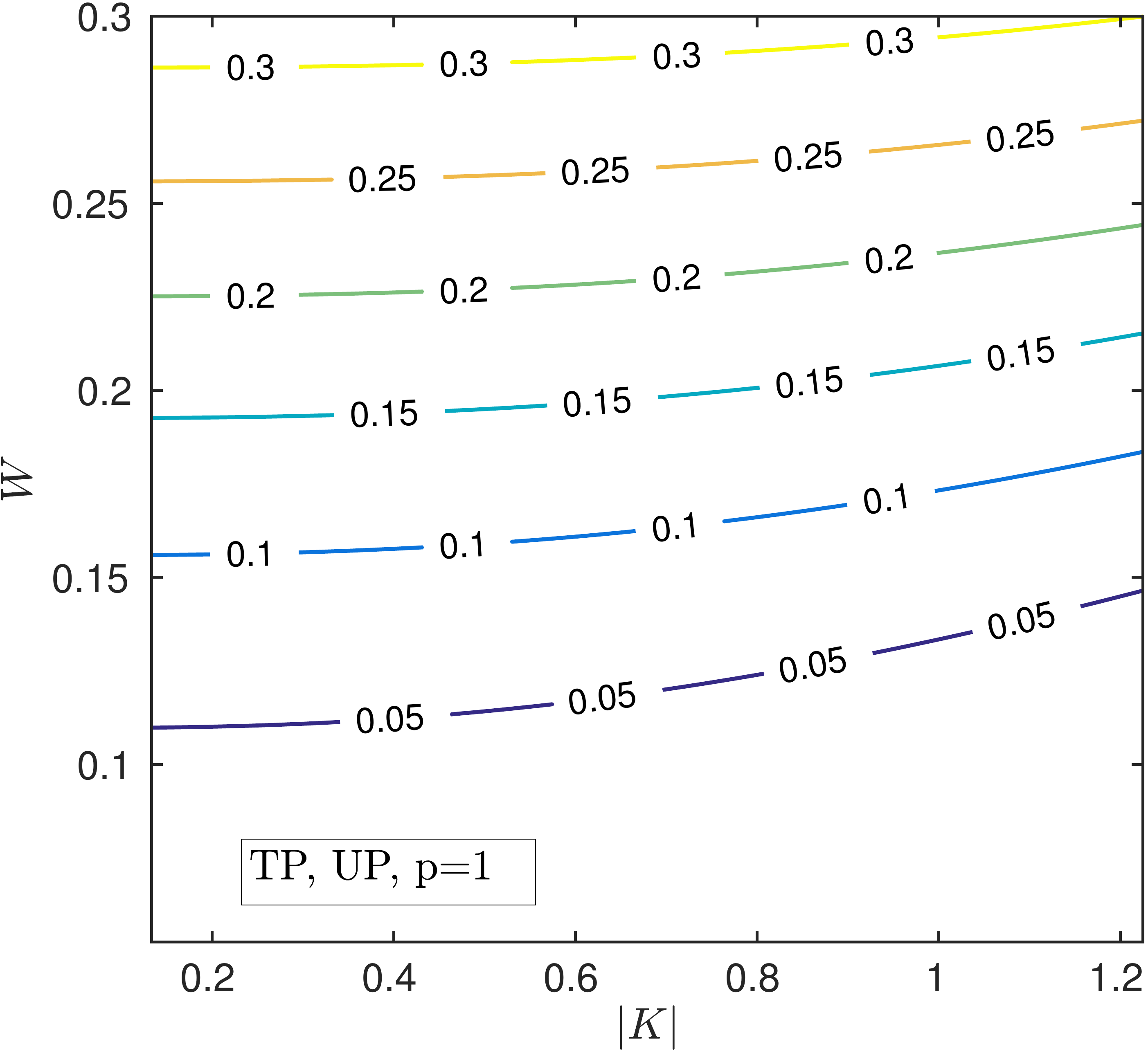}
	\includegraphics[scale= 0.18] {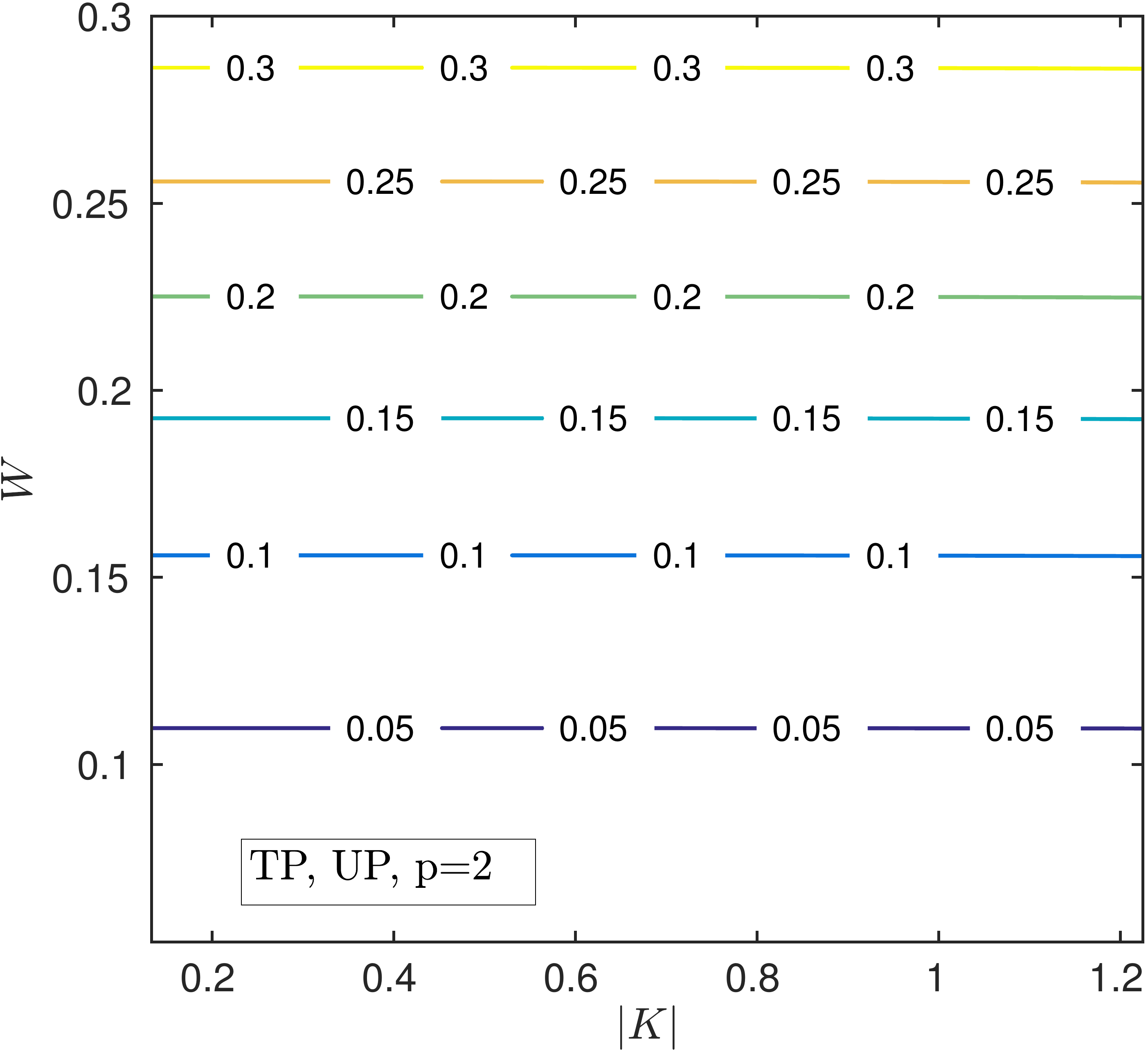}
	\caption{The contour plot of relative phase error of fully discrete DG schemes for the physical modes with trapezoid rule. $\WHO=1$. First row: DG-AL; second row: DG-CE; third row: DG-UP.}
	\label{Fig:Phase_Error_DG_fully2}
\end{figure}

\section{Benchmark on Physical Quantities}
\label{sec:numerical}

In this section, we will  verify the performance of the finite difference and discontinuous Galerkin methods by plotting quantities that are important for wave propagation such as the normalized ratio between the numerical and exact phase velocity (also refractive index); normalized attenuation constant; normalized energy velocity; and normalized group velocity, to validate the performance of the numerical methods (see \cite{gilles2000comparison}). The model parameters in \eqref{eq:parameter} are used in computing all the quantities and plots below.

We first define $\psi$, given as 
\begin{align}
\displaystyle 
\label{def:reflex}
\psi = \frac{k}{\omega} = \sqrt{\epsilon(\WHO;\mathbf{p}) }.
\end{align}
We note that $\psi$ is the complex index of refraction of the medium, whose real part is the real refractive index of the medium, whereas the imaginary part is related to the absorption or extinction coefficient \cite{oughstun1988velocity}. We use $\Re$ and $\Im$ to denote the real   and the imaginary parts of a complex number. Let the superscripts $E$ and $ N$ denote the value of a quantity related to the exact solution of system \eqref{eq:sys} and a numerical approximation, respectively. We have the following definitions (see \cite{gilles2000comparison}):
\begin{itemize}
	\item \textbf{Normalized Phase Velocity}: We consider the ratio between the real parts of the exact and numerical phase velocities, with the phase velocity, $v_p$, defined as $v_{p}=\omega/k = 1/\psi $. We define 
	\begin{align}
	\label{eq:pha_vel}
	\text{Normalized Phase Velocity} 
	=  \frac{\Re(v_p^N)}{\Re(v_p^E)}.
	\end{align}
	
	\item \textbf{Normalized Attenuation Constant}: We consider the ratio between the imaginary parts of the exact and numerical $\psi$, which is also the ratio between the imaginary parts of the exact and numerical indices of refraction. We define 
	\begin{align}
	\label{eq:att_con}
	\text{Normalized Attenuation Constant} = \frac{\Im\left( \psi^N \right)}{\Im\left( \psi^E \right)}.
	\end{align}
	
      \item \textbf{Normalized Energy Velocity}: The {\it velocity of energy transport} of a (monochromatic) plane-wave
field is an important concept of wave propagation in a dispersive medium. In \cite{oughstun1988velocity} this velocity is defined as a ratio of the time-average value of the Poynting vector to the total time-average electromagnetic energy density stored in both the field
and the medium. The normalized energy velocity is a quantity that is defined
        (see \cite{gilles2000comparison, oughstun1988velocity})
        as a function of the real and imaginary parts of the quantity $\psi$ given as 
	\[\displaystyle
	\text{Energy Velocity} = \left[ \Re(\psi) 
	+ \frac{\left( \Re\left( \psi^2\right)  - \epsilon_s\right) \left( \Re\left( \psi^2\right)  - \epsilon_\infty\right)
		+ (\Im\left( \psi^2 \right))^2}{\left( \epsilon_s - \epsilon_\infty\right) \Re\left( \psi \right)}
	\right]^{-1}.
	\]
Based on the definition of the energy transport velocity, we define the ratio between the exact and numerical energy transport velocity to be the normalized energy velocity as
        \begin{align}
	  \label{eq:ene_vel}
          	\text{Normalized Energy Velocity} = \frac{\text{Energy Velocity}\,^N }{\text{Energy Velocity}\,^E }.
	\end{align}
\item \textbf{Normalized Group Velocity}: We define the normalized group velocity to be the real part of the ratio of group velocities of the exact and numerical solutions. We have 
	\begin{align}
	\label{eq:grp_vel}
	\text{Normalized Group Velocity} = \Re \left( \frac{v_g^N}{v_g^E} \right),
	\end{align}
	where the group velocity is defined by $\displaystyle v_g = \frac{\partial \omega}{\partial k}$. Here, both $v^{E}_g$ and $v^{N}_g$ are obtained numerically by  
	$$\displaystyle (v_g)^{-1} = \frac{\partial k}{\partial \WH{\omega}} \frac{\partial \WH{\omega}}{\partial \omega} \approx \frac{k(\WH{\omega}+0.001)-k(\WH{\omega})}{0.001} \frac{\partial \WH{\omega}}{\partial \omega}.$$
	\end{itemize}

In Figures \ref{Fig:phy_FD1} and \ref{Fig:FD_TP_phy}, we plot the four physical quantities defined in \eqref{eq:pha_vel}-\eqref{eq:grp_vel} for the leap-frog and trapezoidal FD schemes in various ranges of values for $\WHO$: below resonance ($\WHO < 1$), near resonance ($\WHO \approx 1$), at the upper edge of the medium absorption band ($\WHO \approx 1.527$), and far above resonance ($\WHO >3$). Figure \ref{Fig:phy_FD1} offers excellent agreement with the plots in \cite{gilles2000comparison} for the (2,2) Yee scheme (leap-frog FDTD scheme with $M=1$) and a (2,4) leap-frog FDTD scheme ($M=2$). Both schemes have large errors at the resonance frequency and the upper edge of the medium absorption band $\WHO=\sqrt{\epsilon_{s}/\epsilon_{\infty}}$.  Higher order schemes have values for the physical quantities that are closer to 1, which indicates smaller dispersion error with increase in the spatial order of the scheme.   We note that, while the increase in spatial order reduces the four physical quantities near resonance for both the leap-frog and trapezoidal FDTD methods, there is virtually no change with spatial order at the upper edge of the medium absorption band. This is also true for the DG schemes.  A comparison between Figures \ref{Fig:phy_FD1} and \ref{Fig:FD_TP_phy} suggests that the main differences between the two temporal discretizations can be observed for frequencies below and far beyond resonance. For $\WHO<1,$ the plots obtained by the trapezoidal FDTD schemes are monotone. This is not the case for the leap-frog FDTD method as shown in Figure \ref{Fig:phy_FD1}. The results can be understood by comparing equations \eqref{eq:dis_LF} with \eqref{eq:dis_TP}. The leading error coefficients in the two time schemes are different, with one being monotone on $\WHO$ and the other not.  For high frequencies, the leap-frog FDTD scheme can no longer resolve frequencies beyond $14.8$, when the fields start to decay exponentially and have an increasing phase velocity. This number changes to around $10$ for the trapezoidal scheme, which shows different resolution offered by the two temporal schemes.  

In Figure  \ref{Fig:phy_DG_AL1}, we plot the four physical quantities defined in \eqref{eq:pha_vel}-\eqref{eq:grp_vel}, obtained by DG-AL scheme using trapezoidal time discretization with a fixed CFL number $\nu=0.7.$ This choice is made based on previous observations that the DG-AL performs the best among all three fluxes. The overall behaviors of the physical quantities for FD and DG schemes are very similar when comparing Figure \ref{Fig:FD_TP_phy} with Figure \ref{Fig:phy_DG_AL1}. The main difference lies in the last column for high frequencies. The increasing resolution in the higher order DG scheme is evident, while increasing order does not impact this much for FD schemes. Thus, high order DG schemes in space can have a better performance in resolving high frequencies when compared with the FD scheme using the same mesh size. Similar conclusion holds with leap frog time discretization, and the plots are omitted for brevity.

\begin{figure}
	\centering
	\includegraphics[scale=0.258]{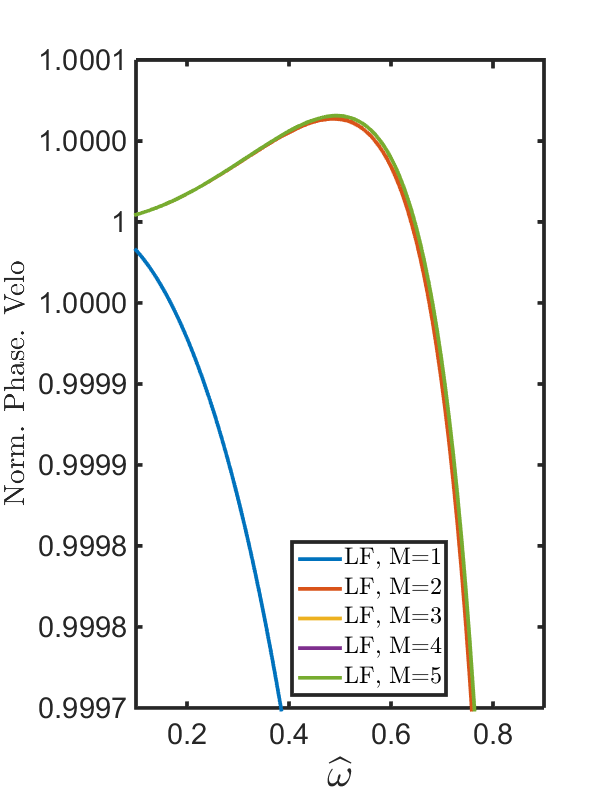}
	\includegraphics[scale=0.258]{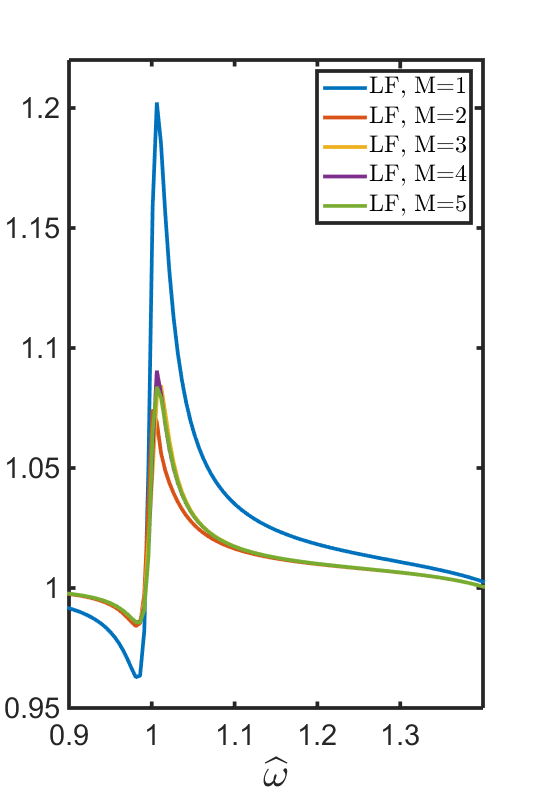}
	\includegraphics[scale=0.258]{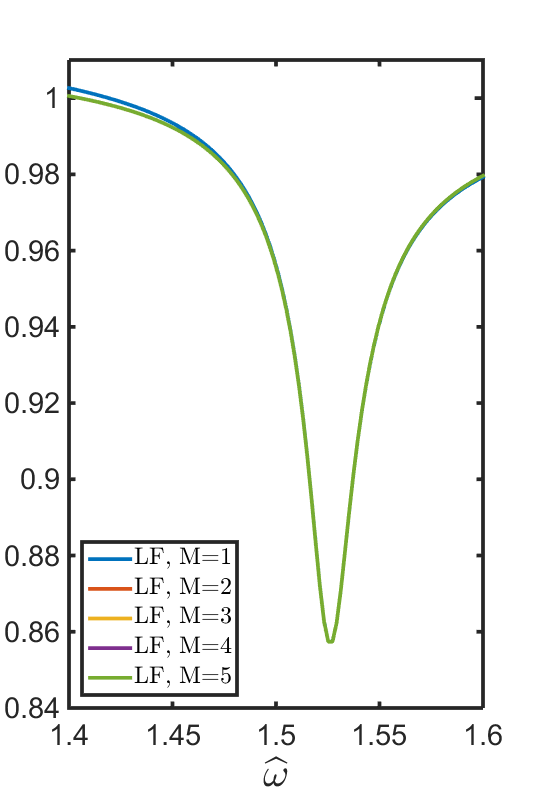}
	\includegraphics[scale=0.258]{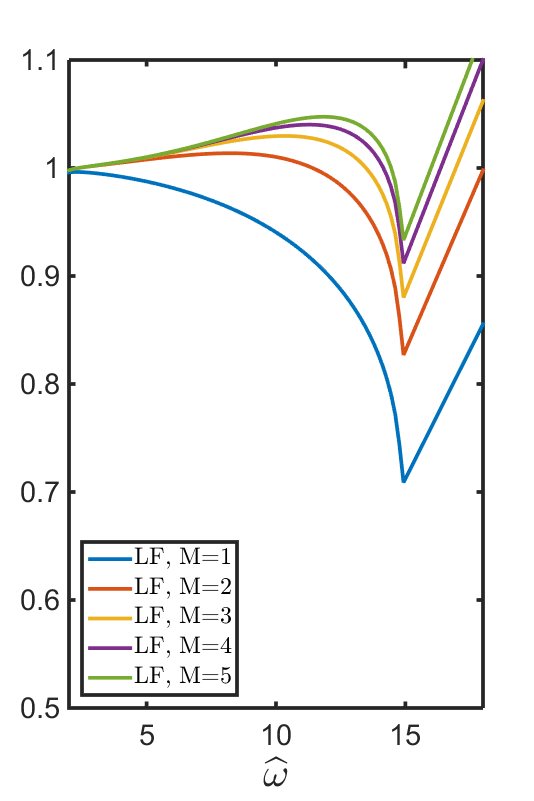} \\
	\includegraphics[scale=0.258]{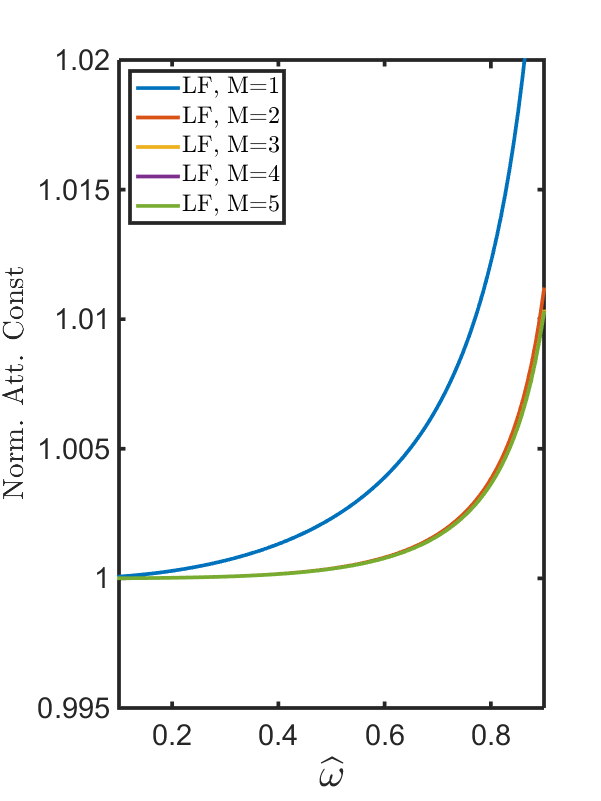}
	\includegraphics[scale=0.258]{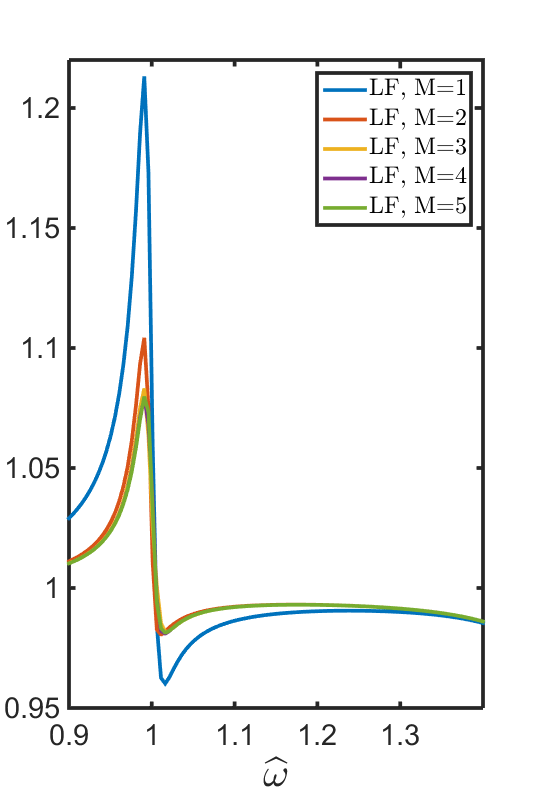}
	\includegraphics[scale=0.258]{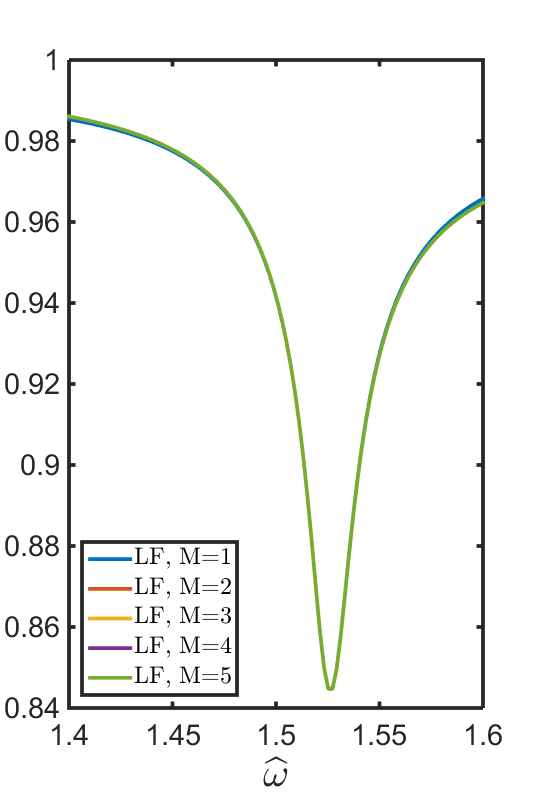}
	\includegraphics[scale=0.258]{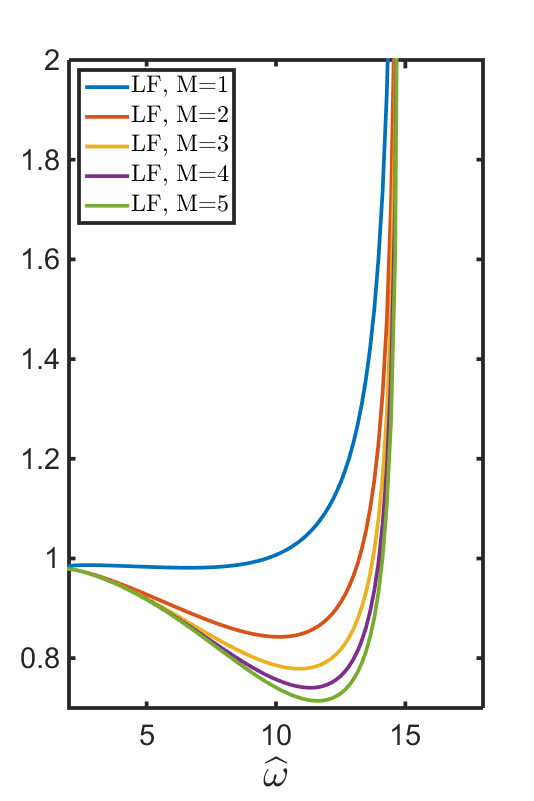} \\
	\includegraphics[scale=0.258]{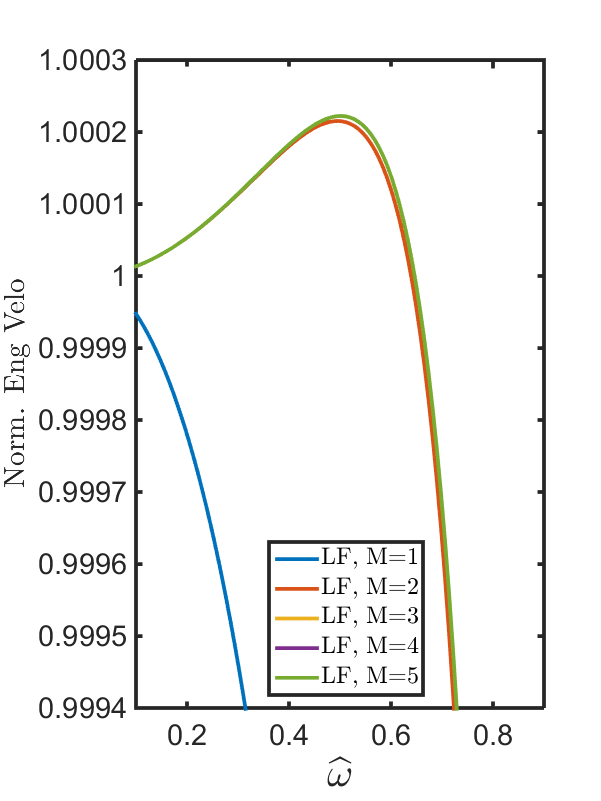}
	\includegraphics[scale=0.258]{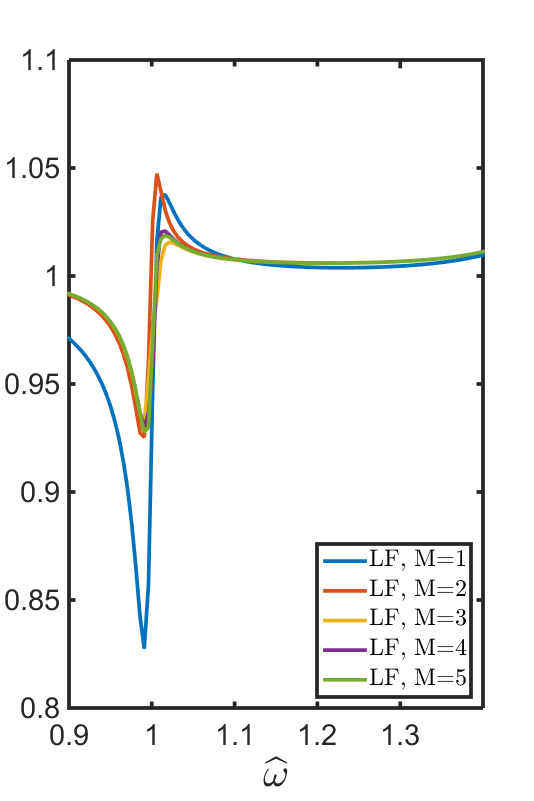}
	\includegraphics[scale=0.258]{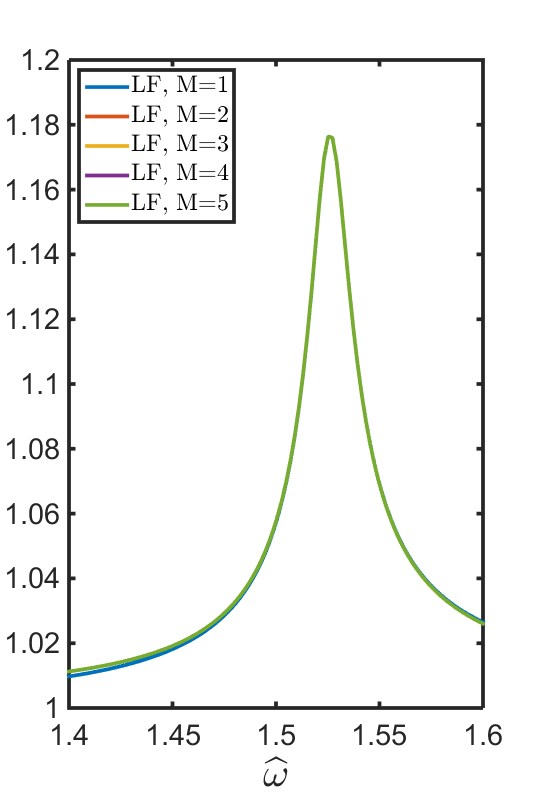}
	\includegraphics[scale=0.258]{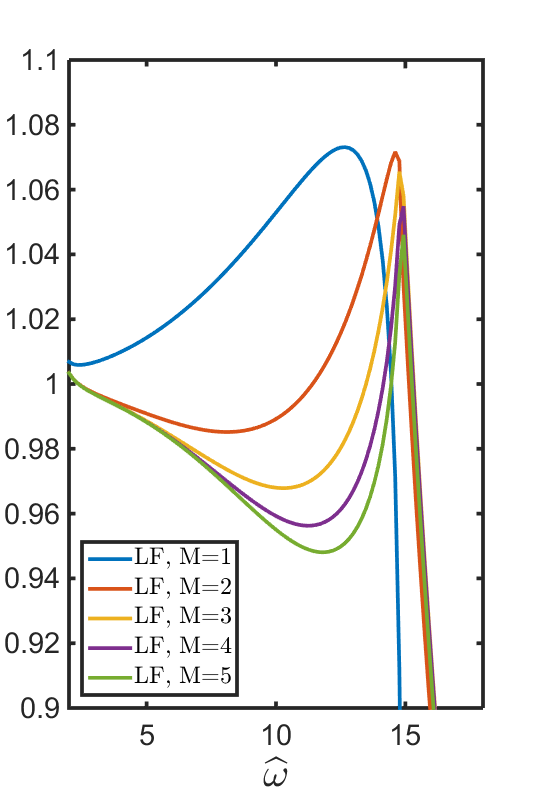} \\
	\includegraphics[scale=0.258]{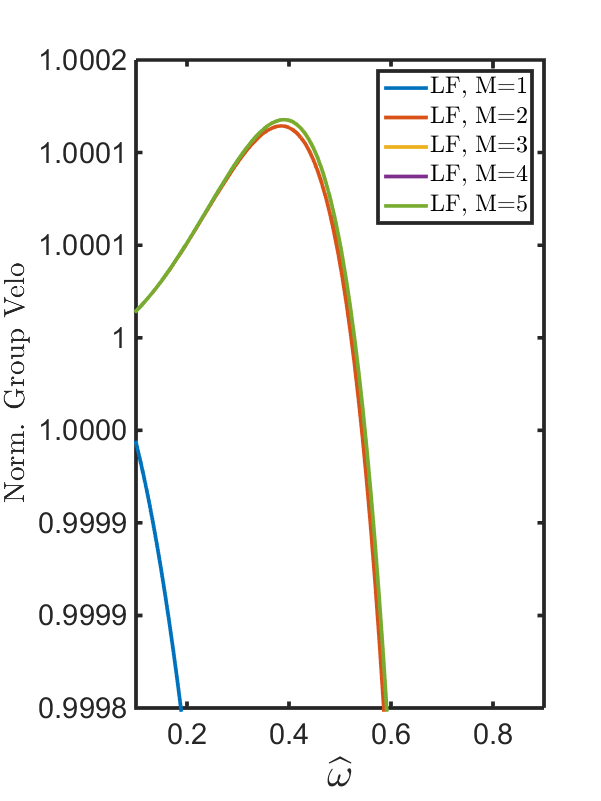}
	\includegraphics[scale=0.258]{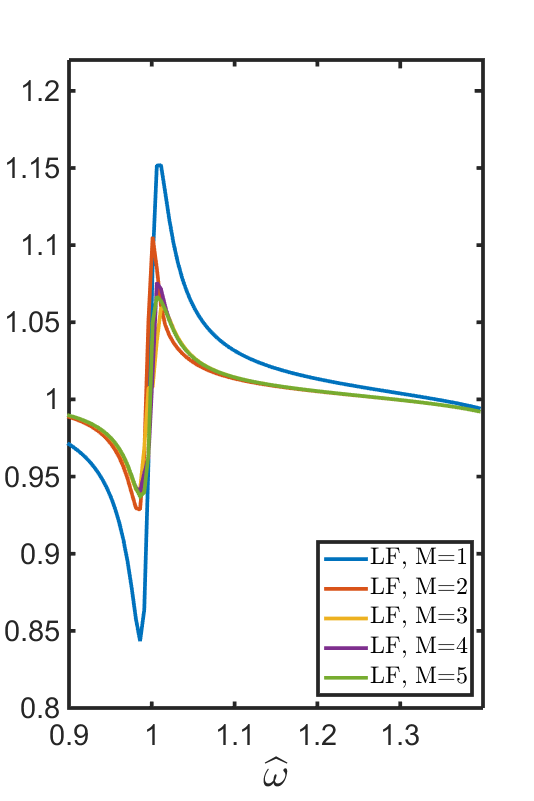}
	\includegraphics[scale=0.258]{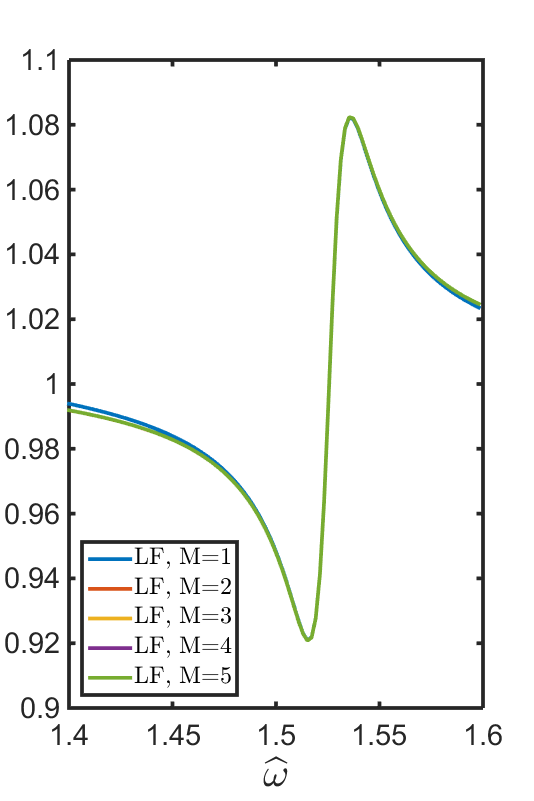}
	\includegraphics[scale=0.258]{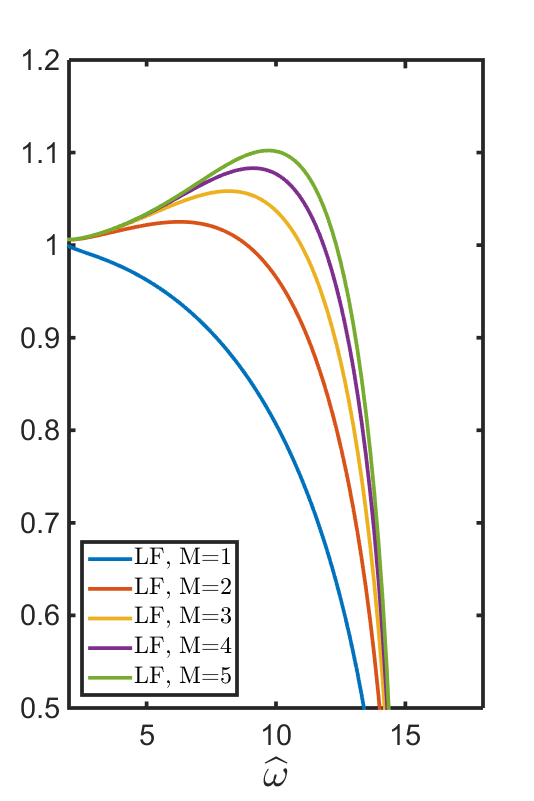} 
	
	\caption{Results for the leap-frog time discretization and FD2M with CFL number $\nu/\nu^{2M}_{max}=0.7$. First row: Normalized phase velocity; Second row: normalized attenuation constants; Third row: normalized energy velocity; Fourth row: normalized group velocity.}
	\label{Fig:phy_FD1}
\end{figure}

\begin{figure}
	\centering
	\includegraphics[scale=0.258]{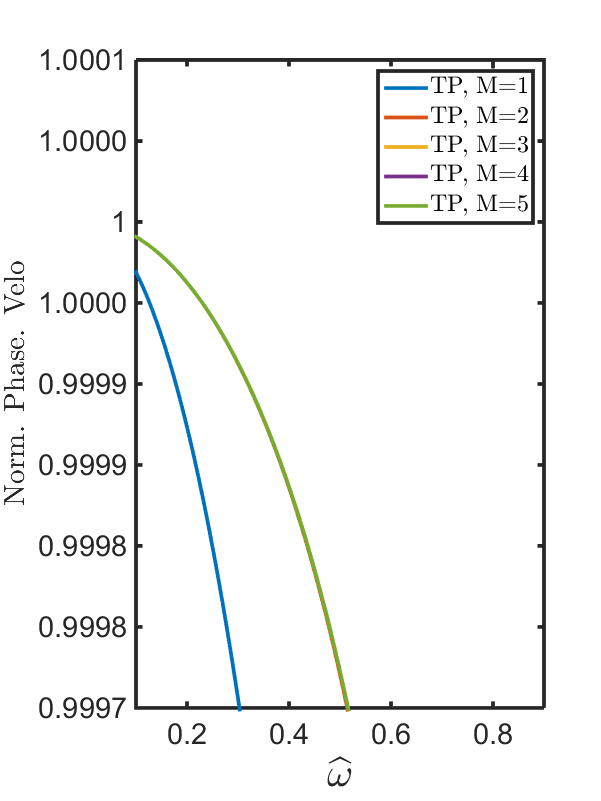}
	\includegraphics[scale=0.258]{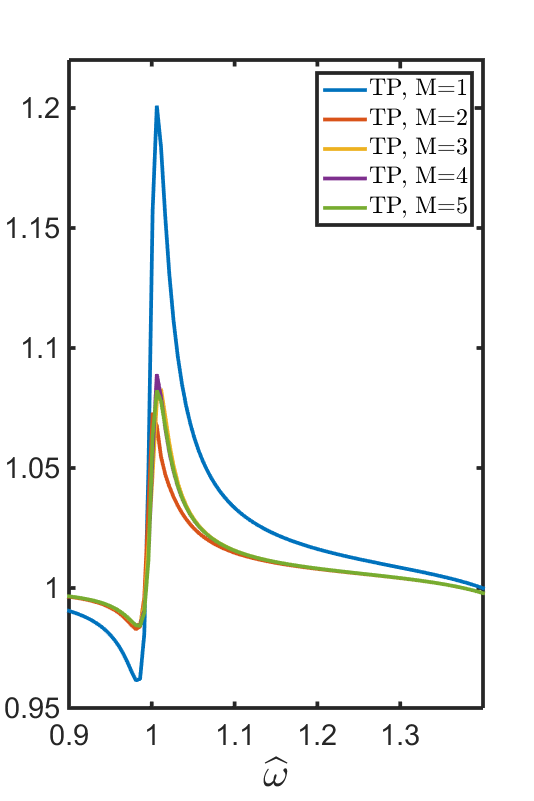}
	\includegraphics[scale=0.258]{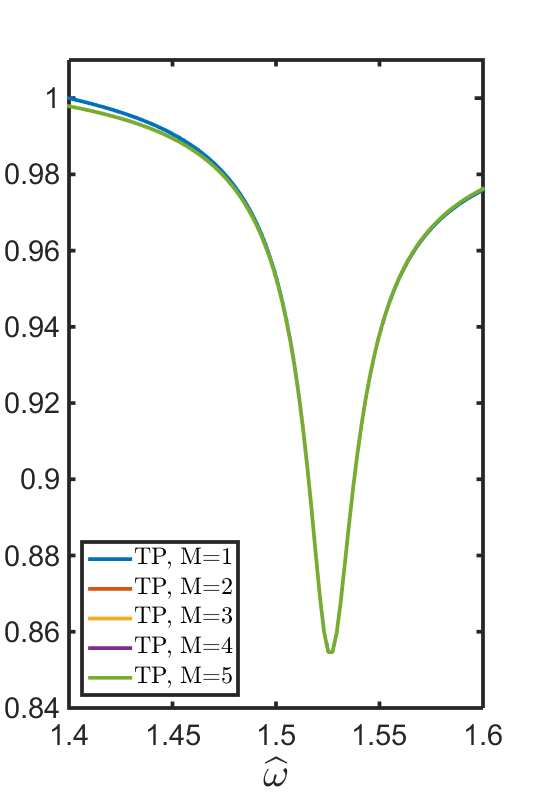}
	\includegraphics[scale=0.258]{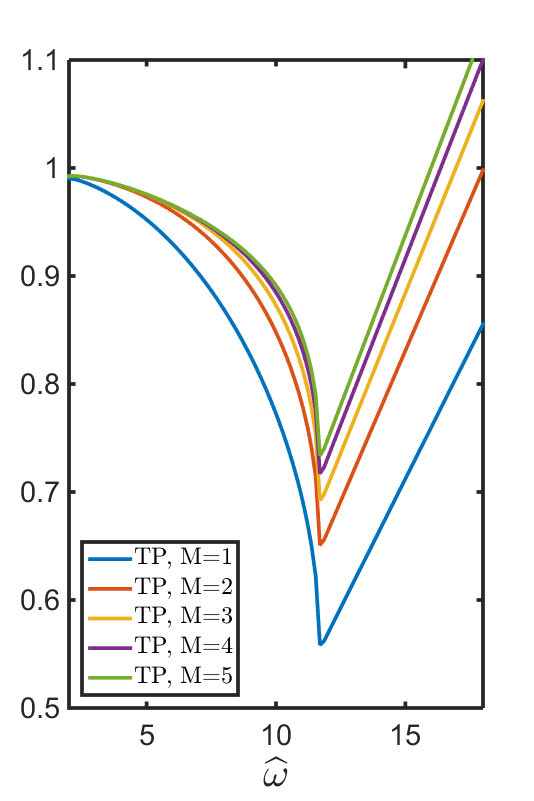} \\
	\includegraphics[scale=0.258]{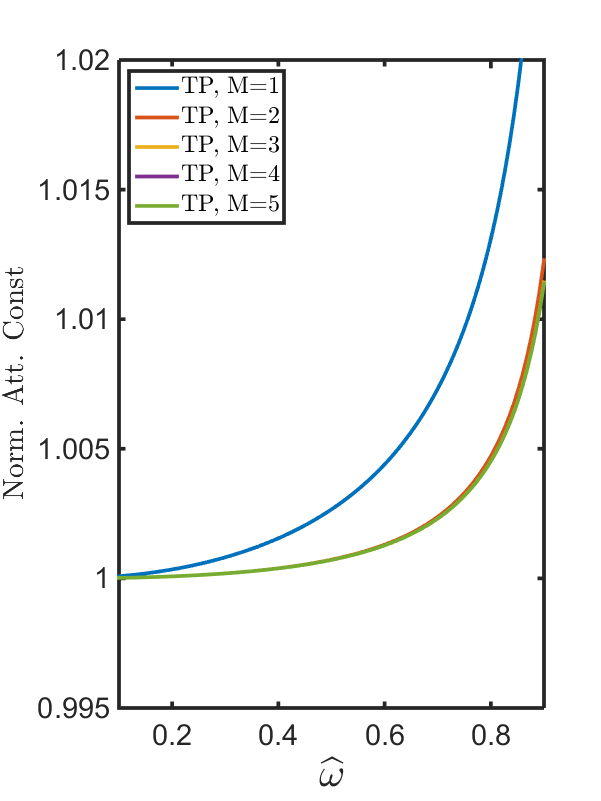}
	\includegraphics[scale=0.258]{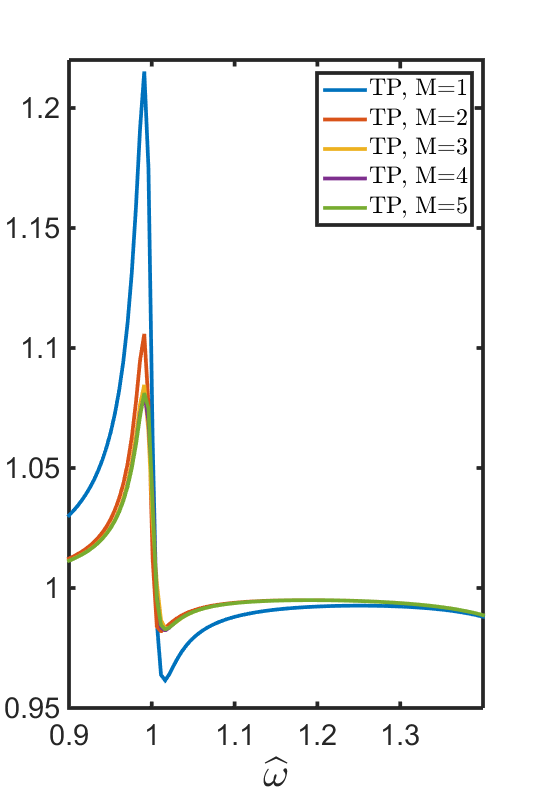}
	\includegraphics[scale=0.258]{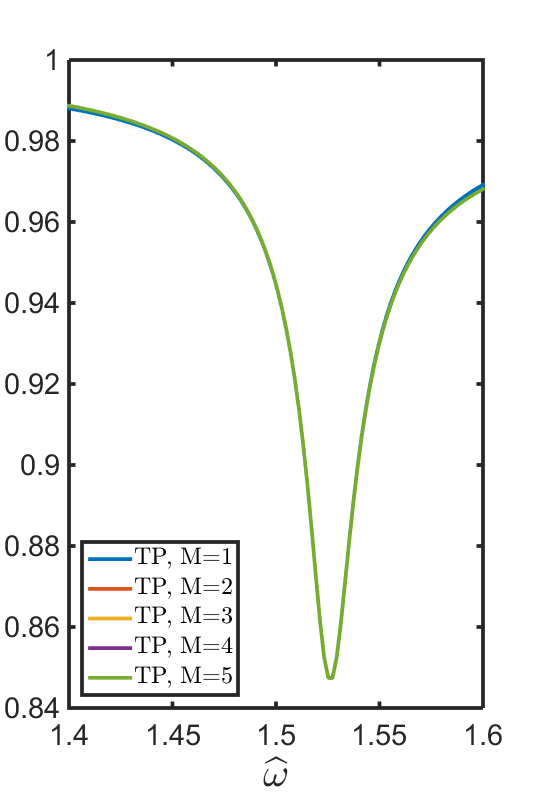}
	\includegraphics[scale=0.258]{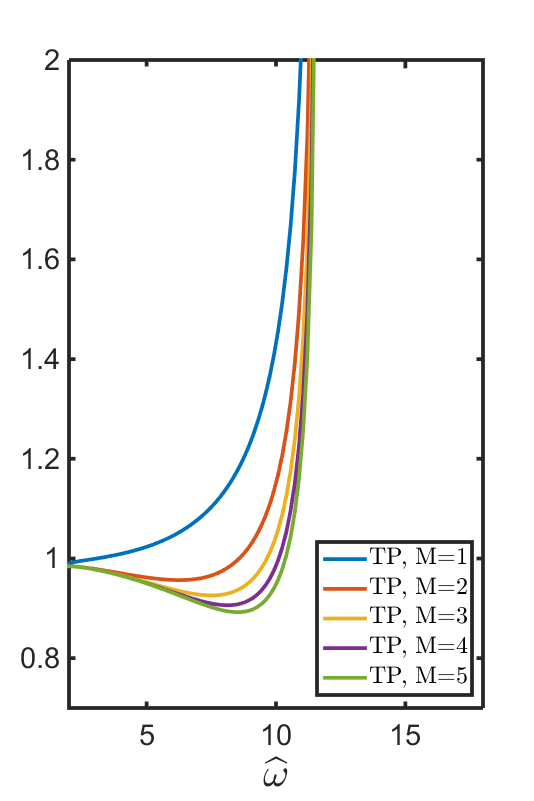} \\
	\includegraphics[scale=0.258]{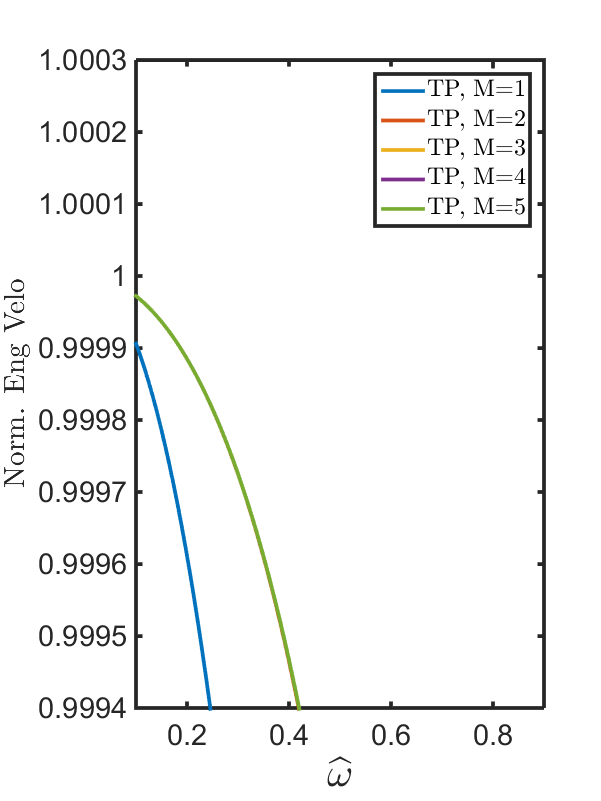}
	\includegraphics[scale=0.258]{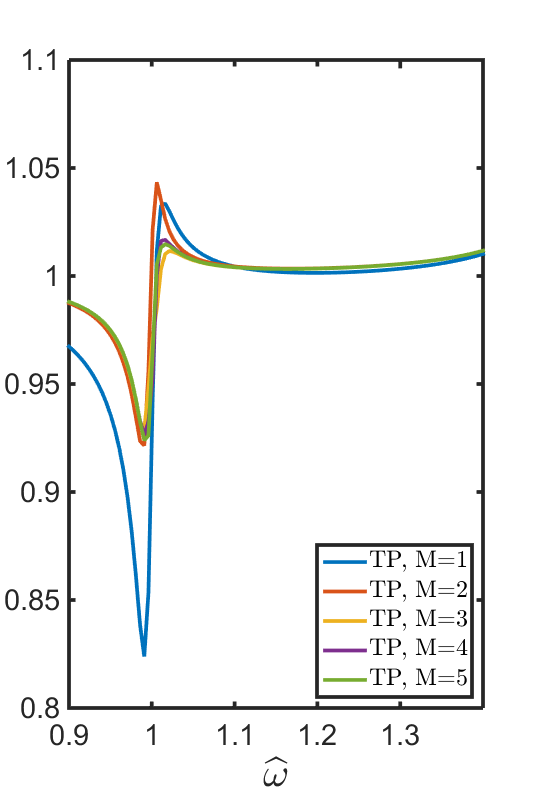}
	\includegraphics[scale=0.258]{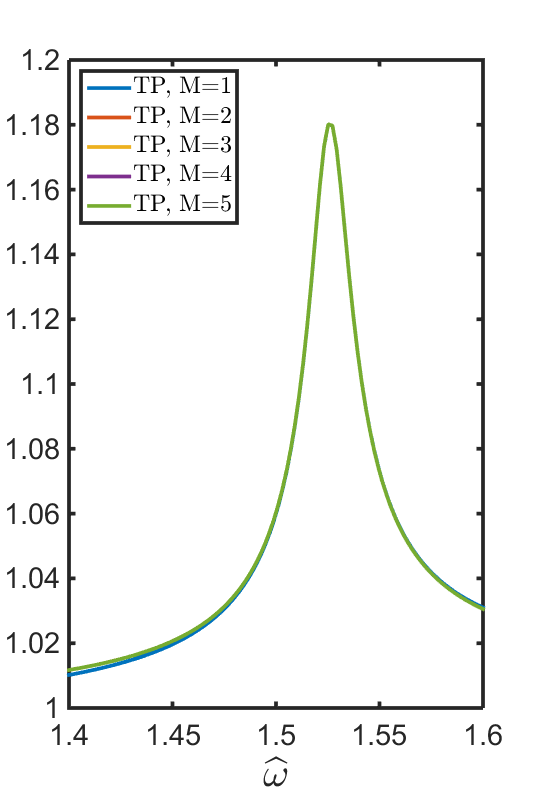}
	\includegraphics[scale=0.258]{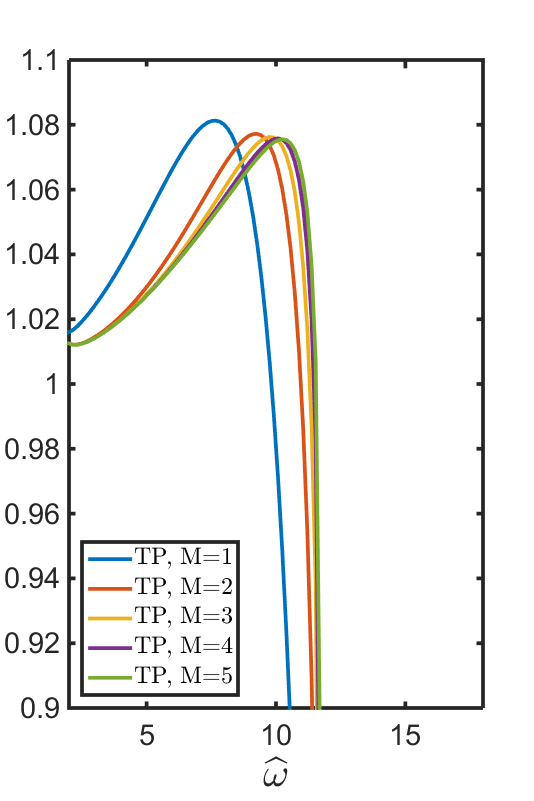} \\
		\includegraphics[scale=0.258]{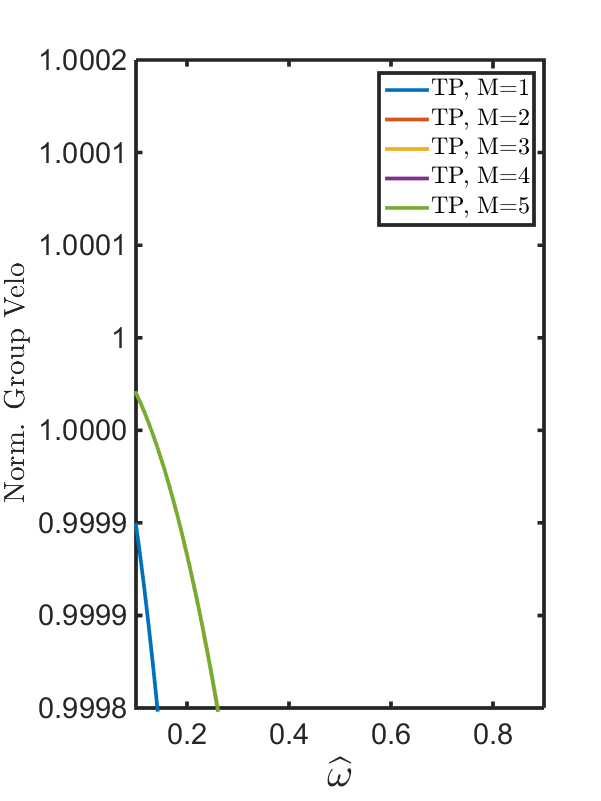}
		\includegraphics[scale=0.258]{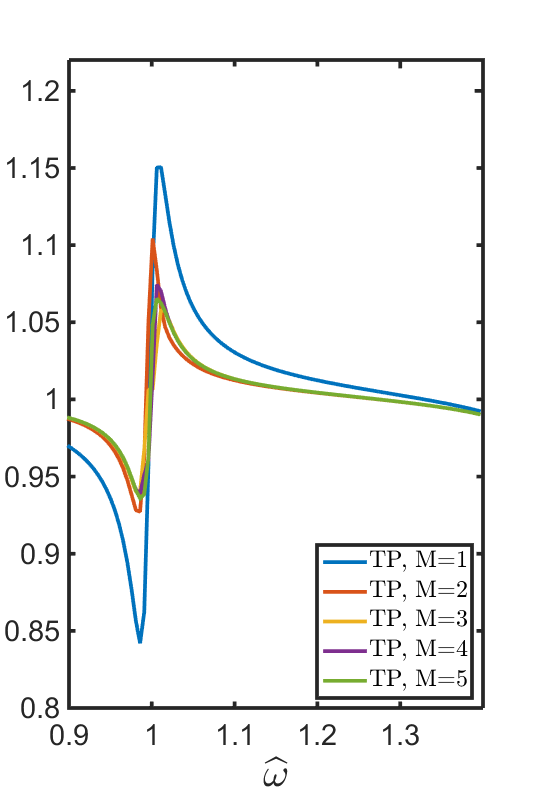}
		\includegraphics[scale=0.258]{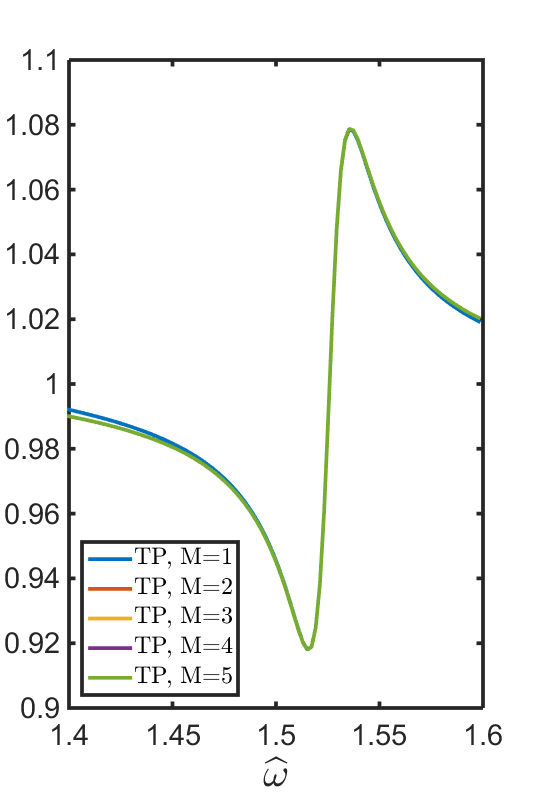}
		\includegraphics[scale=0.258]{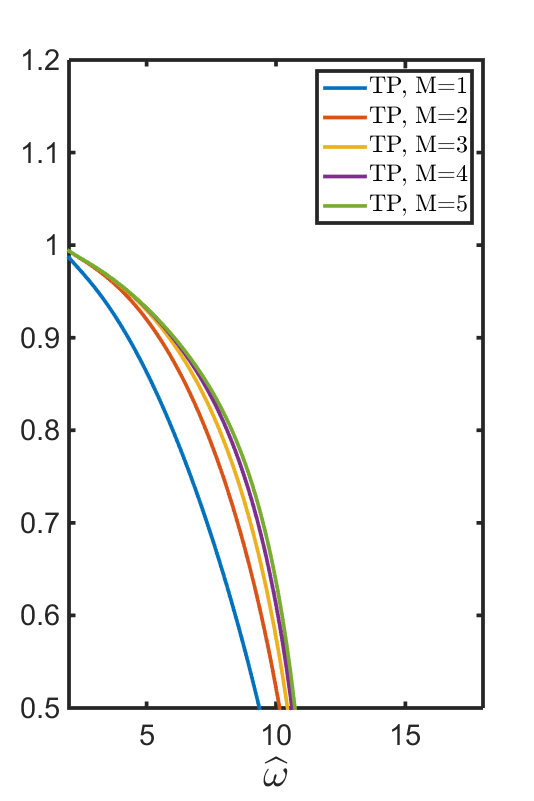} 
	\caption{Results for the trapezoidal time discretization and FD2M with CFL number $\nu/\nu^{2M}_{\max}=0.7$. First row: Normalized phase velocity; Second row: normalized attenuation constants; Third row: normalized energy velocity; Fourth row: normalized group velocity.}
	\label{Fig:FD_TP_phy}
\end{figure}

\begin{figure}
	\centering
	\includegraphics[scale= 0.258] {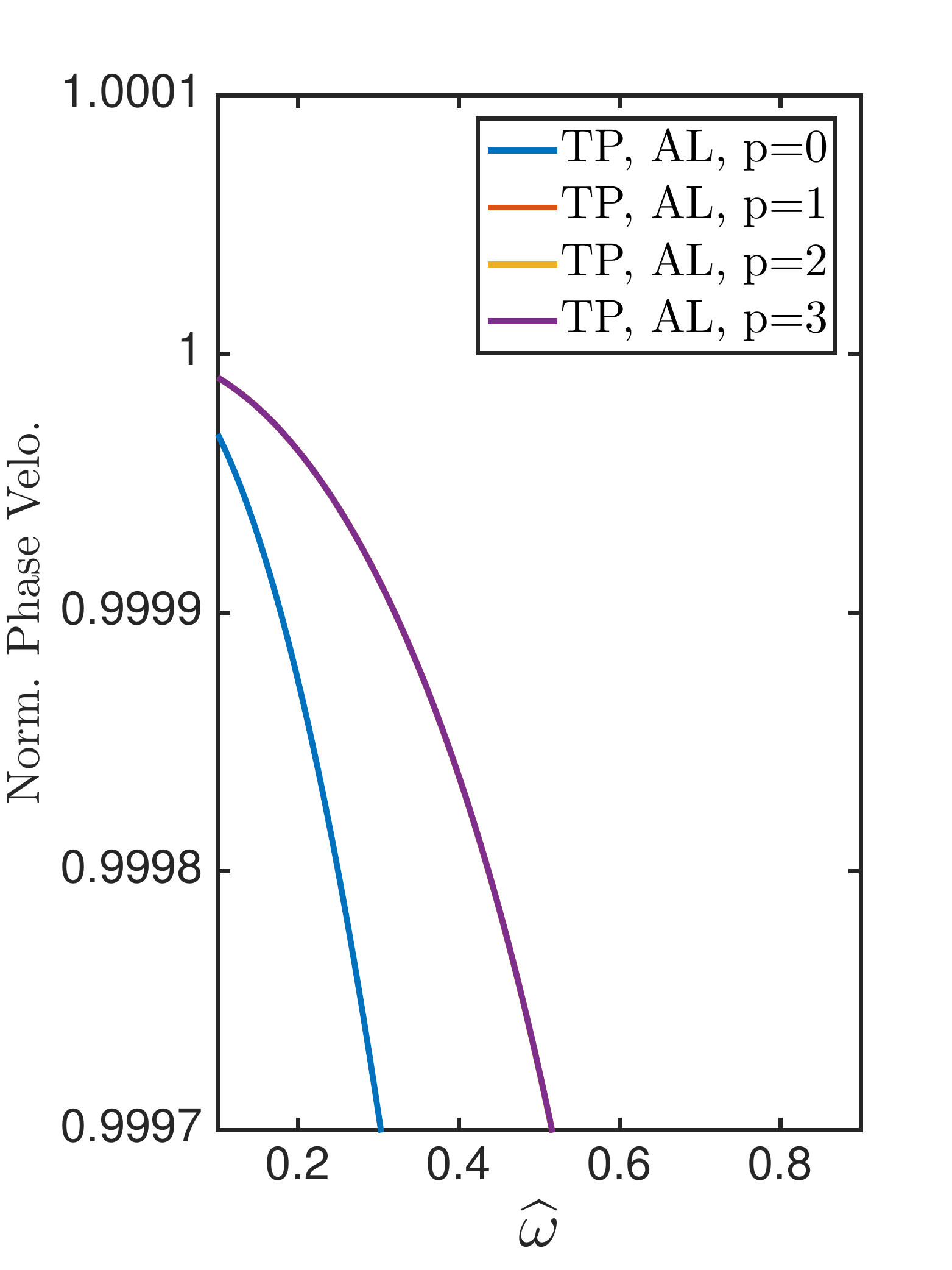}
	\includegraphics[scale= 0.258] {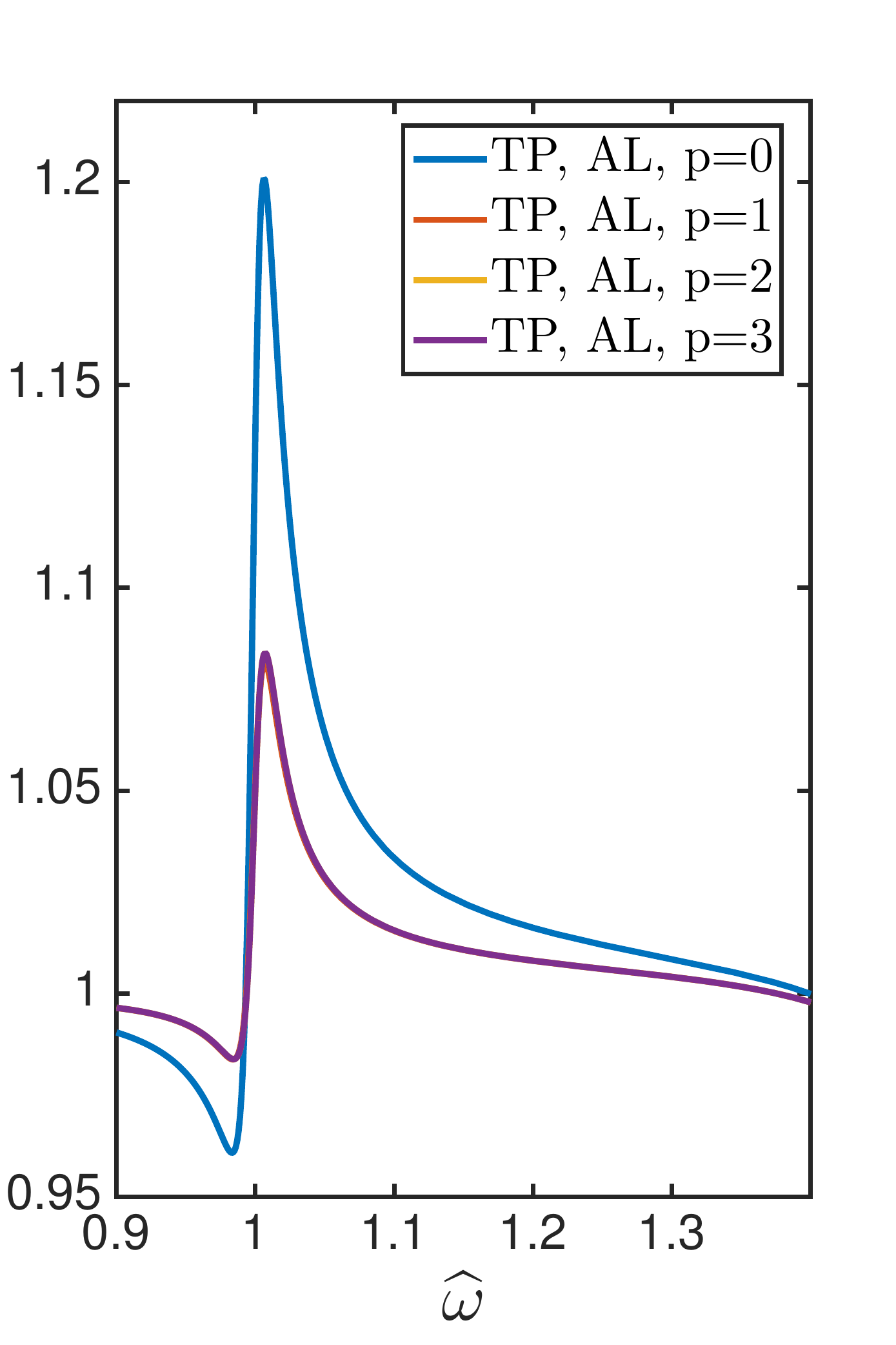}
	\includegraphics[scale= 0.258] {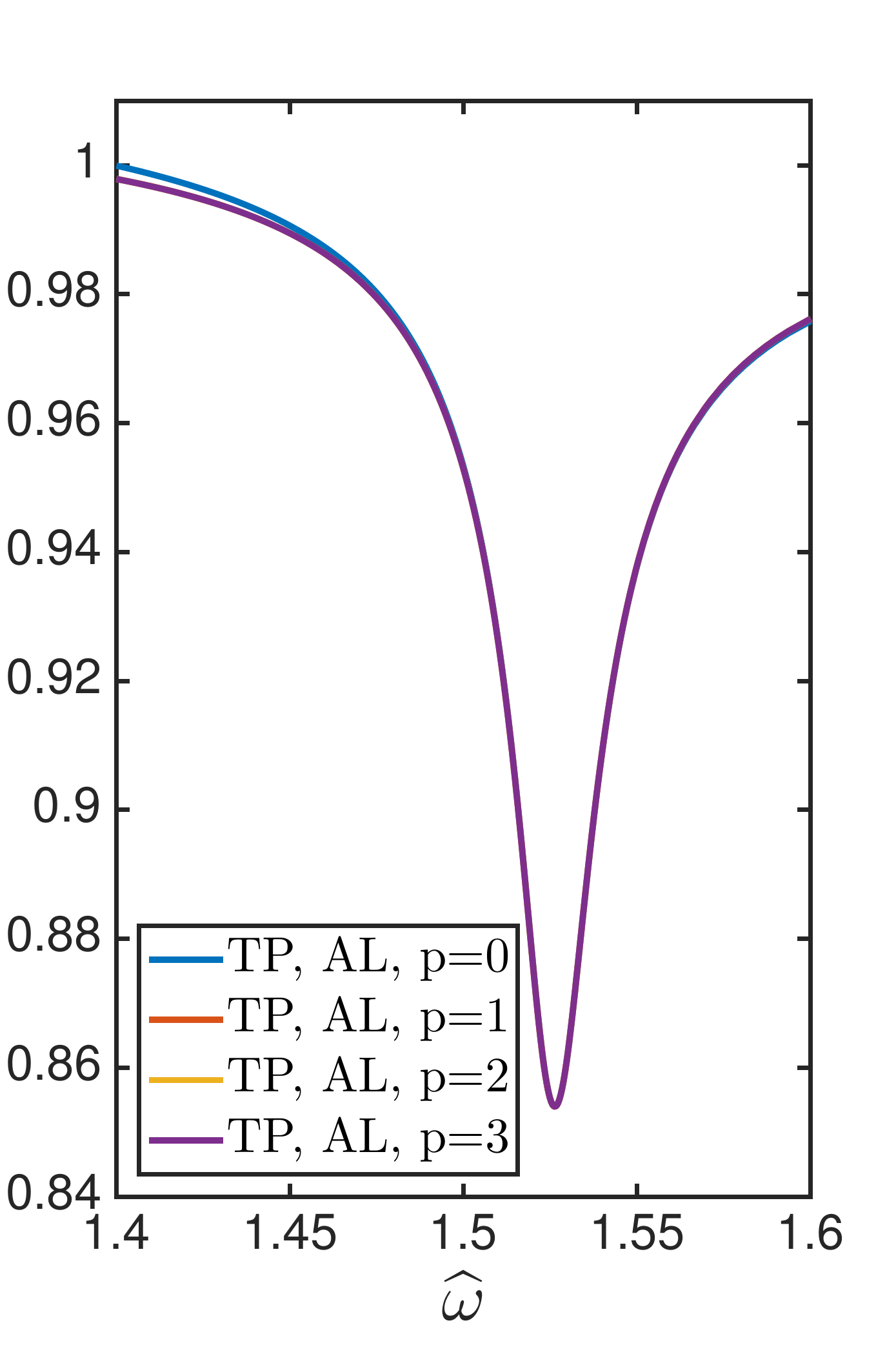}
	\includegraphics[scale= 0.258] {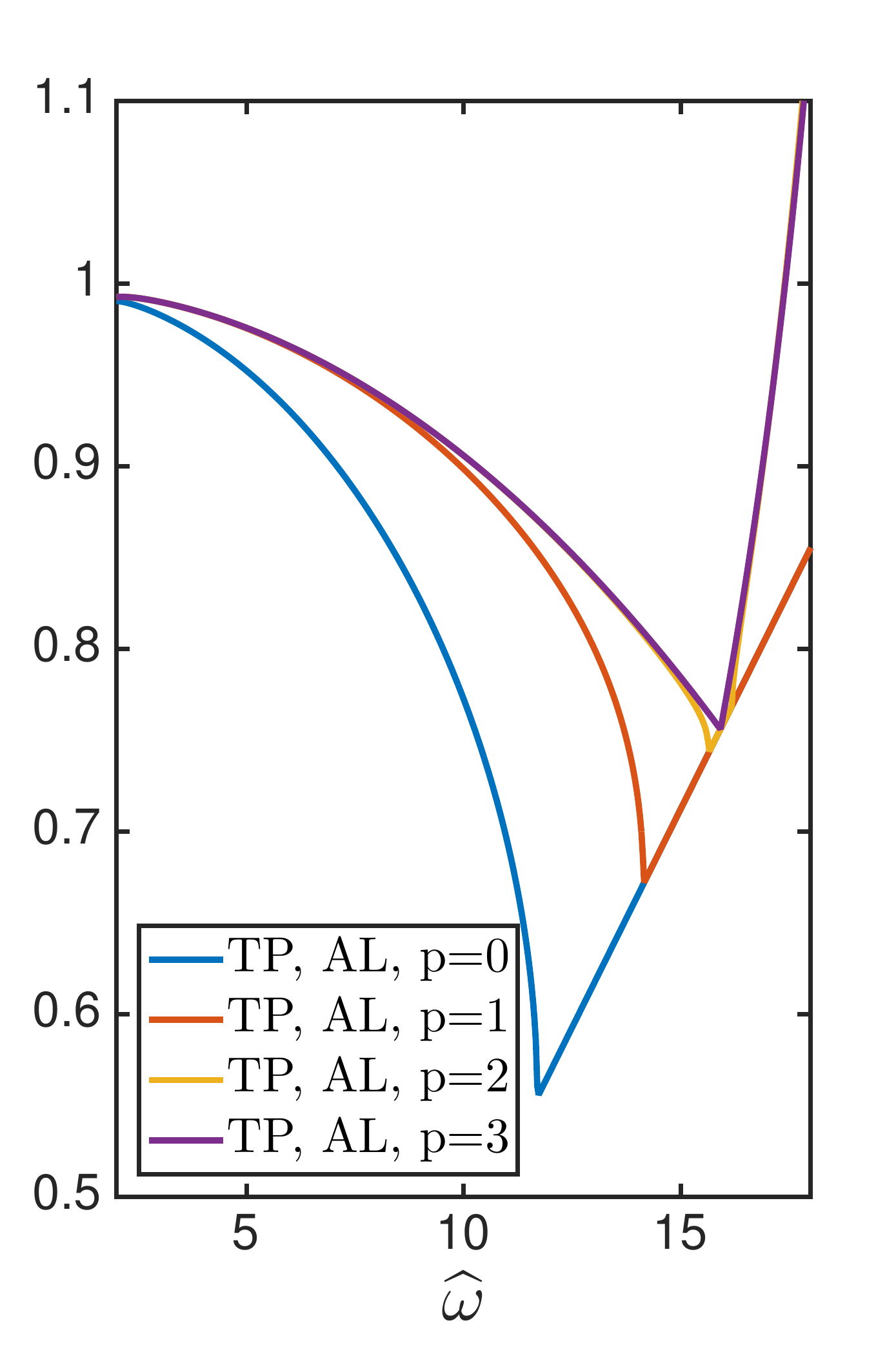} \\
	\includegraphics[scale= 0.258] {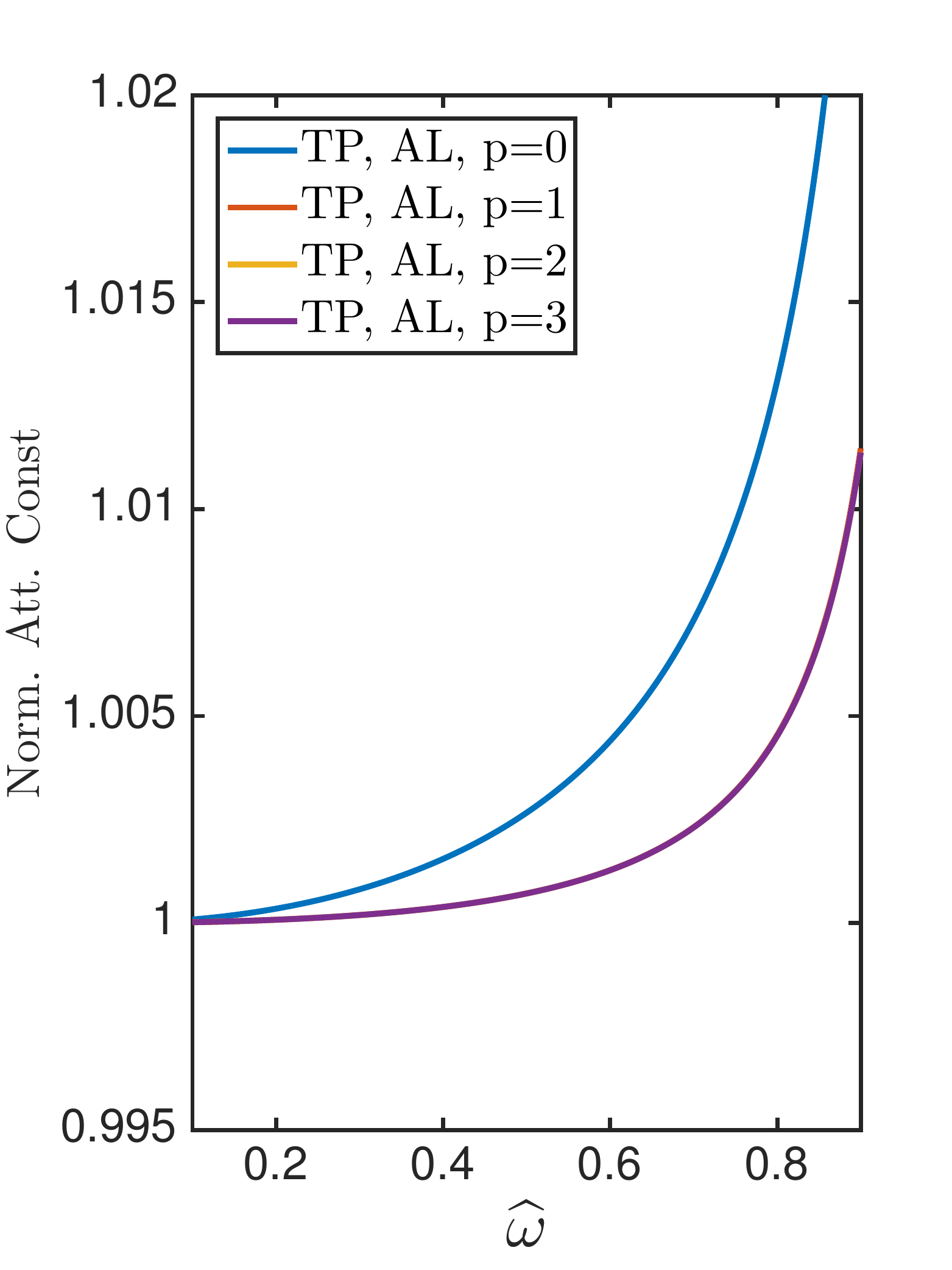}
	\includegraphics[scale= 0.258] {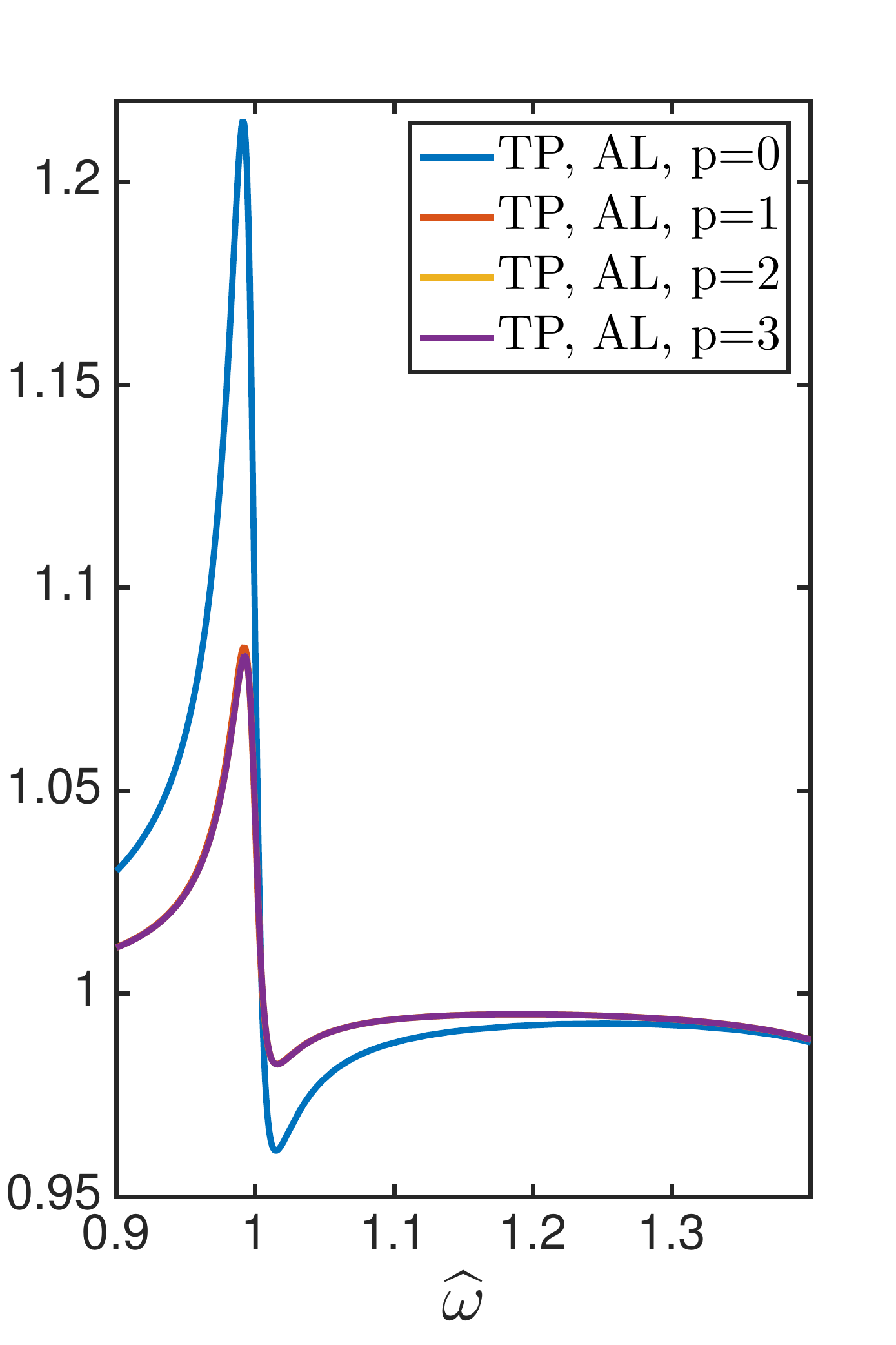}
	\includegraphics[scale= 0.258] {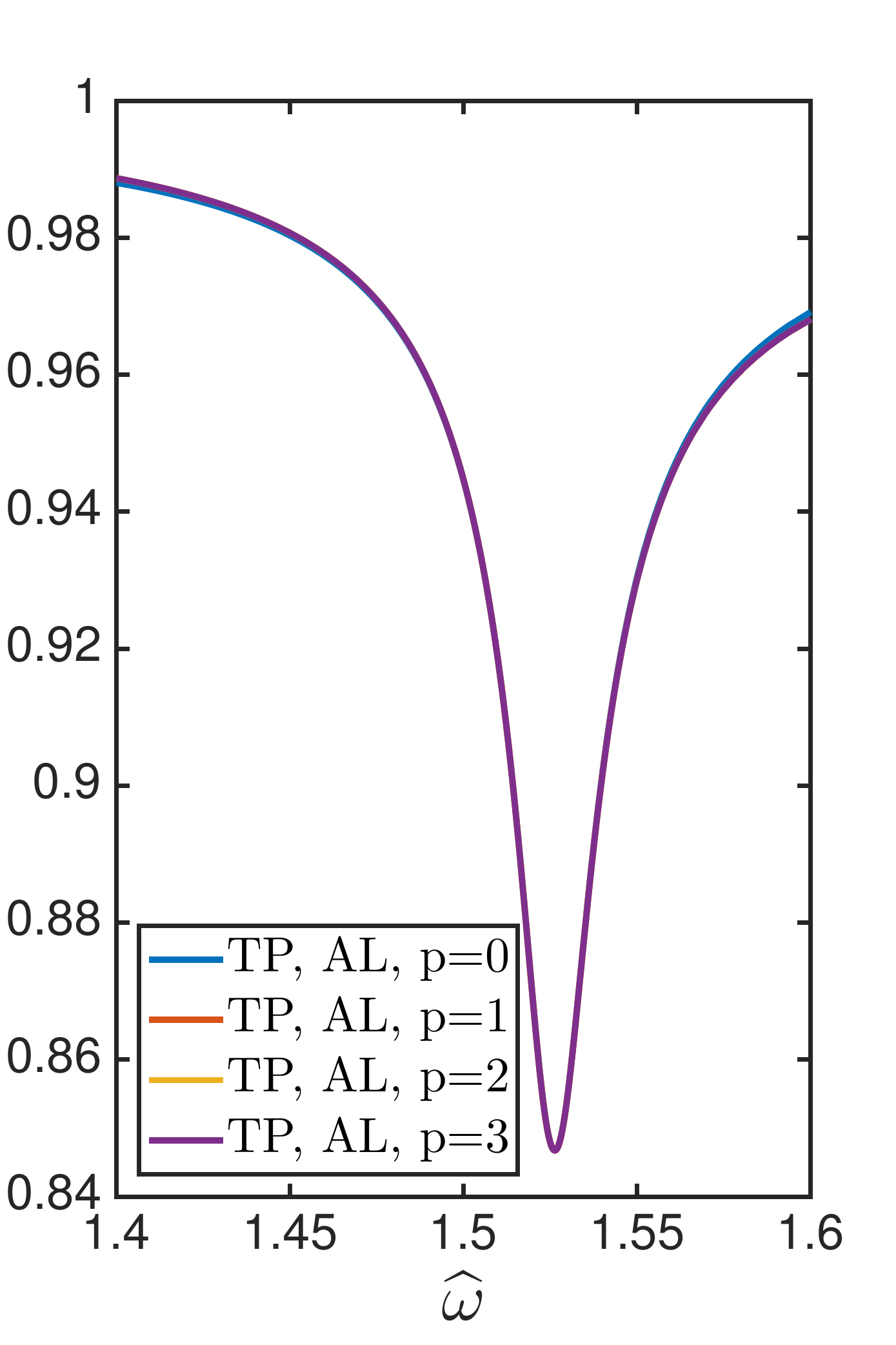}
	\includegraphics[scale= 0.258] {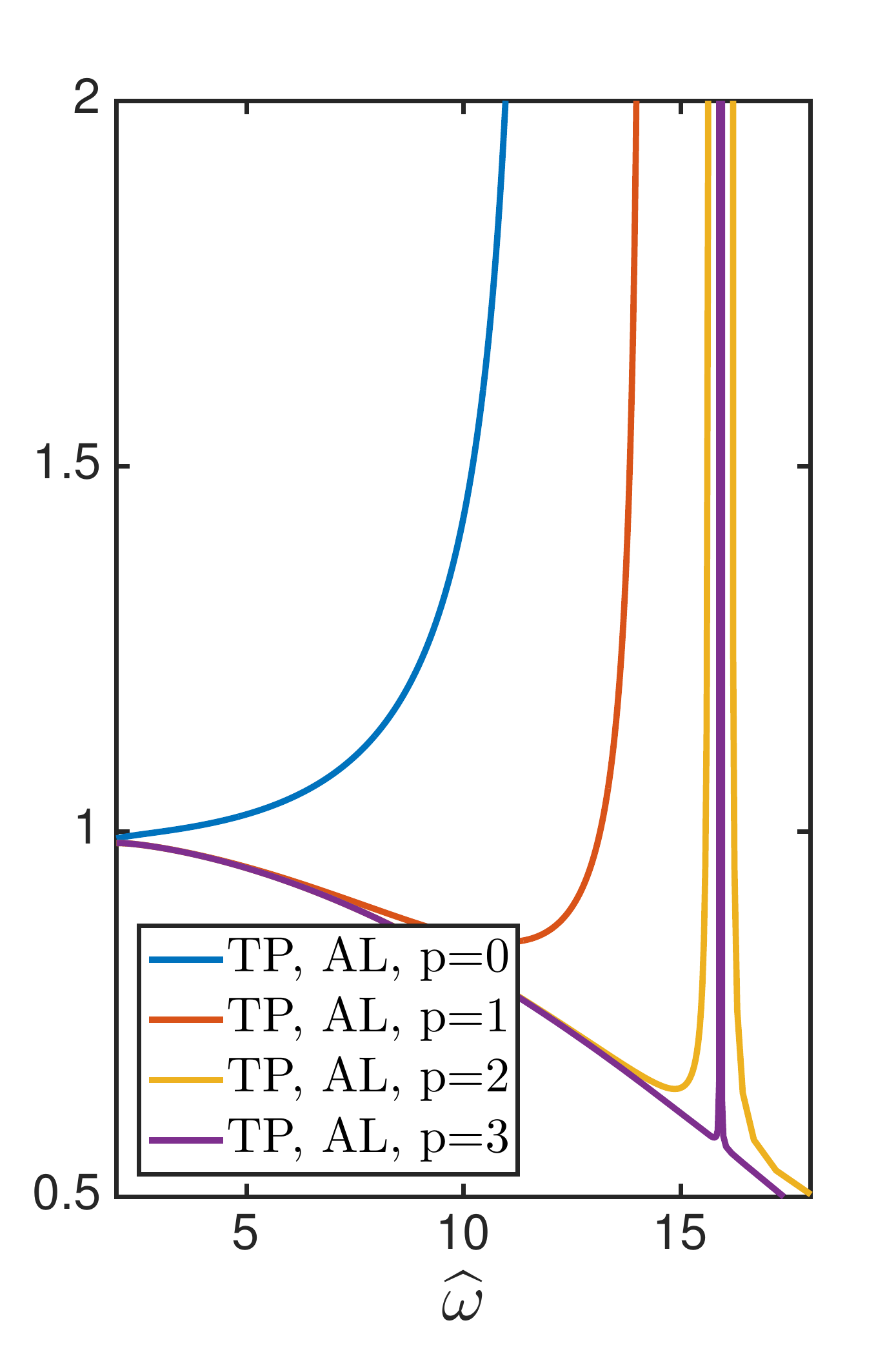}\\
	\includegraphics[scale= 0.258] {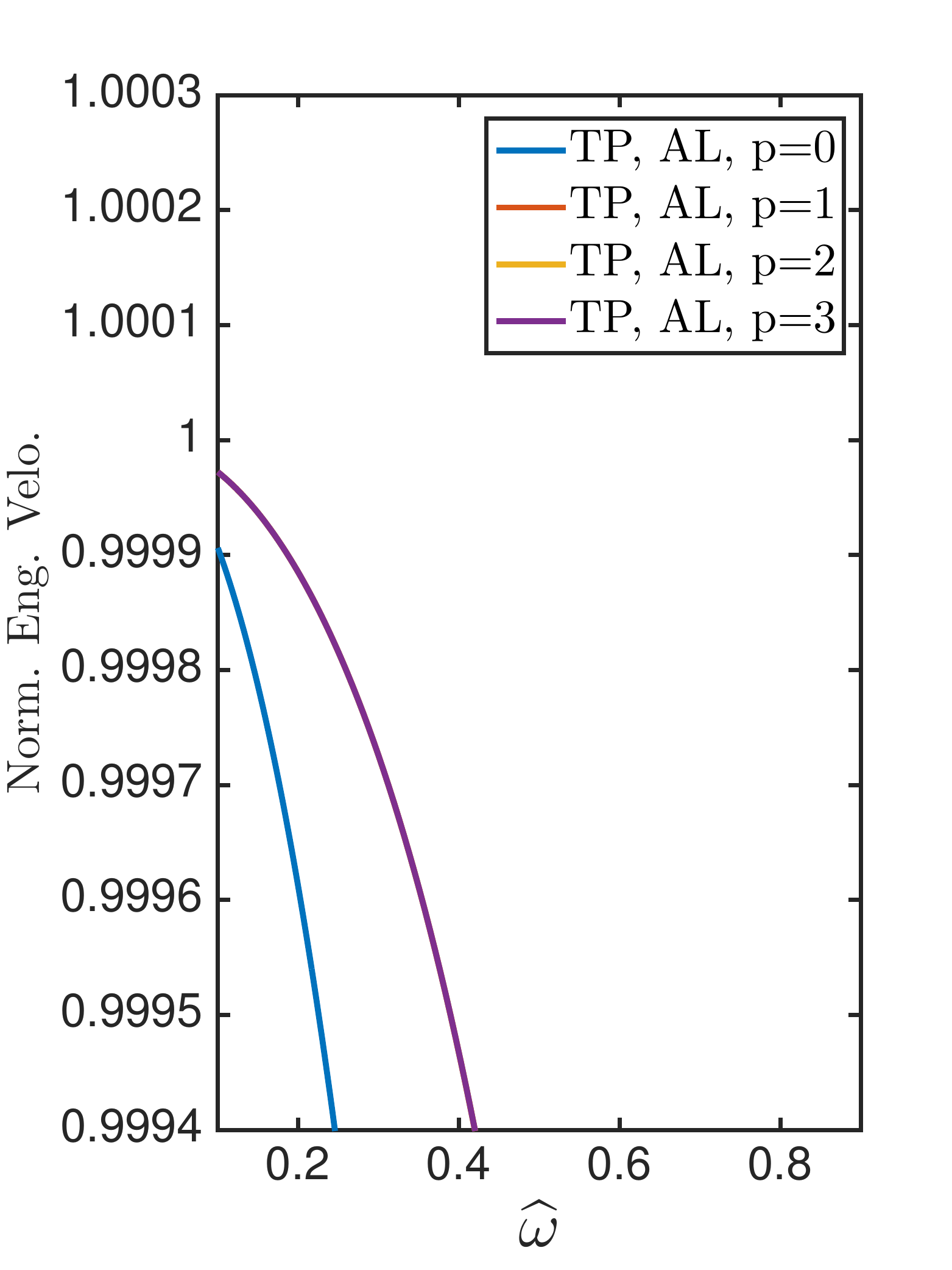}
	\includegraphics[scale= 0.258] {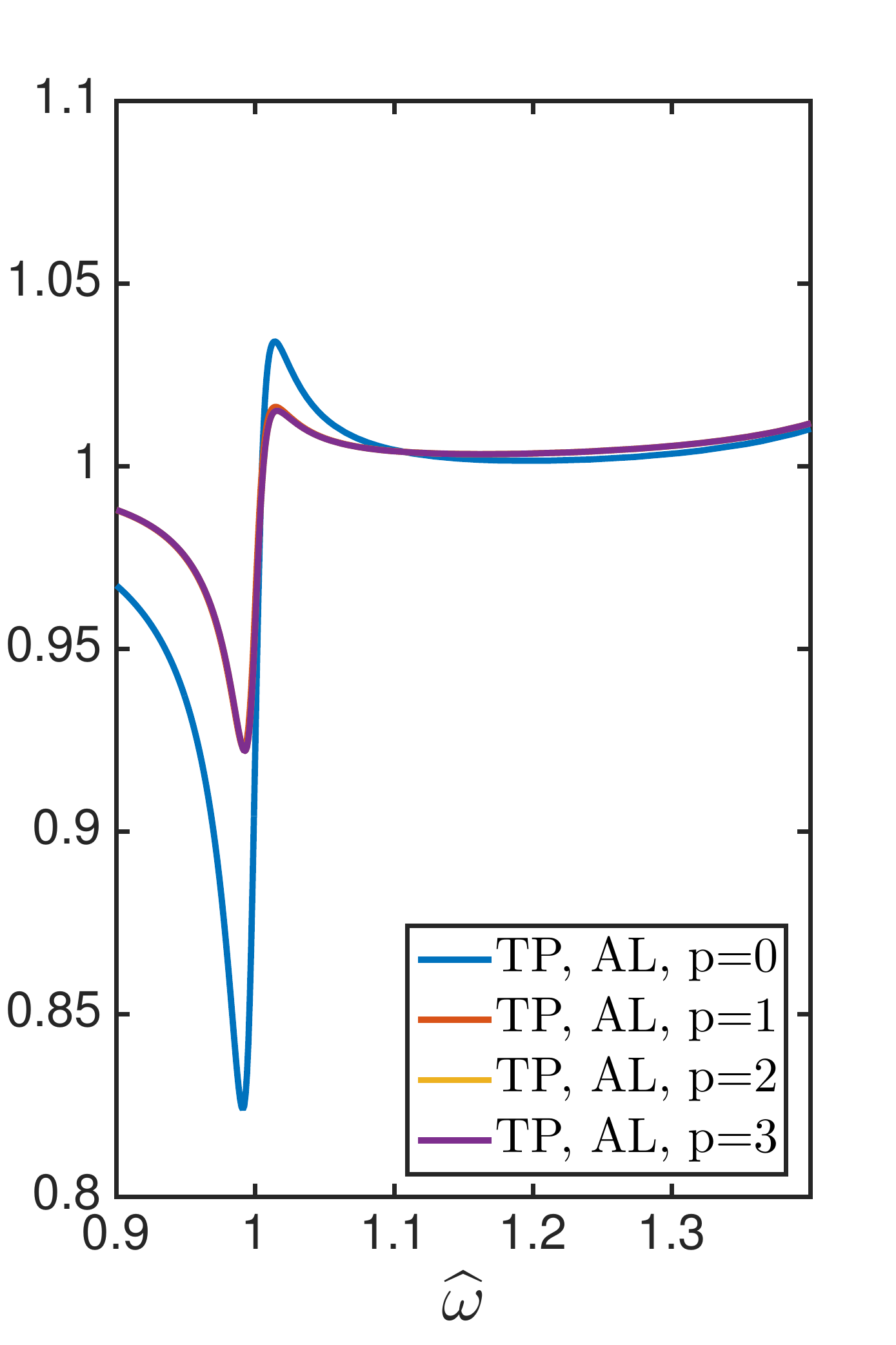}
	\includegraphics[scale= 0.258] {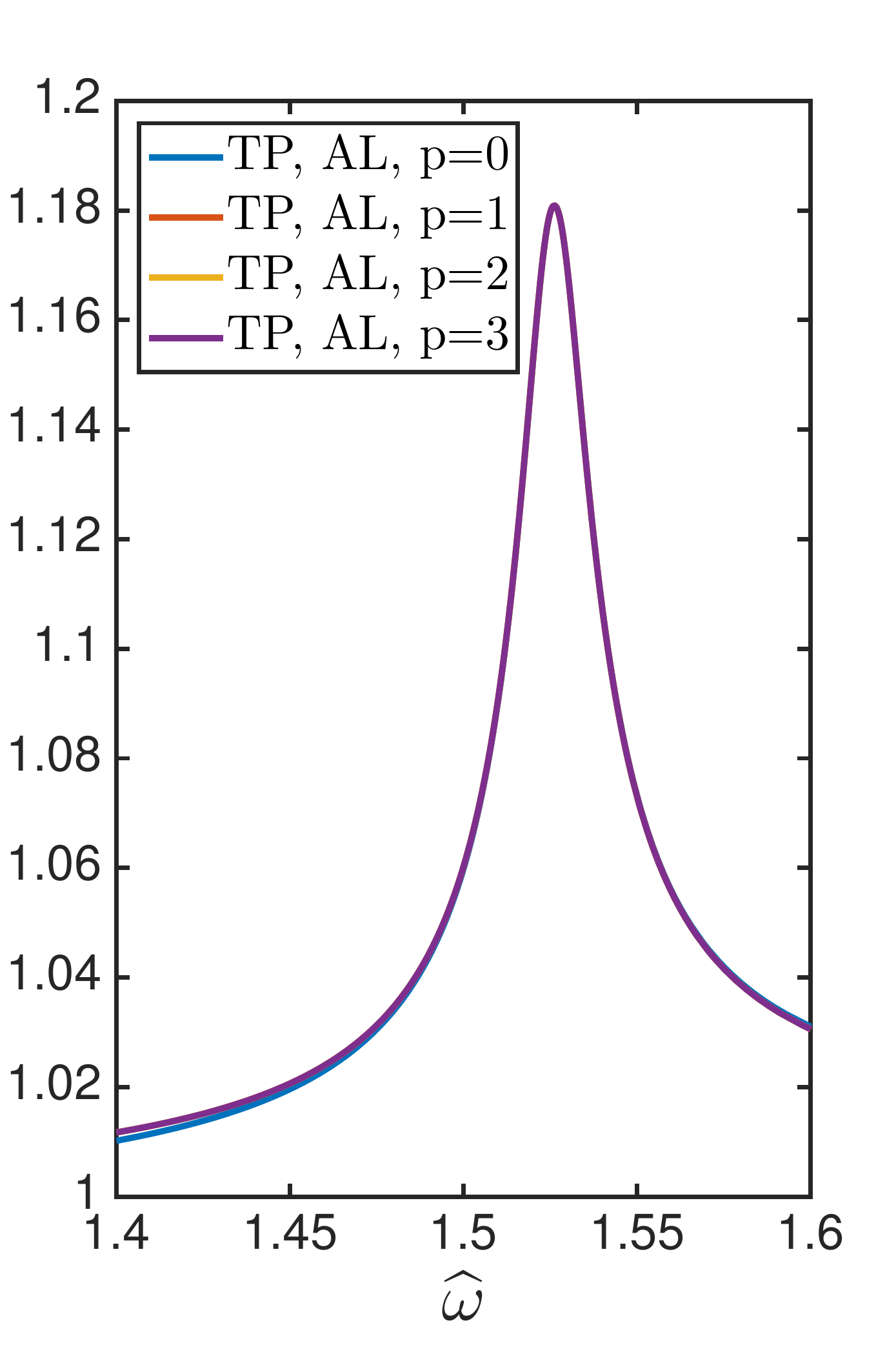}
	\includegraphics[scale= 0.258] {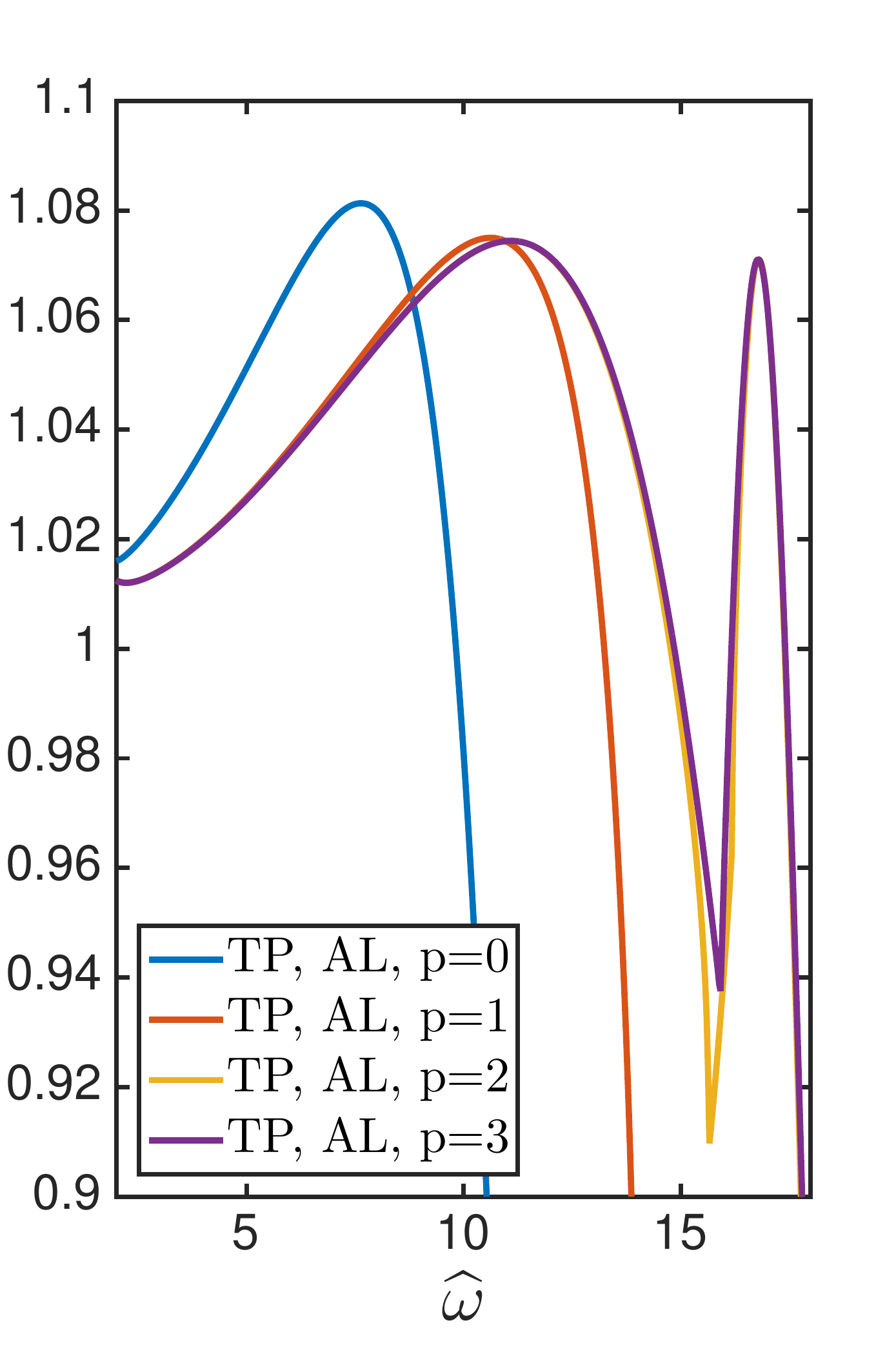} \\
	\includegraphics[scale= 0.258] {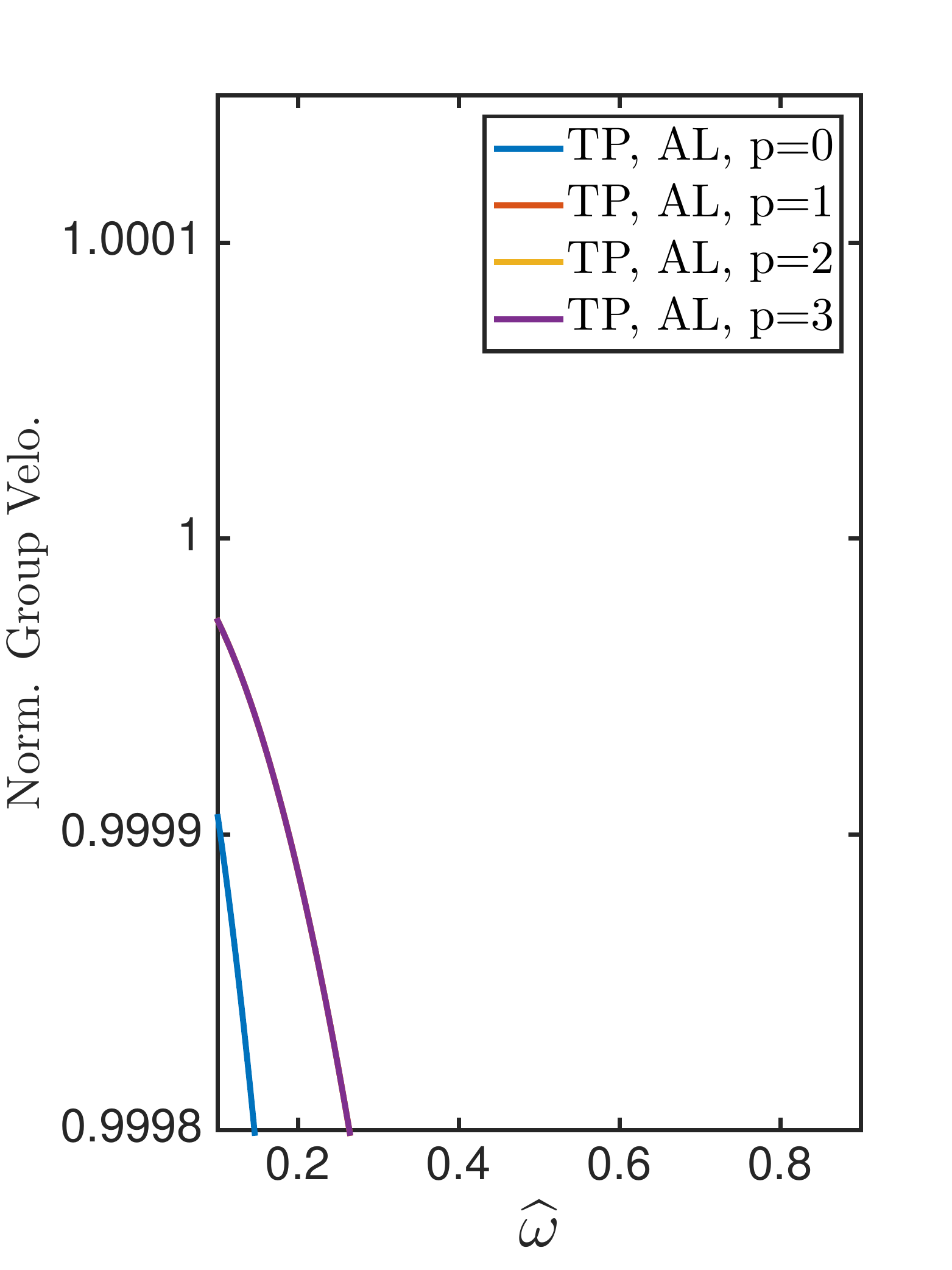}
	\includegraphics[scale= 0.258] {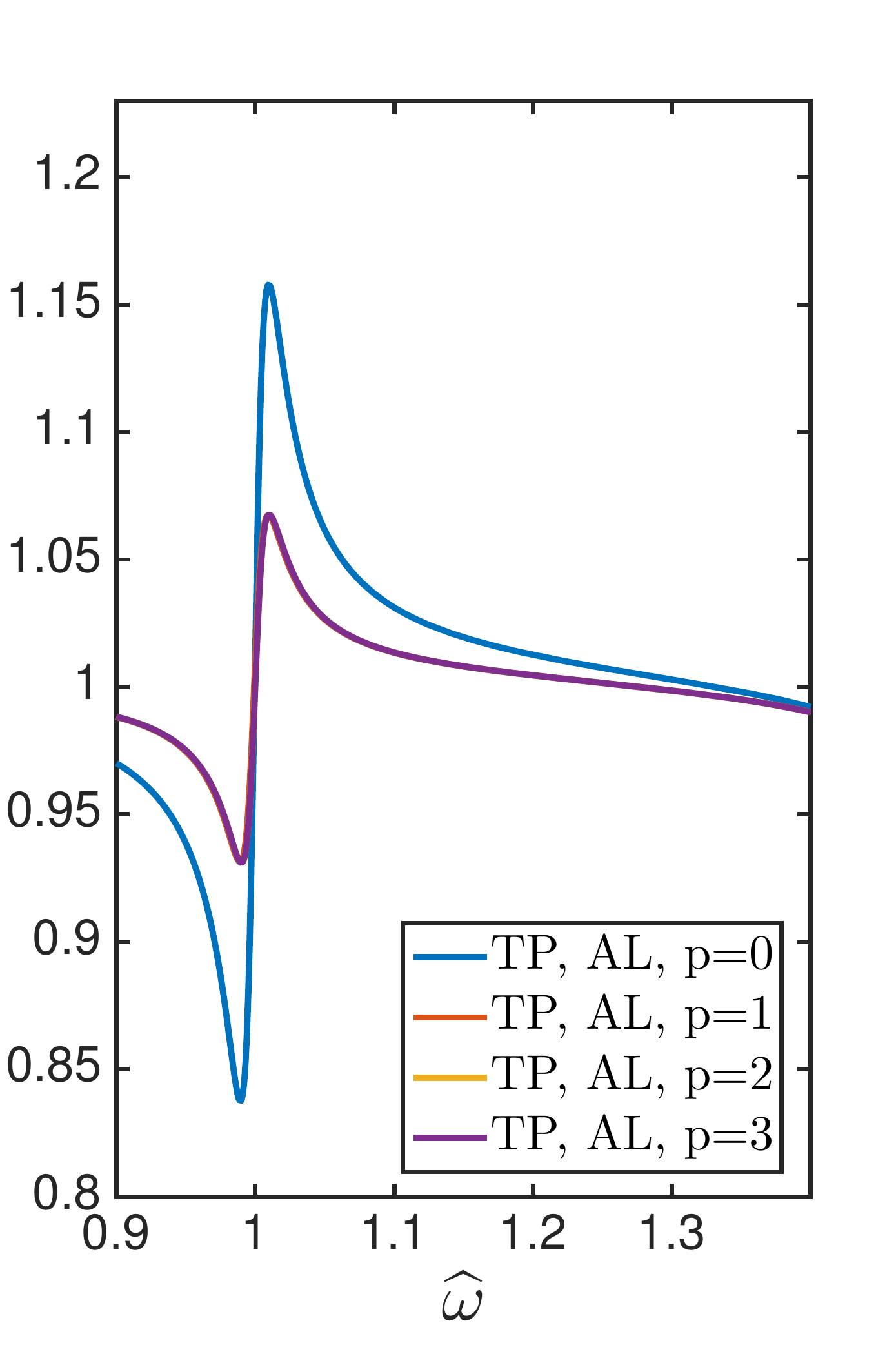}
	\includegraphics[scale= 0.258] {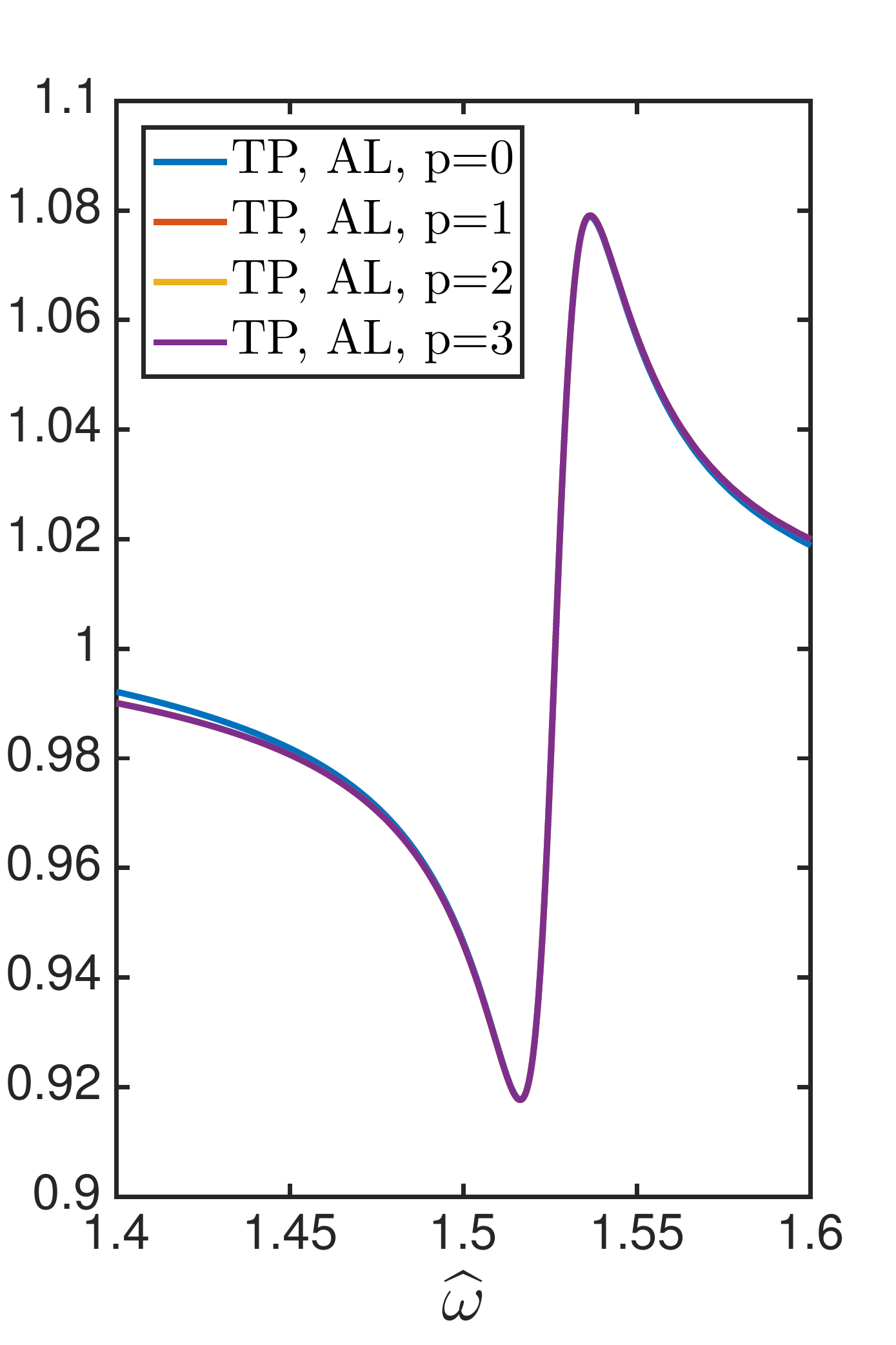}
	\includegraphics[scale= 0.258] {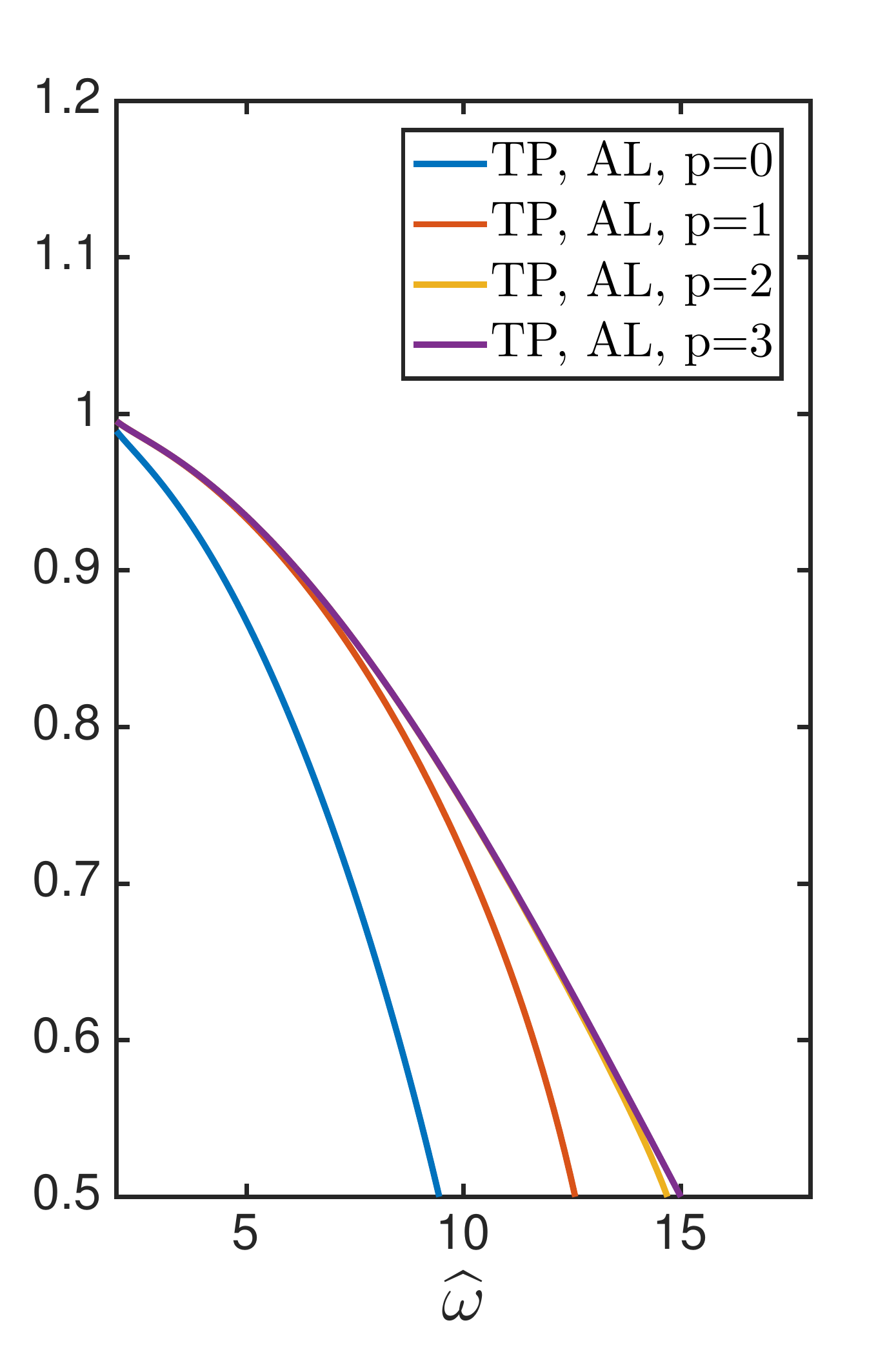} \\
	
	\caption{Results for the trapezoidal time discretization and DG-AL with CFL number $\nu=0.7$. First row: Normalized phase velocity; Second row: normalized attenuation constants; Third row: normalized energy velocity; Fourth row: normalized group velocity.  }
	\label{Fig:phy_DG_AL1}
\end{figure}

\section{Conclusions}  
\label{conclude}
 In this paper, we studied the exact and numerical dispersion relations of a one-dimensional Maxwell's equations in a linear dispersive material characterized by a  single pole Lorentz model for electronic polarization with low loss (i.e. when $\WH{\gamma}$ is small). We  consider two different high order spatial discretizations, the FD and DG methods, each coupled with two different second order temporal discretizations, leap-frog and trapezoidal integrators, to construct both semi-discrete and fully discrete schemes. In addition, for the DG schemes we have considered three different types of fluxes: central, upwind and alternating fluxes. 
 Comparisons based on dispersion analysis are made of the FD and DG methods and the leap-frog and trapezoidal time discretizations.

 It is well known that the FD and DG (which are a class of finite element methods) schemes, both being very popular discretizations, differ quite a lot in how they simulate wave phenomenon in their discrete grids. For example, DG schemes work well for multi-dimensional problems and can be constructed on unstructured meshes for complicated geometries. The FD schemes are simpler to code, and are mostly defined on structured meshes. The extension of the FD methods to non-uniform and unstructured meshes are cumbersome.

 Both types of spatial discretizations can be designed with high spatial order accuracy. The FD scheme achieves this by extending the stencil of the discretization, while higher order polynomials are needed for the DG construction. When we express the dispersion relation for the discrete wavenumber as a function of the angular frequency $\omega$, the number of spurious modes will increase with  $M$ (the accuracy order) of the FD scheme, while for DG schemes, the number of spurious modes is independent of $p$ (the polynomial order). However, as shown in the Appendix of the FD scheme, using an alternative description of phase error, when the discrete angular frequency $\omega$ is expressed as a function of the wave number $k$ the conclusions are reversed. Namely, there are no spurious modes for FD schemes, while more spurious modes will be present for higher order DG schemes, see \cite{cheng2017L2} for relevant discussions in  free space.  

 When comparing the order of numerical errors, the FD schemes manifest the same order of accuracy of the dispersion error and point-wise convergence error, while the DG schemes have higher order of accuracy in dispersion error than  in the $L^2$ errors \cite{cohen1,cohen2} (superconvergence in dispersion error). The CFL numbers for the two methods when coupled with an explicit time stepping are also different. It is known  that the CFL number will approach a constant other than zero when $M\rightarrow\infty$ for the FD scheme, but the CFL number will go to zero when $p\rightarrow\infty$ for the DG scheme. Therefore, high order DG schemes require much smaller time steps than high order FD schemes.

Based on the numerical dispersion results in this paper, we observe that the physical dispersion of the material plays an important role in the numerical dispersion errors. For the low-loss materials considered, we can observe that the error is largest near the resonance frequency. This is no longer true for materials with high loss (i.e. when $\WH{\gamma}$ is not small). An interesting finding is that for some materials and discretization parameters, we observe  counterintuitive results that the dispersion error of a low order scheme can be potentially smaller than that of high order schemes (see for example Figure \ref{Fig:coeff}).
This demonstrates that the dispersion analysis conducted for free space may not be revealing for general dispersive media.

We find that the second order accuracy of the temporal discretizations limits the accuracy of the numerical dispersion errors, and is a good motivator for considering high order temporal discretizations, which are non-trivial to construct for the case of dispersive Maxwell models \cite{Youngcold}. This limiting behavior in the medium absorption band is made clear by the difference in errors in the semi-discrete schemes versus the fully discrete schemes. In our future work we will investigate higher order temporal discretizations.

\appendix
\section{An Alternative Dispersion Analysis for Semi-Discrete Finite Difference Schemes}

In this appendix, we  provide an alternative method of analyzing the dispersion error of the semi-discrete in space high order FD schemes (FD2M). We express the discrete angular frequency $\omega$ as a function of the continuous wavenumber $k\in\mathbb{R}$, and measure the relative errors that result for different $M, M\in \mathbb{N}$, with $2M$ being the spatial accuracy of the schemes. 

We introduce the following definitions
\begin{align}
\displaystyle \label{Nota}
\WH{k}:= kh, \qquad
  F_{2M}(\WH{k}):= 
2\sum_{p=1}^M \frac{[(2p-3)!!]^2}{(2p-1)!} 
\sin^{2p-1}\left( \frac{\WH{k}}{2} \right).
\end{align}
\noindent For the exact dispersion relation of Maxwell's equations in a one spatial dimensional Lorentz dielectric, by solving $\det(\mathcal{A})=0$ with $\mathcal{A}$ given by  \eqref{DisEx2}, we get the following quartic equation for the continuous angular frequency $\WH{\omega}^{\text{ex}} = \omega^{\text{ex}}/\omega_{1}$,
\begin{align}
\displaystyle
(\WH{\omega}^{\text{ex}})^4 + 2i\,\WH{\gamma}\,(\WH{\omega}^{\text{ex}})^3
- \frac{1}{\epsilon_\infty}
\left(
\epsilon_s +  \frac{\WH{k}^2}{(\omega_1h)^2} 
\right) (\WH{\omega}^{\text{ex}})^2
- \frac{2i}{\epsilon_\infty} \WH{\gamma}\, \frac{\WH{k}^2}{(\omega_1h)^2} \, \WH{\omega}^{\text{ex}}
+ \frac{1}{\epsilon_\infty}\frac{\WH{k}^2}{(\omega_1h)^2}  = 0. \label{qq6}
\end{align}
\noindent Similarly, considering the dispersion relation of semi-discrete FD2M scheme \eqref{Dissemi4}, we  have  
\begin{align}
\displaystyle \label{qq9}
(\WH{\omega}^{\text{FD},2M})^4 
+ 2i\,\WH{\gamma}\,(\WH{\omega}^{\text{FD},2M})^3 
- \frac{1}{\epsilon_\infty}
\left(
\epsilon_s + \frac{F_{2M}(\WH{k})^2}{(\omega_1 h)^2}
\right) (\WH{\omega}^{\text{FD},2M})^3 
- \frac{2i}{\epsilon_\infty} \WH{\gamma}\, \frac{F_{2M}(\WH{k})^2}{(\omega_1 h)^2}\, \WH{\omega}^{\text{FD},2M}
+ \frac{1}{\epsilon_\infty} \frac{F_{2M}(\WH{k})^2}{(\omega_1 h)^2}= 0.
\end{align}

Clearly, both \eqref{qq6} and \eqref{qq9} have four (complex) roots each. Therefore, the FD scheme has no spurious modes for the discrete angular frequency. To better understand the errors, similar to previous sections, we first consider the lossless material ($\WH{\gamma}=0$) as an example. In this case, only even order terms appear in  \eqref{qq6} and \eqref{qq9}, and we can get
\begin{subequations}
	\label{QAex}
	\begin{align}
	\displaystyle
	\WH{\omega}^{\text{ex}}_{1,2}( \WH{k} )
	&= \pm \frac{1}{\sqrt{2}}
	\left[
	\frac{\epsilon_s}{\epsilon_\infty} + \frac{\WH{k}^2}{\epsilon_\infty(\omega_1h)^2}
	- \sqrt{
	\left(
	\frac{\epsilon_s}{\epsilon_\infty} + \frac{\WH{k}^2}{\epsilon_\infty(\omega_1h)^2}
	\right)^2
	- \frac{4\WH{k}^2}{\epsilon_\infty(\omega_1h)^2}
	}\,
	\right]^{1/2}
	,   \label{QAex1} \\
	\WH{\omega}^{\text{ex}}_{3,4} (\WH{k})
	&=
	\pm \frac{1}{\sqrt{2}}
	\left[
	\frac{\epsilon_s}{\epsilon_\infty} + \frac{\WH{k}^2}{\epsilon_\infty(\omega_1h)^2}
	+ \sqrt{
		\left(
		\frac{\epsilon_s}{\epsilon_\infty} + \frac{\WH{k}^2}{\epsilon_\infty(\omega_1h)^2}
		\right)^2
		- \frac{4\WH{k}^2}{\epsilon_\infty(\omega_1h)^2}
	}\,
	\right]^{1/2}
	, \label{QAex2} 
	\end{align}
\end{subequations}
and $\WH{\omega}^{\text{FD},2M}_{1,2}( \WH{k} )=\WH{\omega}^{\text{ex}}_{1,2}(F_{2M}(\WH{k})), \quad \WH{\omega}^{\text{FD},2M}_{3,4}( \WH{k} )=\WH{\omega}^{\text{ex}}_{3,4}(F_{2M}(\WH{k})).$ 

In Figure \ref{Fig: AA2}, we present the relative dispersion errors  with $\WH{k}\in[0,2\pi]$ and the parameter values
$$\epsilon_s = 5.25, \quad \epsilon_\infty = 2.25, \quad \omega_1 h = \frac{\pi}{30}.$$ In this figure, we can observe the decrease of error when $M$ (order of the scheme) increases.   The numerical error in the first and second solutions of the discrete angular frequency, are smaller than that of the third and fourth solution, which can be understood if
we consider the small wavenumber limit. In this case, we can derive expressions for the relative phase error as 
\begin{align}
\renewcommand{\arraystretch}{2.5}
\WH{\Psi}_{\text{FD},2M}(\WH{k})
:=
 \left| \frac{\WH{\omega}^{\text{ex}}_{i}(\WH{k}) - \WH{\omega}^{\text{FD},2M}_{i}(\WH{k})}{\WH{\omega}^{\text{ex}}_{i}(\WH{k})} \right|
= \left\{ \begin{array}{ll}
\displaystyle \frac{[(2M-1)!!]^2}{2^{2M}(2M+1)!}\, \WH{k}^{2M} 
+ \mathcal{O}(\WH{k}^{2M+2}), & i = 1, 2, \\
\displaystyle \frac{[(2M-1)!!]^2}{2^{2M}(2M+1)!}
\,\frac{\epsilon_d}{\epsilon_{s}^2 } \,\frac{k^2}{\omega_1^2}
\, \WH{k}^{2M} 
+ \mathcal{O}(\WH{k}^{2M+2}), & i = 3, 4, \\
\end{array}
\right.
\renewcommand{\arraystretch}{1}
\label{A5}
\end{align}
which indicates a dispersion error of order $2M$ and is consistent to  our previous conclusion (see Theorem \ref{thm5}). By comparing the coefficients, we verify that, for the parameters we consider, the leading error coefficient corresponding to $\WH{\omega}^{\text{FD},2M}_{3,4}( \WH{k} )$ is indeed much larger than that for $\WH{\omega}^{\text{FD},2M}_{1,2}( \WH{k} ).$

\begin{figure}[h]
	\centering
	\includegraphics[scale=0.3]{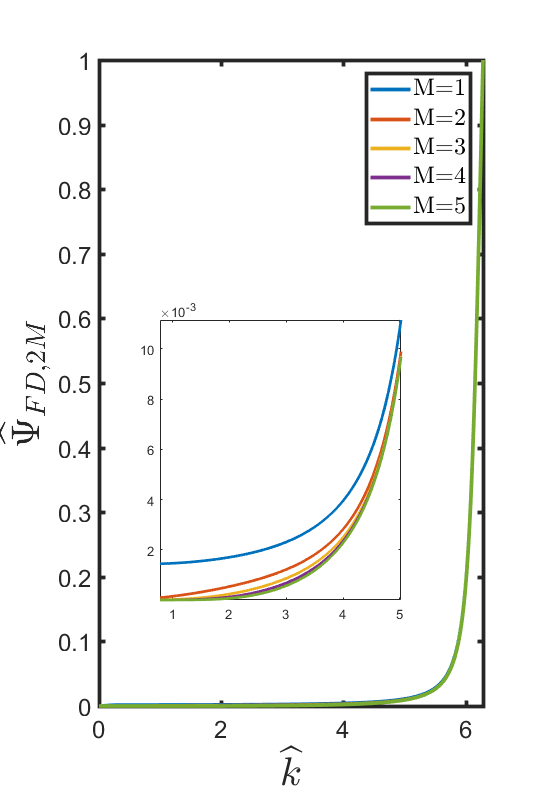}
	\includegraphics[scale=0.3]{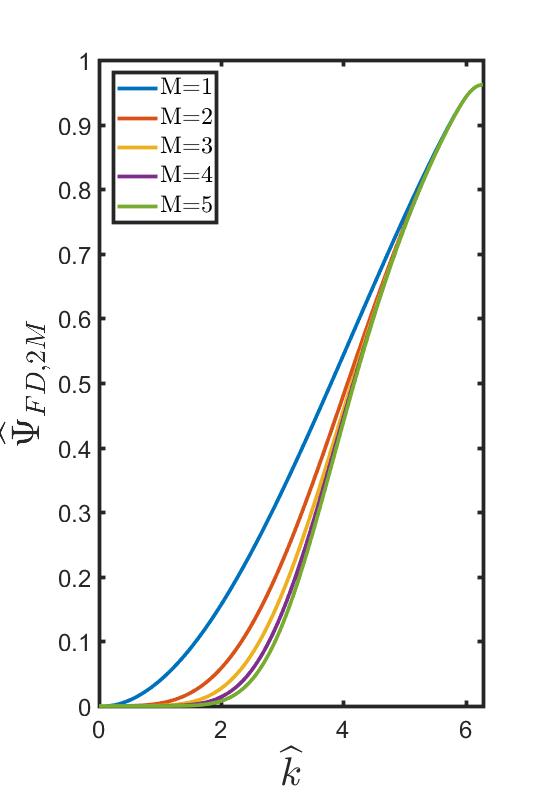} 
	\caption{Relative phase error \eqref{A5} for the spatial discretization FD2M with $\WH{\gamma}=0$. $\WH{k}\in[0,2\pi]$. Left: $i=1,2$; The inset in the left plot displays a zoomed-in region of the relative phase error for low values of $\WH{k}$; Right: $i=3,4$. }
	\label{Fig: AA2}
\end{figure}

For low-loss material, e.g. $\WH{\gamma}= 0.01,$ the conclusions are very similar. The error plots show no visible difference from the no loss case, and are thus omitted.

\bibliographystyle{siam}
\bibliography{Elec_disp}

\end{document}